\documentclass[11pt,reqno]{amsart} 
\pdfoutput=1 

\usepackage[mathlines]{lineno}
\usepackage{microtype}

\usepackage[top=1.25in,bottom=1.25in,left=1.25in,right=1.25in]{geometry} 

\usepackage{xcolor}

\usepackage{scalerel}

\usepackage{enumitem}

\usepackage{tocvsec2}

\usepackage[labelfont=bf]{caption}

\usepackage[font={small},labelfont=md]{subfig}

\usepackage[noadjust,nosort]{cite}

\usepackage{amsmath, amsthm, amssymb, tikz, mathtools, array, mathrsfs, tensor, ifthen, xparse, graphicx}

\usepackage{polynom}

\usepackage{esint}

\usepackage{multicol}

\usepackage{dsfont}
	\newcommand{\one}{\mathds{1}}

\usepackage{framed}

\usepackage{stackengine}

\usepackage{upref}

\usepackage{hyperref}
	\hypersetup{colorlinks=true,linkcolor=blue,citecolor=blue,urlcolor=black} 

\numberwithin{equation}{section}


\newcommand{\eq}[1]{\begin{linenomath}\postdisplaypenalty=0\begin{align*} #1 \end{align*}\end{linenomath}}

\newcommand{\eeq}[1]{\begin{linenomath}\postdisplaypenalty=0\begin{align} \begin{split} #1 \end{split} \end{align}\end{linenomath}}

\newcommand{\stackref}[2]{
\readlist*\mylist{#1}
\stackrel{\mbox{\footnotesize\foreachitem\x\in\mylist[]{\ifnum\xcnt=1\else,\fi\eqref{\x}}}}{#2}
}
\newcommand{\stackrefp}[2]{
\readlist*\mylist{#1}
\stackrel{\hphantom{\mbox{\footnotesize\foreachitem\x\in\mylist[]{\ifnum\xcnt=1\else,\fi\eqref{\x}}}}}{#2}
}
\newcommand{\stackrefpp}[3]{
\readlist*\mylist{#1}
\readlist*\mylistt{#2}
\stackrel{\parbox{\widthof{\footnotesize\foreachitem\x\in\mylistt[]{\ifnum\xcnt=1\else,\fi\eqref{\x}}}}{\centering\footnotesize\foreachitem\x\in\mylist[]{{\ifnum\xcnt=1\else,\fi\eqref{\x}}}}}{#3}
}

\def\eps{\varepsilon}
\def\vphi{\varphi}

\newcommand{\E}{\mathbb{E}}

\newcommand{\N}{\mathbb{N}}
\renewcommand{\P}{\mathbb{P}}
\newcommand{\Q}{\mathbb{Q}}
\newcommand{\R}{\mathbb{R}}

\newcommand{\Z}{\mathbb{Z}}

\renewcommand{\AA}{\mathcal{A}}

\newcommand{\CC}{\mathcal{C}}
\newcommand{\DD}{\mathcal{D}}

\newcommand{\FF}{\mathcal{F}}
\newcommand{\GG}{\mathcal{G}}
\newcommand{\HH}{\mathcal{H}}
\newcommand{\II}{\mathcal{I}}

\newcommand{\LL}{\mathcal{L}}

\newcommand{\NN}{\mathcal{N}}
\newcommand{\OO}{\mathcal{O}}
\newcommand{\PP}{\mathcal{P}}
\newcommand{\QQ}{\mathcal{Q}}
\newcommand{\RR}{\mathcal{R}}

\newcommand{\XX}{\mathcal{X}}

\newcommand{\PPP}{\mathscr{P}}

\newcommand{\RRR}{\mathscr{R}}

    \newcommand{\eff}{\mathsf{F}}

\newcommand{\iprod}[3][]{#1\langle #2, #3#1\rangle}

\newcommand{\vv}[1]{\mathbf{#1}}

\newcommand{\bbeta}{\boldsymbol{\beta}}

\newcommand{\wt}[1]{\widetilde{#1}}

\newcommand{\dd}{\mathrm{d}} 

\newcommand{\bone}{\mathbf{1}}
\newcommand{\bI}{\boldsymbol{I}}
\newcommand{\sT}{\mathsf{T}}

\newcommand{\Bgiven}{\,\Big|\,}
\newcommand{\ka}{\kappa}
\newcommand{\de}{\mathrm{d}}
\newcommand\op{\mathrm{op}}
\newcommand\disc{{\sf disc}}
\newcommand\la{\lambda}
\newcommand{\sF}{\eff}

            \makeatletter
            \DeclareFontFamily{OMX}{MnSymbolE}{}
            \DeclareSymbolFont{MnLargeSymbols}{OMX}{MnSymbolE}{m}{n}
            \SetSymbolFont{MnLargeSymbols}{bold}{OMX}{MnSymbolE}{b}{n}
            \DeclareFontShape{OMX}{MnSymbolE}{m}{n}{
                <-6>  MnSymbolE5
               <6-7>  MnSymbolE6
               <7-8>  MnSymbolE7
               <8-9>  MnSymbolE8
               <9-10> MnSymbolE9
              <10-12> MnSymbolE10
              <12->   MnSymbolE12
            }{}
            \DeclareFontShape{OMX}{MnSymbolE}{b}{n}{
                <-6>  MnSymbolE-Bold5
               <6-7>  MnSymbolE-Bold6
               <7-8>  MnSymbolE-Bold7
               <8-9>  MnSymbolE-Bold8
               <9-10> MnSymbolE-Bold9
              <10-12> MnSymbolE-Bold10
              <12->   MnSymbolE-Bold12
            }{}
            
            \let\llangle\@undefined
            \let\rrangle\@undefined
            \DeclareMathDelimiter{\llangle}{\mathopen}%
                                 {MnLargeSymbols}{'164}{MnLargeSymbols}{'164}
            \DeclareMathDelimiter{\rrangle}{\mathclose}%
                                 {MnLargeSymbols}{'171}{MnLargeSymbols}{'171}
            \makeatother
            
    \DeclareFontFamily{U}{matha}{\hyphenchar\font45}
    \DeclareFontShape{U}{matha}{m}{n}{ <-6> matha5 <6-7> matha6 <7-8>
    matha7 <8-9> matha8 <9-10> matha9 <10-12> matha10 <12-> matha12 }{}
    \DeclareSymbolFont{matha}{U}{matha}{m}{n}
    \DeclareFontFamily{U}{mathx}{\hyphenchar\font45}
    \DeclareFontShape{U}{mathx}{m}{n}{ <-6> mathx5 <6-7> mathx6 <7-8>
    mathx7 <8-9> mathx8 <9-10> mathx9 <10-12> mathx10 <12-> mathx12 }{}
    \DeclareSymbolFont{mathx}{U}{mathx}{m}{n}
    
    \DeclareMathDelimiter{\llbrack} {4}{matha}{"76}{mathx}{"30}
    \DeclareMathDelimiter{\rrbrack} {5}{matha}{"77}{mathx}{"38}

\DeclareMathOperator{\dgr}{deg}

\newcommand{\diag}{{\mathrm{diag}}}
\newcommand{\bal}{\textup{\textsf{bal}}}

\newcommand{\Law}{{\mathsf{Law}}}
\newcommand{\pmu}{\omega}
\newcommand{\tr}{\mathrm{tr}}
\newcommand{\err}{\mathrm{err}}
\newcommand{\zee}{\mathsf{Z}}
\newcommand{\nonpsd}{\RRR_\kappa}

\newcommand{\Prob}{\mathsf{Prob}}

\newcommand\bbullet{\raisebox{0.13ex}{{\scaleobj{0.8}{\bullet}}}} 
\newcommand{\act}{\mkern 3mu\bbullet\mkern 3mu}

\newtheorem{theorem}{Theorem}[section]
\newtheorem{proposition}[theorem]{Proposition}
\newtheorem{corollary}[theorem]{Corollary}
\newtheorem{lemma}[theorem]{Lemma}
\newtheorem{claim}[theorem]{Claim}

\newtheorem{assumption}[theorem]{Assumption}

\newtheorem{theirthm}[theorem]{Theorem} 

\theoremstyle{definition} 
\newtheorem{definition}[theorem]{Definition}

\newtheorem{remark}[theorem]{Remark}

\newenvironment{proofclaim}[1][Proof]
	{\begin{proof}[#1]}
	{\end{proof}}

\title[Balanced Potts spin glass]{Parisi formula for balanced Potts spin glass}



\author{Erik Bates}
\thanks{E.B. was partially supported by NSF grant DMS-2246616.} 
\address{\newline Department of Mathematics \newline North Carolina State University \newline  SAS Hall, 2311 Stinson Drive \newline   Raleigh, North Carolina 27695-8205 USA \newline \textup{\tt ebates@ncsu.edu}}

\author{Youngtak Sohn}
\thanks{Y.S. is supported by Simons-NSF Collaboration on Deep Learning NSF DMS-2031883 and Vannevar Bush Faculty Fellowship award ONR-N00014-20-1-2826.}
\address{\newline Department of Mathematics \newline Massachusetts Institute of Technology \newline 77 Massachusetts Avenue \newline Cambridge, Massachusetts 02139-4307 USA 
\newline \textup{\tt youngtak@mit.edu}}

\begin{document}


\begin{abstract}
The Potts spin glass is a generalization of the Sherrington--Kirkpatrick (SK) model that allows for spins to take more than two values. 
Based on a novel synchronization mechanism, Panchenko (2018) showed that the limiting free energy is given by a Parisi-type variational formula. 
The functional order parameter in this formula is a probability measure on a monotone path in the space of positive-semidefinite matrices.
By comparison, the order parameter for the SK model is much simpler: a probability measure on the unit interval. Nevertheless, a longstanding prediction by Elderfield and Sherrington (1983) is that the order parameter for the Potts spin glass can be reduced to that of the SK model.

We prove this prediction for the balanced Potts spin glass, where the model is constrained so that the fraction of spins taking each value is asymptotically the same. It is generally believed that the limiting free energy of the balanced model is the same as that of the unconstrained model, in which case our results reduce the functional order parameter of Panchenko's variational formula to probability measures on the unit interval. The intuitive reason---for both this belief and the Elderfield--Sherrington prediction---is that no spin value is a priori preferred over another, and the order parameter should reflect this inherent symmetry.

This paper rigorously demonstrates how symmetry, when combined with synchronization, acts as the desired reduction mechanism. Our proof requires that we introduce a generalized Potts spin glass model with mixed higher-order interactions, which is interesting it its own right.
 We prove that the Parisi formula for this model is differentiable with respect to inverse temperatures.
 This is a key ingredient for guaranteeing the Ghirlanda--Guerra identities without perturbation, which then allow us to exploit symmetry and synchronization simultaneously.
%
%
\end{abstract}

\maketitle


\section{Introduction}

\subsection{The model} \label{intro_model}
The Potts spin glass was introduced by Elderfield and Sherrington~\cite{elderfield-sherrington83a} and has been extensively studied in statistical mechanics \cite{nishimori-stephen83, elderfield-sherrington83b, gross-kanter-sompolinsky85, binder_young86, santis-parisi-ritort95, caltagirone-parisi-rizzo12}.
In dimension $\kappa\geq2$, the model is defined as follows. 
At volume $N$, the configuration space is the product set $\Sigma^N$, where $\Sigma = \{\vv e_1,\dots,\vv e_\kappa\}$ is the standard basis in $\R^\kappa$.
Each configuration $\sigma = (\sigma_1,\ldots, \sigma_N) \in \Sigma^N$ is thought of as a $\kappa\times N$ matrix, and its energy is given by the Hamiltonian
\eeq{ \label{potts_hamiltonian}
H_N(\sigma)
\coloneqq \frac{\beta}{\sqrt{N}}\sum_{i,\,j=1}^Ng_{i,j}\one\{\sigma_i=\sigma_j\}
= \frac{\beta}{\sqrt{N}}\sum_{i,\,j=1}^N g_{i,j}\iprod{\sigma_i}{\sigma_j},
}
where $(g_{i,j})_{i,j=1}^N$ are independent standard normal random variables, $\iprod{\lambda}{\sigma}=\lambda^{\sT}\sigma$ is the inner product of $\lambda,\sigma\in\R^\kappa$, and $\beta>0$ is an inverse temperature parameter. 
The associated free energy is 
\begin{equation} \label{meif3}
\eff_N\coloneqq\frac{1}{N}\E \log \zee_N,\quad\textnormal{where}\quad \zee_N\coloneqq\sum_{\sigma \in \Sigma^{N}} \exp H_N(\sigma).
\end{equation}
For $\kappa=2$, the Hamiltonian \eqref{potts_hamiltonian} is equivalent to the classical Sherrington--Kirkpatrick (SK) model~\cite{sherrington-kirkpatrick75} by the mapping $\sigma \mapsto \tau=(\tau_1,\ldots \tau_N)$, where
\begin{equation}
\label{SK_transform}
    \tau_i = 
    \begin{cases}
        +1 &\quad\textnormal{if $\sigma_i= \vv e_1$}\\
        -1 &\quad\textnormal{if $\sigma_i= \vv e_2$}.
    \end{cases}
\end{equation}
By this transformation, the free energy $\eff_N$ equals the free energy of the SK model up to rescaling $\beta$. In this case, the limiting free energy $\lim_{N\to\infty}\eff_N$ is given by the celebrated Parisi formula, which is a variational expression predicted by Parisi~\cite{parisi79,parisi80a,parisi80b,parisi83} but not proved until the seminal work of Guerra~\cite{guerra03} and Talagrand~\cite{talagrand06a}. 

The fundamental insight behind the Parisi formula is that the SK free energy can be understood by keeping track of a single random variable, namely the \textit{replica overlap}.
In our context, this is the quantity $N^{-1}\sum_{i=1}^N\iprod{\sigma_i^1}{\sigma_i^2}$, where $\sigma^1$ and $\sigma^2$ are independent samples from the Gibbs measure $G_N(\sigma) = \exp H_N(\sigma)/Z_N$. 
This quantity is exactly the fraction of coordinates at which $\sigma^1$ and $\sigma^2$ agree.
The law of this random variable is a probability measure on $[0,1]$, and Parisi's formula is a minimization problem over such measures.



For $\kappa\geq3$, this perspective runs into difficulty because of the additional degrees of freedom.
Namely, the transformation \eqref{SK_transform} does not have a natural generalization, and so there is no obvious way to relate the free energy of the Potts spin glass to a scalar statistic.
In principle, one needs to keep track of the entire $\kappa\times\kappa$ matrix $N^{-1}\sigma^1(\sigma^2)^\sT$, whereas the replica overlap is just the trace.
Nevertheless, using an ingenious synchronization mechanism, 
Panchenko~\cite{panchenko18a} showed that in the large-$N$ limit, this overlap matrix is 
some deterministic map of its trace (Theorem~\ref{sync_thm}).
This led to a generalized Parisi formula that optimizes over probability measures on $[0,1]$ \textit{together with} so-called synchronization maps (Theorem~\ref{thm:Panchenko}).

The program initiated by this paper is to go even further: there is only one possible choice for the synchronization map.
Theorem~\ref{thm:main:Potts} accomplishes this for the balanced Potts spin glass and is presented in the next section.

\subsection{Limiting free energy}
A crucial fact of synchronization is that it requires the self-overlap matrix $N^{-1}\sigma\sigma^\sT\in\R^{\kappa\times\kappa}$ to take a fixed value (see \eqref{deterministic_for_same_ell} in Theorem~\ref{sync_thm}).
But of course the function $\sigma\mapsto N^{-1}\sigma\sigma^\sT$ is not constant over the configuration space $\Sigma^N$; it can be any diagonal matrix whose entries belong to $\{0,\frac{1}{N},\dots,\frac{N-1}{N},1\}$ and sum to $1$.
Therefore, the strategy of \cite{panchenko18b} is to consider subsets of $\Sigma^N$ on which this function is approximately constant, and then derive a Parisi formula for the model constrained to these subsets.
The number of subsets needed grows only polynomially in $N$, whereas \eqref{meif3} concerns an exponential growth rate.
Therefore, classical Gaussian concentration allows one to determine that the limiting free energy of the unconstrained model is simply the largest limiting free energy among the constrained models.
We now proceed to make things precise.

Denote the $\kappa$-dimensional unit simplex by
\eq{
\DD \coloneqq \Big\{d = (d_1,\dots,d_\kappa)\in[0,1]^\kappa :\, \sum_{k=1}^\kappa d_k= 1\Big\}.
}
An element $d\in \DD$ is called a \textit{magnetization} of the Potts spin glass. 
The configuration space with magnetization $d$ and approximation parameter $\eps\geq 0$ is
\eeq{\label{eq:def:Sigma:N:d:eps}
\Sigma^N(d, \eps) \coloneqq \Big\{\sigma\in\Sigma^N:\, \Big|\frac{1}{N}\sum_{i=1}^N\one\{\sigma_i=\vv e_k\} -d_k\Big|\leq \eps \text{ for each $k\in\{1,\dots,\kappa\}$}\Big\}.
}
The associated \textit{constrained} free energy is
\eeq{\label{def:free:energy:potts}
\eff_N(d,\eps) \coloneqq \frac{1}{N}\E\log \zee_N(d,\eps), \qquad \text{where} \qquad \zee_N(d,\eps) \coloneqq \sum_{\sigma\in\Sigma^N(d,\eps)}\exp H_N(\sigma).
}
When $\eps=0$, we write 
\begin{equation*}
    \Sigma^N(d)\coloneqq\Sigma^N(d,0),\quad \zee_N(d)\coloneqq \zee_N(d,0),\quad\eff_N(d)\coloneqq \eff_N(d,0).
\end{equation*}
This last free energy only makes sense if $\Sigma^N(d)$ is nonempty, which occurs precisely when $d$ belongs to the set
\eeq{ \label{DN_def}
\DD_N \coloneqq \DD\cap(\Z/N)^{\kappa} = \{d\in\DD:\, \Sigma^N(d)\neq\varnothing\}.
}

Next we introduce the order parameter for Panchenko's variational formula.
Define 
\eq{ 
\Gamma_\kappa \coloneqq \{\gamma\in\R^{\kappa\times\kappa}:\, \text{$\gamma$ is symmetric and positive-semidefinite}\}.
}
Given $d\in\DD$, let $\Gamma_\kappa(d)$ be the subset of $\Gamma_\kappa$ consisting of matrices with nonnegative entires whose row sums are given by $d$:
\eeq{ \label{Gamma_d_def}
\Gamma_\kappa(d) \coloneqq \Big\{\gamma\in\Gamma_\kappa\cap[0,1]^{\kappa\times\kappa} :\,\sum_{k^\prime=1}^\kappa \gamma_{k,k^\prime} = d_{k} \text{ for each $k \in\{1,\dots,\kappa\}$}\Big\}, \quad d\in\DD.
}
We write $\preceq$ for the positive-semidefinite order on symmetric matrices.
Then consider the following collection of paths on $\Gamma_\kappa(d)$:
\eeq{ \label{def:Pi_d}
\Pi_d \coloneqq \big\{\pi\colon(0,1]\to\Gamma_\kappa(d):\,\text{$\pi$ is left-continuous, $\pi(s)\preceq\pi(t)$ if $s\leq t$}\big\}. 
}
To place this definition into context with the previous section: each $\pi\in\Pi_d$ is the combination of a probability measure $\mu$ on $[0,1]$ with a synchronization map $\Phi\colon[0,1]\to\Gamma_\kappa(d)$.
More specifically, $\pi$ is the composition $\Phi\circ Q_\mu$, where $Q_\mu$ is the quantile function of $\mu$.
See Remark~\ref{rmk:fop} for an important special case, and Section~\ref{sec_overview_arrays} for the origin of the map $\Phi$.

The following result summarizes the outcome of Panchenko's synchronization scheme for the Potts spin glass. 
The statement below combines the results of two papers.
The Parisi functional $\PP$ is defined in Section~\ref{subsec:Parisi}, specifically \eqref{def:parisi:ftl} with $\xi$ given in \eqref{Potts_xi}.
\begin{theirthm}\textup{\cite{panchenko18a, panchenko18b}}
\label{thm:Panchenko}
There is an explicit functional
$\PP\colon\cup_{d\in\DD}\Pi_d\to\R$
such that for every $d\in\DD$, the constrained free energy has the following limit:
\begin{equation} \label{eq:Panchenko:cformula}
\lim_{\eps\searrow 0}\limsup_{N\to\infty}\eff_N(d,\eps)=\lim_{\eps\searrow 0}\liminf_{N\to\infty}\eff_N(d,\eps)=\inf_{\pi \in \Pi_d}\PP(\pi).
\end{equation}
Furthermore, the limiting unconstrained free energy is given by
\begin{equation}\label{eq:Panchenko:formula}
\lim_{N\to\infty} \eff_N =\sup_{d\in \DD}\inf_{\pi \in \Pi_d}\PP(\pi).
\end{equation}
\end{theirthm}
\begin{remark} \label{rmk_compare_defs}
 In \cite{panchenko18a}, the definition of $\Pi_d$ included additional stipulations that $\pi(0)=0$ and $\pi(1)=\diag(d)$.
 The first condition is unimportant because the value of $\pi(0)$ plays no role in determining the value of $\PP(\pi)$.
     Regarding the second condition $\pi(1)=\diag(d)$, our definition \eqref{def:Pi_d}  only implies $\pi(1)\preceq\diag(d)$, since $\diag(d)$ is the maximal element of $\Gamma_\kappa(d)$ (Lemma~\ref{lem_dominance}).
     The infimum in \eqref{eq:Panchenko:cformula} is not sensitive to this difference because every $\pi\in\Pi_d$ can be approximated 
     by $\tilde\pi\in\Pi_d$ such that $\tilde\pi(1)=\diag(d)$, and the Parisi functional is continuous 
     (Proposition~\hyperref[prop:continuity_b]{\ref*{prop:continuity}\ref*{prop:continuity_b}}). 
     On the other hand, our definition guarantees the existence of a minimizer. 
\end{remark}

It was further predicted by Elderfield and Sherrington~\cite{elderfield-sherrington83a} when they introduced the Potts spin glass model that the variational expression $\sup_{d\in \DD}\inf_{\pi\in \Pi_d}\PP(\pi)$ in \eqref{eq:Panchenko:formula} is achieved on a restricted set. In our notation, \cite[Eq.~(15)]{elderfield-sherrington83a} predicts the following:
\begin{enumerate}[label=\textup{(\alph*)}]
    \item \label{es_pred_a}
    The supremum over $d\in \DD$ is achieved at the balanced case $d=d_{\bal}$, where 
   \begin{equation*}
       d_{\bal}\coloneqq \kappa^{-1}\bone,
   \end{equation*}
   and $\bone\in \R^{\kappa}$ is the vector of all ones.
   This prediction was echoed in \cite[Rmk.~3]{panchenko18a}.
    \item \label{es_pred_b}
    The infimum over $\pi\in \Pi_{d_{\bal}}$ is achieved by some $\pi$ such that for every $t\in[0,1]$, the matrix $\pi(t)$ is constant on its diagonal and also constant off its diagonal.
\end{enumerate}
In fact, these predictions were central to analysis of \cite{elderfield-sherrington83a} regarding the phase transitions of the Potts spin glass. 
See \cite[Sec.~VI.H.1]{binder_young86} for further discussion.

Our main result confirms prediction~\ref{es_pred_b}. 
To this end, let $\bI_{\kappa}\in \R^{\kappa\times \kappa}$ denote the identity matrix, and consider the following subset of $\Gamma_\kappa(d_\bal)$:
\eeq{ \label{def:Gamma:star}
\Gamma^\star\coloneqq\Big\{\gamma\in \Gamma_{\kappa}:\, \gamma=\frac{q}{\kappa}\bI_{\kappa}+\frac{1-q}{\kappa^2}\bone\bone^{\sT}~\textnormal{for some $q\in [0,1]$}\Big\}.
}
Then define the corresponding subset of the path space $\Pi_{d_\bal}$:
\begin{equation}\label{def:Pi:star}
\Pi^{\star}\coloneqq\big\{\pi\colon(0,1]\to\Gamma^\star:\,\textnormal{$\pi$ is left-continuous, $\pi(s)\preceq\pi(t)$ if $s\leq t$}\big\}.
\end{equation}

\begin{theorem}\label{thm:main:Potts}
    For $d_{\bal}=\kappa^{-1}\bone$, we have
    \begin{equation*}
   \lim_{\eps\searrow 0}\limsup_{N\to\infty}\eff_N(d_{\textup{\textsf{bal}}},\eps)=\lim_{\eps\searrow 0}\liminf_{N\to\infty}\eff_N(d_{\bal},\eps)=\inf_{\pi \in \Pi^{\star}}\PP(\pi).
    \end{equation*}
    Moreover, for any $d^N\in\DD_N$ converging to $d_\bal$ as $N\to\infty$, we have
    \eq{ 
\lim_{N\to\infty}\eff_N(d^N) = \inf_{\pi\in\Pi^\star}\PP(\pi).    
}
\end{theorem}
We note that Theorem~\ref{thm:main:Potts} is only interesting when $\ka \geq 3$ since for $\ka=2$, we have $\Pi^\star=\Pi_{d_{\bal}}$. For $\ka\geq3 $, however, $\Gamma^\star$ remains a $1$-dimensional space, whereas $\Gamma_{\ka}(d_{\bal})$ has dimension $\frac{\ka(\ka-1)}{2}$. 
Theorem~\ref{thm:main:Potts} thus shows that when $d=d_\bal$, the infimum in \eqref{eq:Panchenko:formula}  can be taken over a much smaller space. 
In this way Theorem~\ref{thm:main:Potts} can viewed as a deterministic statement on top of Theorem~\ref{thm:Panchenko}, although our proof is probabilistic.
In fact, we will deduce Theorem~\ref{thm:main:Potts} from a more general result, Theorem~\ref{thm:main}, presented in Section~\ref{sec:gen_Par}.

 One may naively hope to prove Theorem~\ref{thm:main:Potts} by showing that $\Pi_{d_\bal}\ni \pi \mapsto \PP(\pi)$ is convex, since then the infimum in \eqref{eq:Panchenko:formula} must be achieved at a ``symmetric" $\pi$, by which me mean $\pi\in\Pi^\star$.
 But unfortunately this convexity fails when $\kappa\geq3$, because of complications in comparing synchronization maps with different images.
 In fact, a recurring challenge in settings relying on synchronization---such as multi-species models or vector spin models---is to extract information about the minimizer of the Parisi functional without relying on convexity, e.g.~\cite{bates-sloman-sohn19,dey-wu21,bates-sohn22b}.
 Theorem~\ref{thm:main:Potts} is a new example of how to accomplish this, and we anticipate it can actually enable convexity arguments as in the SK model \cite{panchenko05a,bovier_klimovsky09,auffinger-chen15b,jagannath-tobasco16}.
 This is because the synchronization map is now fixed, as explained below in Remark~\ref{rmk:fop}.
 Indeed, following the initial release of this manuscript, Chen~\cite{chen23d_arxiv} proved a version of Theorem~\ref{thm:main:Potts} for the model with self-overlap correction, and the resulting Parisi functional is strictly convex \cite{chen23c_arxiv}.


\begin{remark}\label{rmk:fop}
   There is a one-to-one correspondence between $\Pi^\star$ and $\Prob([0,1])$, the set of Borel probability measures on $[0,1]$.
   Given $\mu\in\Prob([0,1])$, denote its quantile function by $Q_\mu(t) = \inf\{q\geq0:\mu([0,q])\geq t\}$.
   Then the correspondence is achieved through the relation $\pi = \Phi^\star\circ Q_\mu$, where $\Phi^\star(q) = q\kappa^{-1}\bI_\kappa + (1-q)\kappa^{-2}\bone\bone^\sT$.
   In this way, the functional order parameter in the balanced Potts spin glass is understood as an element of $\Prob([0,1])$.
   Moreover, if prediction~\ref{es_pred_a} is also true, then Theorems~\ref{thm:Panchenko} and~\ref{thm:main:Potts} together express the unconstrained free energy $\lim_{N\to\infty} \eff_N$ as an infimum over probability measures on $[0,1]$, just as for the classical SK model.
   
\end{remark}

Regardless of the status of prediction~\ref{es_pred_a}, 
the balanced Potts spin glass has an important application in combinatorial optimization. Namely, denote by $\textsf{MaxCut}(G,\kappa)$ the size of the maximum cut of graph $G$ into $\kappa$ parts, and consider either the sparse Erd\H{o}s--Rényi graph $G_N\sim G(N,d/N)$ or the sparse random regular graph $G_N\sim G_{\mathrm{reg}}(N,d)$ with degree $d$. 
It was shown by Sen~\cite[Sec.~2.2]{sen18} that for $d$ sufficiently large, the following asymptotic holds with high probability as $N\to\infty$:
\begin{equation*}
\frac{\textsf{MaxCut}(G_N,\kappa)}{N}=\frac{d}{2}\left(1-\frac{1}{\kappa}\right)+P_{\star}(\kappa)\frac{\sqrt{d}}{2}+o(\sqrt{d}),
\end{equation*}
where $P_{\star}(\kappa)$ denotes the ground-state energy of the balanced Potts spin glass model:
\begin{equation*}
P_{\star}(\kappa)\coloneqq\lim_{\beta\to\infty}\frac{1}{\beta}\inf_{\pi\in \Pi_{d_{\bal}}}\PP(\pi; \beta).
\end{equation*}
Here we are making explicit the dependence of the Parisi functional on $\beta$. 
Now Theorem~\ref{thm:main:Potts} simplifies this to a lower-dimensional minimization problem:
\eeq{ \label{bqa9}
P_{\star}(\kappa)=\lim_{\beta\to\infty}\frac{1}{\beta}\inf_{\pi\in \Pi^{\star}}\PP(\pi; \beta).
}
For $\kappa=2$ (in which case $\Pi_{d_\bal}=\Pi^\star$), numerical studies \cite{crisanti_rizzo02, schmidt08} give $P_{\star}(2)\approx 0.763166$. 
Since $\Pi^\star$ has no dimensional dependence on $\kappa$ (unlike $\Pi_{d_\bal}$), \eqref{bqa9} opens the possibility of numerically computing $P_{\star}(\kappa)$ for $\kappa \geq 3$.

\subsection{Potts spin glass with general covariance function} \label{sec:gen_Par}
We will obtain Theorem~\ref{thm:main:Potts} from the analogous result for a more general Hamiltonian.
But this is not simply for the sake of generality: our proof of Theorem~\ref{thm:main:Potts} actually requires that we first replace \eqref{potts_hamiltonian} with a Hamiltonian that includes higher-order interactions.
Namely, we assume henceforth that $(H_{N,\xi}(\sigma))_{\sigma\in\Sigma^N}$ is a centered Gaussian process with covariance
\eeq{ \label{general_cov}
\E[H_{N,\xi}(\sigma)H_{N,\xi}(\sigma')] = N\xi\Big(\frac{\sigma\sigma'^\sT}{N}\Big), \quad \sigma,\sigma'\in\Sigma^N,
}
where $\xi\colon\R^{\kappa\times\kappa}\to\R$ is any function satisfying assumptions~\ref{xi_power} and~\ref{xi_convex} below.
The ordinary Potts spin glass \eqref{potts_hamiltonian} is the special case
\eeq{ \label{Potts_xi}
\xi(R) = \beta^2\sum_{k,k'=1}^\kappa R_{k,k'}^2 = \beta^2\tr(R^\sT R).
}
To state our more general assumptions, we require some definitions.

Let $\N = \{1,2,...\}$.
Then define the following set of parameters:
\eeq{\label{eq:def:Theta}
\Theta \coloneqq \Big\{(p,m,n_1,\dots,n_m,w_1,\dots,w_m)\,:\, p,m,n_1,\dots,n_m\in\N, w_1,\dots,w_m\in[-1,1]^\kappa\Big\}.
}
For $\theta\in\Theta$, define the function $\xi_\theta\colon\R^{\kappa\times\kappa}\to\R$ by
\eeq{ \label{def:xi:theta}
\xi_\theta(R) \coloneqq \prod_{j=1}^m \iprod{R^{\circ p}w_j}{w_j}^{n_j},
}
where $R^{\circ p}$ denotes the $p^\mathrm{th}$ Hadamard power of the matrix $R$, i.e.~$(R^{\circ p})_{k,k'} = (R_{k,k'})^p$.
The function $\xi_\theta$ is a homogeneous polynomial of degree
\eeq{ \label{deg_def}
    \dgr(\theta) \coloneqq p(n_1+\cdots+n_m).
}
Then we assume the following about the covariance function $\xi\colon\R^{\kappa\times\kappa}\to\R$ in \eqref{general_cov}:
\begin{enumerate}[label=\textup{(A\arabic*)}]

      \item\label{xi_power} There exists a countable subset $\Theta_0 \subset \Theta$ and coefficients $(\alpha_{\theta})_{\theta\in\Theta_0}$ such that
      \begin{equation*}
      \xi(R)=\sum_{\theta\in\Theta_0}\alpha_{\theta}^2\xi_{\theta}(R) \quad \text{and} \quad
      \sum_{\theta\in\Theta_0}\alpha_\theta^2(1+\eps)^{\deg(\theta)} < \infty \quad \text{for some $\eps>0$}.
      \end{equation*}
   
    \item \label{xi_convex}
    $\xi$ is convex when restricted to the following set: 
    \eeq{ \label{def_nonpsd}
    \nonpsd \coloneqq \{R\in[0,1]^{\kappa\times\kappa}:\|R\|_1\leq 1\}.
    }
\end{enumerate}

\begin{remark} \label{rmk_xi}
We always have 
\eeq{ \label{rmk_xi_1}
\|\xi_\theta(R)\| \leq \|R\|_1^{\deg(\theta)}.
}
Therefore, the decay condition on $\alpha_\theta$ in~\ref{xi_power} ensures that $\sum_{\theta\in\Theta_0}\alpha_\theta^2\xi_\theta(R)$ converges uniformly on the set $\{R\in\R^{\kappa\times\kappa}:\, \|R\|_1\leq 1\}$.
In particular, $\xi$ is smooth on this set.
Outside this set, the sum may diverge, and this is okay because we will only ever need to apply $\xi$ on the set $\nonpsd$ defined in~\ref{xi_convex}.
This is because for any $\sigma,\sigma'\in\Sigma^N$, we have $N^{-1}\sigma\sigma'^\sT\in\nonpsd$.
Nevertheless, we will continue to write $\R^{\kappa\times\kappa}$ as the domain of $\xi$ simply to remind the reader of the ambient matrix space.
\end{remark}

\begin{remark} \label{rmk_Potts_case}
Comparing \eqref{Potts_xi} and \eqref{def:xi:theta}, one sees that the ordinary Potts spin glass \eqref{potts_hamiltonian} corresponds to the case when $\Theta_0$ has a single element: $(p=1,m=1,n=2,w=\bone)$.
\end{remark}


Extending the notation of \eqref{def:free:energy:potts}, we write
\begin{equation} \label{explicit_xi}
    \eff_{N,\xi}(d,\eps) = \frac{1}{N}\E\log \zee_{N,\xi}(d,\eps), \ \ \ \text{where} \ \ \ \zee_{N,\xi}(d,\eps) = \sum_{\sigma\in\Sigma^N(d,\eps)}\exp H_{N,\xi}(\sigma),
\end{equation}
and $\eff_{N,\xi}(d)= \eff_{N,\xi}(d,0)$ and $\zee_{N,\xi}(d)= \zee_{N,\xi}(d,0)$ as before. We first generalize Theorem~\ref{thm:Panchenko} 
to this setting. 

\begin{theorem} \label{gen_con_par_thm}
Assume $\xi\colon\R^{\kappa\times \kappa}\to\R$ satisfies~\ref{xi_power} and~\ref{xi_convex}.
Then for any $d\in\DD$, 
\begin{subequations} \label{gen_con_par_eq}
\eeq{ \label{gen_con_par_eq_a}
\lim_{\eps\searrow0}\limsup_{N\to\infty}\eff_{N,\xi}(d,\eps) 
=\lim_{\eps\searrow0}\liminf_{N\to\infty}\eff_{N,\xi}(d,\eps)
=\inf_{\pi\in\Pi_d} \PP_{\xi}(\pi),
}
where the Parisi functional $\PP_{\xi}(\pi)$ is defined in \eqref{def:parisi:ftl}.
Moreover, for any $d^N\in \DD_N$ such that $d^N\to d$ as $N\to\infty$, we have
\eeq{ \label{gen_con_par_eq_b}
\lim_{N\to\infty}\eff_{N,\xi}(d^N)
=\inf_{\pi\in\Pi_d} \PP_{\xi}(\pi).
}
\end{subequations}
\end{theorem}

More importantly, we show that when $d=d_\bal$ and $\xi$ has a certain symmetry---which is the case for the usual Potts spin glass \eqref{Potts_xi}---the Parisi formula \eqref{gen_con_par_eq} can be reduced to have order parameter $\pi \in \Pi^\star$. 
First we define the symmetry condition.
Let $S_\kappa$ denote the symmetric group on $\{1,\dots,\kappa\}$.
Given a permutation $\pmu\in S_\kappa$ and a matrix $R\in\R^{\kappa\times\kappa}$, let $\pmu\act R$ be the matrix obtained by permuting the rows and columns according to $\pmu$. That is,
\eeq{ \label{pmuR_def}
(\pmu\act R)_{k,k'} \coloneqq R_{\pmu^{-1}(k),\pmu^{-1}(k')}, \quad k,k'\in\{1,\dots,\kappa\}.
}
We then make the following assumption:
\begin{equation} \label{xi_symmetric} \tag{A3}
\xi(\pmu\act R)=\xi(R) \quad \text{for every $\pmu\in S_{\kappa}$ and $R\in \R^{\kappa\times \kappa}$}.
\end{equation}

The following is our main result: it generalizes Theorem~\ref{thm:main:Potts} to the setting \eqref{general_cov}.

\begin{theorem}\label{thm:main}
Assume $\xi\colon\R^{\kappa\times \kappa}\to\R$ satisfies~\ref{xi_power},~\ref{xi_convex}, \eqref{xi_symmetric}. 
Then for any $d^N\in\DD_N$ such that $d^N\to d_\bal$ as $N\to\infty$, we have
\eeq{ \label{thm:main_eq}
\lim_{N\to\infty} \eff_{N,\xi}(d^N)
=\lim_{\eps\searrow 0}\limsup_{N\to\infty}\eff_{N,\xi}(d_{\bal},\eps)
=\lim_{\eps\searrow 0}\liminf_{N\to\infty}\eff_{N,\xi}(d_{\bal},\eps)
=\inf_{\pi\in\Pi^\star}\PP_\xi(\pi),
}
where $\Pi^\star$ is defined in \eqref{def:Pi:star} and $\PP_{\xi}(\pi)$ is defined in \eqref{def:parisi:ftl} (see also \eqref{rmk_lambda_eq}). 
\end{theorem}

We close this section by explaining why Theorem~\ref{thm:main} only applies to the balanced case $d=d_\bal$.
Given a permutation $\pmu\in S_\kappa$ and a vector $u=(u_1,\dots,u_\kappa)\in\R^\kappa$, let $\pmu\act u$ be the vector obtained by permuting the coordinates of $u$:
\begin{subequations} \label{dvd2}
\eeq{ \label{dvd2_a}
(\pmu\act u)_k = u_{\pmu^{-1}(k)}, \quad k\in\{1,\dots,\kappa\}.
}
In particular, for any standard basis vector $\vv e_a$, we have $\pmu\act\vv e_a=\vv e_{\pmu(a)}$.
Now extend this action to $\sigma = (\sigma_1,\dots,\sigma_N)\in\Sigma^N$ by
\eeq{ \label{dvd2_b}
\pmu\act\sigma = (\pmu\act\sigma_1,\dots,\pmu\act\sigma_N), \quad \sigma\in\Sigma^N.
}
\end{subequations}
The original Potts Hamiltonian \eqref{potts_hamiltonian} is clearly unchanged in distribution under any such action.
This fact is generalized by assumption \eqref{xi_symmetric}, which implies
(see Lemma~\hyperref[lem:symmetry_cavity_a]{\ref*{lem:symmetry_cavity}\ref*{lem:symmetry_cavity_a}})
\begin{equation}\label{eq:symmetry}
\big(H_{N,\xi}(\sigma)\big)_{\sigma\in \Sigma^{N}}\stackrel{\mathrm{law}}{=}\big(H_{N,\xi}(\pmu\act\sigma)\big)_{\sigma\in \Sigma^{N}}.
\end{equation}
Unfortunately this symmetry is not particularly helpful for the constrained model \eqref{explicit_xi}, since the constrained configuration space $\Sigma^N(d)$ from \eqref{eq:def:Sigma:N:d:eps} is not invariant under these permutations.
But there is one exception: when $d = d_\bal$.
In this case, the map $\sigma\mapsto\pmu\act\sigma$ is a bijection $\Sigma^N(d_\bal)\to\Sigma^N(d_\bal)$.
The central objective of the paper is to capitalize on this basic observation.
Roughly speaking, Theorem~\ref{thm:main} is the statement that when $d=d_\bal$, the order parameter $\pi\colon(0,1]\to\Gamma_\kappa(d_\bal)$ in \eqref{eq:Parisi:ftl:lagrange} must reflect the symmetry offered by \eqref{eq:symmetry}.
That is, $\pi$ should actually map to the space $\Gamma^\star$ defined in \eqref{def:Gamma:star}, which is precisely the subset of $\Gamma_\kappa(d_\bal)$ consisting of those matrices $\gamma$ which satisfy $\pmu\act\gamma=\gamma$ for every $\pmu\in S_\kappa$. 


\subsection{Parisi functional}
\label{subsec:Parisi}
In this section we define the Parisi functional $\PP_{\xi}$ for any $\xi$ satisfying~\ref{xi_power}.
We denote the derivative of $\xi$ by $\nabla\xi\colon\R^{\kappa\times\kappa}\to \R^{\kappa\times \kappa}$, which is a matrix-valued function with entries given by
\eq{ 
[\nabla\xi(R)]_{k,k^\prime} \coloneqq \frac{\partial\xi}{\partial R_{k,k^\prime}}(R), \quad k,k^\prime\in\{1,\dots,\kappa\}.
}
In addition, define the function $\vartheta_{\xi}\colon\R^{\kappa\times \kappa}\to \R$ by
\eeq{ \label{vthet_def}
\vartheta_\xi(R)
\coloneqq\iprod{R}{\nabla \xi(R)} -\xi(R),
}
where $\iprod{A}{B} =\tr(A^{\sT}B)$ denotes the inner product of two matrices $A,B$. 
Under assumption~\ref{xi_power}, the functions $\xi$, $\nabla\xi$, and $\vartheta_{\xi}$ are nondecreasing when restricted to the set $\Gamma_\kappa$ of positive-semidefinite matrices; see Proposition~\hyperref[prop:xi:theta_b]{\ref*{prop:xi:theta}\ref*{prop:xi:theta_b}}.

Recall the set of paths $\Pi_d$ from \eqref{def:Pi_d}.
We first define $\PP_{\xi}(\pi)$ for any discrete path $\pi$, and then extend continuously to general $\pi$.
By a discrete path we mean an element of the set
\eeq{ \label{def_Pi_disc}
\Pi_d^\disc \coloneqq \{\pi\in\Pi_d:\, \text{$\pi$ takes only finitely many values, $\pi(1)=\diag(d)$}\}.
}
Thus every $\pi\in\Pi_d^\disc$ has the form
\begin{subequations} \label{eq:pi:discrete}
\eeq{
\pi(t)= \gamma_r\quad\textnormal{for $t\in(m_{r-1},m_r]$, $r\in\{1,\dots,s\}$},
}
for some sequence of weights
\eeq{ \label{eq:m}
0 = m_0 < m_1 < \cdots < m_{s-1} < m_s = 1,
}
and some sequence of matrices
\eeq{ \label{eq:gamma:discrete}
0 = \gamma_0 \prec \gamma_1 \prec \cdots \prec \gamma_{s-1} \prec \gamma_{s} = \mathrm{diag}(d).
}
\end{subequations}
Given these sequences, consider independent centered Gaussian random vectors $z_0,\dots,z_{s-1}$ in $\R^\kappa$, with covariance structure
\eeq{ \label{ghce3}
\E(z_rz_r^\sT) = \nabla\xi(\gamma_{r+1}) - \nabla\xi(\gamma_{r})\one\{r>0\}.
}
Given a parameter $\lambda = (\lambda_1,\dots,\lambda_\kappa) \in \R^{\kappa}$ which serves as a Lagrange multiplier, define
\eq{
X_{s} = \log\sum_{k=1}^\kappa\exp\iprod[\Big]{\sum_{r=0}^{s-1}z_r+\lambda}{\vv e_k}.
}
Using $\E_r$ to denote expectation with respect to $z_r$, we then define inductively
\eq{
X_r = \frac{1}{m_r}\log\E_r\exp(m_rX_{r+1}) \quad \text{for $r\in\{1,\dots,s-1\}$},\quad\textnormal{and}\quad  X_0 = \E_0(X_1).
}
We then record the non-random quanity $X_0$ as a function of $\pi$ and $\lambda$:
\begin{equation}\label{eq:def:Psi}
\PPP_{\xi}^{(1)}(\pi,\lambda)
\coloneqq 
  X_0.
\end{equation}
We show in Lemma~\ref{pconlem} 
that $\Pi_d^{\disc} \ni \pi \mapsto \PPP_{\xi}^{(1)}(\pi,\lambda)$ is Lipschitz continuous with respect to the following $L^1$ norm on paths:
\eeq{ \label{norm}
\pi\mapsto\int_0^1\|\pi(t)\|_1\ \dd t.
}
Since $\Pi_d^{\disc}\subset \Pi_d$ is dense with respect to this norm, there is a unique continuous extension of $\PPP_{\xi}^{(1)}$ to all of $\Pi_d$. 
Next define
\eeq{ \label{PPP2_def}
\PPP_{\xi}^{(2)}(\pi)
&\coloneqq \frac{1}{2}\int_0^1\vartheta_\xi(\pi(t))\ \dd t - \frac{1}{2}\vartheta_\xi(\diag(d)), \quad \pi\in\Pi_d,
}
and set
\eeq{\label{eq:Parisi:ftl:lagrange}
\PPP_{\xi}(\pi, \lambda)
\coloneqq \PPP_{\xi}^{(1)}(\pi,\lambda) + \PPP_{\xi}^{(2)}(\pi).
}
Finally, the Parisi functional is defined as
\begin{equation}\label{def:parisi:ftl}
    \PP_{\xi}(\pi)\coloneqq\inf_{\lambda\in \R^{\kappa}}[\PPP_{\xi}(\pi,\lambda)-\iprod{\lambda}{d}], \quad \pi\in\Pi_d.
\end{equation}

\begin{remark} \label{rmk_lambda}
The reason for the Lagrange multiplier $\lambda$ and the infimum in \eqref{def:parisi:ftl} is later explained in Remark~\ref{rmk_dual}.
It is straightforward to check that $\PPP_{\xi}(\pi,\lambda)=\PPP_{\xi}(\pi,\lambda+c\bone )$ holds for any $c\in \R$. 
Therefore, one could fix the last coordinate $\lambda_\kappa=0$ so that the infimum in \eqref{def:parisi:ftl} is over $\R^{\kappa-1}$; this was done in \cite{panchenko18a}.
More importantly, we show in Lemma~\ref{pr457} that for $\pi \in \Pi^\star$ and $\xi$ satisfying the symmetry condition \eqref{xi_symmetric}, the infimum 
is achieved at $\lambda=0$:
\eeq{ \label{rmk_lambda_eq}
\PP_\xi(\pi) = \PPP_{\xi}(\pi,0) \quad \text{for any $\pi\in\Pi^\star$}.
}
Consequently, Theorem~\ref{thm:main} shows that the Lagrange parameter $\lambda$ is not needed for symmetric models, although it remains necessary for the proof.
\end{remark}

\subsection{Related literature}

The Potts spin glass model is closely related to the max $\kappa$-cut problem on random graphs.
As mentioned after Remark~\ref{rmk:fop}, the balanced case addressed in this paper corresponds to sparse Erd\H{o}s--Rényi graphs or random regular graphs \cite{sen18}, both of which have distributional symmetry analogous to \eqref{xi_symmetric}.
An unbalanced version (with $\kappa=2$) was studied in \cite{jagannath_sen21}, corresponding to the generalized SK model introduced in \cite{panchenko05b}. 
Motivated by the max $\kappa$-cut problem on inhomogeneous random graphs, \cite{jagannath-ko-sen18} introduced a vector version of multi-species SK model and combined approaches from \cite{panchenko15,panchenko18a} to obtain a Parisi formula for the limiting free energy.
As in this paper, a key technical component was the synchronization mechanism discussed in Section~\ref{intro_model}.

The synchronization technique was first introduced by Panchenko~\cite{panchenko15} to study the multi-species SK model. 
More specifically, synchronization enabled a free energy lower bound that matched the upper bound obtained in \cite{barra-contucci-mingione-tantari15} via Guerra's replica symmetry breaking interpolation \cite{guerra03}. 
These methods were generalized in \cite{panchenko18a} to obtain the variational formula \eqref{eq:Panchenko:formula}, which is a special case of the mixed vector spin model considered in \cite{panchenko18b}. 
Synchronization has since been used in a variety of generalized spin glasses, including 
a multi-scale SK model \cite{contucci-mingione19}, 
the quantum SK model \cite{adhikari_brennecke20}, 
spherical spin glasses with constrained overlaps \cite{ko20}, 
and multi-species spherical models \cite{bates-sohn22a}.
These generalized models are related to a number of problems in statistical inference and combinatorial optimization, such as 
spiked random tensors \cite{chen19,chen-handschy-lerman21}, 
principal submatrix recovery~\cite{gamarnik_jagannath_sen21}, 
and the $\ell^p$--Gaussian--Grothendieck problem~\cite{chen_sen23,dominguez22}. 

Synchronization also been used in conjunction with a strategy initiated by Mourrat~\cite{mourrat21a,mourrat22} that identifies the Parisi formula as a solution to a Hamilton--Jacobi equation.
This program has also led to results on spin glasses enriched by a certain magnetic field \cite{mourrat-panchenko20}, non-convex models such as the bipartite SK model \cite{mourrat21b}, and vector spin models \cite{mourrat23,chen_mourrat23_arxiv}.
Building on these developments, Chen~\cite{chen23a_arxiv} recently showed that by introducing a self-overlap correction term in the free energy, one can remove the supremum over magnetizations in \eqref{eq:Panchenko:formula}.
Moreover, the self-overlap concentrates in that setting, thereby softly enforcing a balanced constraint \cite{chen23b_arxiv}.

For spherical vector spin glasses, Ko~\cite{ko?} obtained the Crisanti--Sommers variational formula for the limiting free energy.
Various properties of the minimizer to this formula were obtained by Auffinger and Zhou~\cite{auffinger_zhou22}, who also extended the formula to zero temperature.
In the SK case, the limiting free energy was computed earlier by Panchenko and Talagrand~\cite[Thm.~2]{panchenko-talagrand07}.
More recently, Husson and Ko~\cite{husson_ko22_arxiv} gave an alternative proof using spherical integrals, even allowing $\kappa$ to grow sublinearly with $N$.

\subsection{Organization of the paper}
The paper's format is meant to prioritize new and central ideas, by postponing certain technical aspects that could be otherwise distracting. 
Section~\ref{sec:proof:overview} contains the proof of our main result, Theorem~\ref{thm:main}.
The proof invokes various inputs that are introduced in Section~\ref{sec:proof:overview} but proved later in the paper.
Section~\ref{sec_preliminaries} establishes some notation involving the Parisi functional and Ruelle probability cascades that will be needed in all subsequent parts.
Section~\ref{sec:diff} concerns the differentiability of the Parisi functional with respect to certain inverse temperatures, while
Section~\ref{sec:continuity_and_duality} contains various continuity properties of the Parisi functional.
Section~\ref{sec_lower_bound} unites these ingredients with symmetry to complete the lower bound portion of Theorem~\ref{thm:main}. 

Several intermediate results are either standard or very similar to previous works.
For some of these results, the proofs can be read essentially verbatim elsewhere and are thus omitted.
For others, we defer the proofs to one of several appendices.
These include the existence of certain Gaussian processes (Appendix~\ref{sec:app:generic}), proof of the Aizenman--Sims--Starr scheme (Appendix~\ref{sec_ass_proof}), checking various lemmas involving Ruelle probability cascades (Appendix~\ref{sec_app_rpc}), generalizing a variational argument from \cite{panchenko18a} related to the infimum in \eqref{def:parisi:ftl} (Appendix~\ref{subsec:lem:duality}), and finally verification of the upper bound via Guerra interpolation (Appendix~\ref{sec_gen_par_proof}).

\subsection{Acknowledgments}
We are grateful to Hong-Bin Chen, Amir Dembo, Justin Ko, and Nike Sun for helpful discussions.
%


\section{Proof overview}
\label{sec:proof:overview}
We introduce a series of intermediate results in Sections~\ref{sec_overview_arrays}--\ref{subsec:symmetric:Parisi}, and then use them to prove Theorem~\ref{thm:main} (and Theorem~\ref{thm:main:Potts}) in Section~\ref{sec:proof_main}.
The results labeled as propositions will be proved in the remainder of the paper, while lemmas are argued immediately.
Meanwhile, Theorem~\ref{gen_con_par_thm} is proved using the same strategy as in \cite{panchenko18a}, and so we will not focus on it here.
Instead a review of the proof is provided in Appendix~\ref{sec_gen_par_proof}.

Toward our goal of Theorem~\ref{thm:main}, we first observe that the upper bound
\eq{
\lim_{\eps\searrow0}\limsup_{N\to\infty}\eff_{N,\xi}(d_{\bal},\eps)\leq \inf_{\pi \in \Pi^\star}\PP_{\xi}(\pi)
}
is immediate from Theorem~\ref{gen_con_par_thm}, since $\Pi^\star$ is a subset of $\Pi_{d_{\bal}}$.
What requires novel justification is the lower bound, which amounts to showing
\begin{equation}\label{eq:goal}
\liminf_{N\to\infty} \eff_{\kappa N,\xi}(d_{\bal})\geq \PP_{\xi}(\pi) \quad \text{for some $\pi\in\Pi^\star$}.
\end{equation}
The factor $\kappa$ in $\kappa N$ is to guarantee that $\Sigma^{\kappa N}(d_{\bal})$ is nonempty.


The standard approach for proving a Parisi formula lower bound such as \eqref{eq:goal} is the Aizenman--Sims--Starr (A.S.S.) scheme~\cite{aizenman-sims-starr03,aizenman-sims-starr07}.
Indeed, this method is what yields the lower bound portion of Theorem~\ref{gen_con_par_thm}.
The shortcoming is that this implementation alone yields a minimizer $\pi$ belonging to $\Pi_{d_\bal}$, whereas we want more specific information, namely that $\pi$ can be found in $\Pi^\star$.
So we will actually use the A.S.S. scheme a \textit{second} time, but with some interventions that preserve the symmetry needed to infer Theorem~\ref{thm:main} from Theorem~\ref{gen_con_par_thm}.
The next three sections explain how this is accomplished.

\subsection{Synchronized overlap arrays} \label{sec_overview_arrays}

The main characters in the A.S.S. scheme are overlap arrays.
Given any $\sigma^1,\sigma^2,\ldots\in\Sigma^N$, we can define an array $\RR=(\RR_{\ell,\ell'})_{\ell,\ell'\geq1}$, where $\RR_{\ell,\ell'}$ is the $\kappa\times\kappa$ overlap matrix associated to $\sigma^\ell$ and $\sigma^{\ell'}$:
\begin{equation}\label{def:overlap}
\RR_{\ell,\ell^\prime}= R(\sigma^{\ell},\sigma^{\ell^\prime})\coloneqq\frac{\sigma^{\ell} (\sigma^{\ell^\prime})^{\sT}}{N}=\frac{1}{N}\sum_{i=1}^{N}\sigma^{\ell}_i (\sigma^{\ell^\prime}_i)^\sT \in \R^{\kappa\times \kappa}.
\end{equation}
In other words, the $(k,k')$ entry of $R(\sigma,\sigma')$ is the fraction of columns $i$ that satisfy $\sigma_i=\vv e_k$ and $\sigma'_i = \vv e_{k'}$:
\eeq{ \label{def:overlap_entry}
R(\sigma,\sigma')_{k,k'} = \frac{1}{N}\sum_{i=1}^N\one\{\sigma_i=\vv e_k\}\one\{\sigma'_i=\vv e_{k'}\}.
}

The case of interest is when $(\sigma^\ell)_{\ell\geq1}$ are independent samples from the Gibbs measure $G_{\kappa N,\xi}$ associated to the Hamiltonian 
$H_{\kappa N,\xi}$ from \eqref{general_cov}:
\begin{equation}\label{def:gibbs}
G_{\kappa N,\xi}(\sigma)\propto \exp H_{\kappa N,\xi}(\sigma), \quad\sigma\in \Sigma^{\kappa N}(d_{\bal}).
\end{equation}
This measure is random depending on the Gaussian process $H_{\kappa N,\xi}$.
Therefore, the array $\RR$ is generated by first fixing the realization of $G_{\kappa N,\xi}$, and then drawing i.i.d.~samples $(\sigma^\ell)_{\ell\geq1}$ from $G_{\ka N,\xi}$. 
We denote the resulting distribution by $\Law(\RR;\E(G_{\kappa N,\xi}^{\otimes\infty}))$.
This is a probability measure on the space $(\R^{\kappa\times\kappa})^{\N\times\N}$ equipped with the usual product $\sigma$-algebra.
Its properties generalize those of classical Gram--de Finetti arrays, as captured by the following definition.
Recall that $R^{\circ p}$ denotes the $p^\mathrm{th}$ Hadamard power of matrix $R$, and the permutation action $R\mapsto\pmu\act R$ from \eqref{pmuR_def}.


\begin{definition} \label{type_def}
A \textit{$\kappa$-dimensional Gram--de Finetti array} is a random array of $\kappa\times\kappa$ matrices $\RR = (\RR_{\ell,\ell'})_{\ell,\ell'\geq1}$ such that for every positive integer $p$ and $w\in\R^\kappa$, the array $\QQ = (\iprod{\RR_{\ell,\ell'}^{\circ p} w}{w})_{\ell,\ell'\geq1}$ is a ($1$-dimensional) Gram--de Finetti array. 
That is, $\QQ$ is almost surely symmetric and positive-semidefinite, and is exchangeable under any finite permutation applied to its rows and columns simultaneously. 
In this case, we say $\Law(\RR)$ is a \textit{Gram--de Finetti law}.
\end{definition}


Roughly speaking, the A.S.S. scheme works by identifying a functional $\Psi_\xi$ on Gram--de Finetti laws (see Section~\ref{prelimit_section}) such that
\eeq{ \label{fnc73}
\eff_{\kappa N,\xi} + o(1) \geq \Psi_\xi\big(\Law(\RR;\E(G_{\kappa N,\xi}^{\otimes\infty}))\big).
}
As a first step to go from \eqref{fnc73} to \eqref{eq:goal}, one assumes (by passing to a subsequence) that 
\eeq{ \label{gwpx9}
\Law(\RR;\E(G_{\kappa N,\xi}^{\otimes\infty})) \quad \text{converges weakly as $N\to\infty$ to some $\LL$}.
}
By continuity of $\Psi_\xi$ (see Corollary~\ref{extension_cor}), this results in
\eeq{ \label{gv83}
\liminf_{N\to\infty}\eff_{\kappa N,\xi}(d_\bal) \geq \Psi_\xi(\LL).
}
But this limiting law $\LL$ is now divorced from the representation \eqref{def:overlap}.
There is not necessarily any way of realizing $\LL$ by sampling from a Gibbs measure, unless $\LL$ satisfies a certain family of identities introduced in \cite{panchenko18a} and recalled in the following definition.

\begin{definition}\label{def:GG}
A Gram--de Finetti law $\LL=\Law(\RR)$ is said to satisfy the ($\kappa$-dimensional) Ghirlanda--Guerra (G.G.) identities if the following holds for every 
$p\geq1$, $m\geq1$, $w_1,\ldots,w_m\in[-1,1]^\kappa$, and bounded measurable function $\vphi\colon\R^m\to\R$.
Define the scalar array
\eeq{ \label{Qphi}
    \AA_{\ell,\ell^\prime}\coloneqq\vphi\big(\iprod{\RR_{\ell,\ell'}^{\circ p} w_{1}}{w_{1}},\ldots, \iprod{\RR_{\ell,\ell'}^{\circ p} w_{m}}{w_{m}}\big).
}
Then for any $n\geq1$ and any bounded measurable function $f$ of the finite subarray $\RR^{(n)}= (\RR_{\ell,\ell'})_{1\leq\ell,\ell'\leq n}$, we have
    \begin{equation}\label{eq:GG}
    \E[f(\RR^{(n)})\AA_{1,n+1}]=\frac{1}{n}\E[f(\RR^{(n)})]\E[\AA_{1,2}]+\frac{1}{n}\sum_{\ell=2}^{n}\E[f(\RR^{(n)})\AA_{1,\ell}].
    \end{equation}
\end{definition}

The theorem quoted below clarifies the role of these identities: to guarantee \textit{synchronization}.
In fact, they also guarantee that each matrix in the array is positive-semidefinite.
The condition \eqref{deterministic_for_same_ell} is why the Gibbs measure in \eqref{def:gibbs} is restricted to $\Sigma^{\kappa N}(d_\bal)$.
This restriction ensures that the array in \eqref{def:overlap} satisfies 
\eeq{ \label{jed8}
\RR_{\ell,\ell}=\diag(d_\bal)=\kappa^{-1}\bI_\kappa \quad \text{for every $\ell$}.
}

\begin{theirthm} \label{sync_thm} \textup{\cite[Thm.~4]{panchenko18b}}
Assume $\Law(\RR)$ is a Gram--de Finetti law in the sense of Definition~\ref{type_def} and satisfies the Ghirlanda--Guerra identities in Definition~\ref{def:GG}.
Suppose further that there is some deterministic $D\in\Gamma_\kappa$ with $\tr(D)=1$ such that
\eeq{ \label{deterministic_for_same_ell}
\RR_{\ell,\ell} = D \quad \text{for all $\ell\geq1$, with probability one}.
}
Then there exists a map $\Phi\colon[0,1]\to\Gamma_\kappa$ such that with probability one,
\eq{ 
\RR_{\ell,\ell'} = \Phi\big(\tr(\RR_{\ell,\ell'})\big) \quad \text{for every $\ell,\ell'\geq1$}.
}
Furthermore, this map can be taken to be nondecreasing,
\eq{
\Phi(q)\preceq\Phi(q') \quad \text{for $q\leq q'$},
}
and Lipschitz continuous, 
\eq{
\|\Phi(q)-\Phi(q')\|_1\leq C_\kappa |q-q'|,
}
for some constant $C_\kappa$ depending only on $\kappa$.
\end{theirthm}

So when the $\kappa$-dimensional G.G. identities are satisfied, the matrix array $(\RR_{\ell,\ell'})_{\ell,\ell\geq1}$ is simply a function of the scalar array $(\tr(\RR_{\ell,\ell'}))_{\ell,\ell'\geq1}$.
Moreover, this scalar array satisfies the $1$-dimensional (or canonical) G.G. identities (see Lemma~\ref{lem_kappa_to_1}).
In turn, the $1$-dimensional G.G. identities \cite{ghirlanda-guerra98} are known to yield the following existence and uniqueness result via ultrametricity \cite{panchenko13b}.
We use $\QQ$ and $\bar\LL$ (instead of $\RR$ and $\LL$) as notational cues that we are speaking of $1$-dimensional Gram--de Finetti arrays.

\begin{theirthm} \label{representation_thm}
\textup{\cite[Thm.~2.13, 2.16, and 2.17]{panchenko13a}}
Let $\mu$ be a probability measure on $[0,1]$.
There is exactly one Gram--de Finetti law $\bar\LL_\mu =\Law(\QQ)$ that has all three of the following properties: (i) $\bar\LL_\mu$ satisfies the $1$-dimensional G.G. identities; (ii) $\Law(\QQ_{1,2})=\mu$; and (iii) $\QQ_{\ell,\ell}=1$ with probability one, for every $\ell\geq1$.
Furthermore, the map $\mu\mapsto\bar\LL_\mu$ is continuous with respect to weak convergence.
\end{theirthm}

In summary, every Gram--de Finetti law $\LL=\Law(\RR)$ satisfying the $\kappa$-dimensional G.G. identities can be described by a pair $(\Phi,\mu)$, where $\Phi$ is the synchronization map from Theorem~\ref{sync_thm}, and $\mu=\Law(\tr(\RR_{1,2}))$.
This pair induces a path $\pi=\Phi\circ Q_\mu$, where $Q_\mu$ is the quantile function of $\mu$.
There is thus a correspondence $\LL\leftrightarrow\pi$, and under this correspondence one checks that $\Psi_\xi(\LL)=\PP_\xi(\pi)$ (see Lemma~\ref{PPPPsi}).
In our case, the $\LL$ of interest satisfies \eqref{jed8}, which means $\Phi$ maps into $\Gamma_\kappa(d_\bal)$, and thus $\pi\in\Pi_{d_\bal}$.
Hence \eqref{gv83} leads to
\eeq{ \label{5cvr}
\liminf_{N\to\infty} \eff_{\kappa N,\xi}(d_\bal) \geq \PP_\xi(\pi) \quad \text{for some $\pi\in\Pi_{d_\bal}$}.
}
To go further and say $\pi$ belongs to the smaller set $\Pi^\star$ from \eqref{def:Pi:star},
we make a simple but important observation stated in Lemma~\ref{lemma:symmetric} below.
We preface the result with a definition.

\begin{definition} \label{def:L:symmetry}
For an array of $\kappa\times\kappa$ matrices $\RR = (\RR_{\ell,\ell'})_{\ell,\ell'\geq1}$ and a permutation $\pmu\in S_\kappa$, we write $\pmu\act\RR$ to denote the array $(\pmu\act \RR_{\ell,\ell'})_{\ell,\ell'\geq1}$, where $\pmu\act\RR_{\ell,\ell'}$ is defined in \eqref{pmuR_def}. 
When $\RR$ is a random array, we say that $\Law(\RR)$ is \textit{symmetric} if $\Law(\RR)=\Law(\pmu\act\RR)$ for every $\pmu\in S_\kappa$.
\end{definition}

The following result shows that once this symmetry condition is added to the hypotheses of Theorem~\ref{sync_thm}, each matrix in the array has all of its diagonal entries equal, and all of its off-diagonal entries equal.
In the statement below, $\RR^{k,k'}_{\ell,\ell'}$ denotes the $(k,k')$ entry of $\RR_{\ell,\ell'}$.

\begin{lemma}\label{lemma:symmetric}
    Assume $\Law(\RR)$ satisfies the hypotheses of Theorem~\ref{sync_thm} and is symmetric in the sense of Definition~\ref{def:L:symmetry}.
   Then for all $k_1\neq k_1^\prime$ and $k_2\neq k_2^\prime$ and all $\ell,\ell^\prime \geq 1$,
\eeq{\label{eq:prop:symmetric}
\RR_{\ell,\ell'}^{k_1,k_1}=\RR_{\ell,\ell'}^{k_2,k_2}\quad\textnormal{and}\quad \RR_{\ell,\ell'}^{k_1,k_1^\prime}=\RR_{\ell,\ell'}^{k_2,k_2^\prime} \quad\mathrm{a.s.}
}
In particular, if $\RR_{\ell,\ell^\prime}\in \Gamma_{\kappa}(d_{\bal})$ almost surely, then $\RR_{\ell,\ell'}\in\Gamma^\star$ almost surely.
\end{lemma}

\begin{proof}
Consider any $k,k^\prime \in\{1,\dots,\kappa\}$, $\ell,\ell^\prime \geq 1$, and $\pmu\in S_{\kappa}$. 
Since $\RR_{\ell,\ell^\prime}\stackrel{\text{law}}{=}\pmu\act \RR_{\ell,\ell^\prime}$ by assumption, it follows that
\eq{
\big(\RR_{\ell,\ell'}^{k,k^\prime}, \tr(\RR_{\ell,\ell'})\big) \stackrel{\text{law}}{=} \big(\RR_{\ell,\ell'}^{\pmu^{-1}(k),\pmu^{-1}(k^\prime)},\tr (\RR_{\ell,\ell'})\big),
}
where we used the fact that $\tr(R_{\ell,\ell^\prime})= \tr(\pmu\act R_{\ell,\ell^\prime})$ holds for any $\pmu\in S_{\kappa}$. In particular,
\eq{
\E\Big[\RR_{\ell,\ell^\prime}^{k,k^\prime}\Bgiven \tr(\RR_{\ell,\ell^\prime})\Big]=\E\Big[\RR_{\ell,\ell^\prime}^{\pmu^{-1}(k),\pmu^{-1}(k^\prime)}\Bgiven \tr(\RR_{\ell,\ell^\prime})\Big] \quad \mathrm{a.s.}
}
On the other hand, by Theorem~\ref{sync_thm}, $R_{\ell,\ell^\prime}$ is measurable with respect to $\tr(R_{\ell,\ell^\prime})$. 
Hence
\eq{
\RR_{\ell,\ell^\prime}^{k,k^\prime}=\E\Big[\RR_{\ell,\ell^\prime}^{k,k^\prime}\Bgiven \tr(\RR_{\ell,\ell^\prime})\Big]=\E\Big[\RR_{\ell,\ell^\prime}^{\pmu^{-1}(k),\pmu^{-1}(k^\prime)}\Bgiven \tr(\RR_{\ell,\ell^\prime})\Big]=\RR_{\ell,\ell^\prime}^{\pmu^{-1}(k),\pmu^{-1}(k^\prime)} \quad \mathrm{a.s.}
}
As this holds for any permutation $\pmu\in S_\kappa$, \eqref{eq:prop:symmetric} follows. 

To justify the final sentence in the proposition, recall from definition \eqref{def:Gamma:star} that $\Gamma^\star$ is precisely the subset of $\Gamma_\kappa(d_\bal)$ consisting of matrices that are constant both on the diagonal and off the diagonal.
\end{proof}

Under the condition of symmetry, Lemma~\ref{lemma:symmetric} specifies exactly what the synchronization map in Theorem~\ref{sync_thm} must be:
\eq{
\Phi^\star(q) = \frac{q}{\kappa}\bI_\kappa + \frac{1-q}{\kappa^2}\bone\bone^\sT, \quad q\in[0,1].
}
This suggests that under the assumption \eqref{xi_symmetric}, the previous lower bound \eqref{5cvr} can be improved to \eqref{eq:goal}.
Indeed, \eqref{eq:symmetry} implies that $\Law(\RR;\E(G_{N,\xi}^{\otimes\infty}))$ is symmetric in the sense of Definition~\ref{def:L:symmetry} (see Lemma~\hyperref[lem:symmetry_cavity_b]{\ref*{lem:symmetry_cavity}\ref*{lem:symmetry_cavity_b}}), and this symmetry trivially passes to the limit law $\LL$ in \eqref{gwpx9}.
The crucial detail is that $\LL$ must also satisfy the G.G. identities for us to obtain \eqref{5cvr} in the first place, but there is a major obstacle: the G.G. identities are not known to hold!

There are two natural options for moving forward, both of which run into difficulty:
\begin{enumerate}[label=\textup{(\Roman*)}]
\item \label{approach1}
The standard approach for obtaining the G.G. identities is to perturb the Hamiltonian by lower-order terms with random coefficients, and then average over these coefficients \cite{ghirlanda-guerra98,talagrand11b}. 
Indeed, this strategy was used in \cite{panchenko18a,panchenko18b} to obtain Theorem~\ref{thm:Panchenko}.
Because the perturbations do not change the limiting free energy, one can obtain a law $\LL$ satisfying both \eqref{gv83} and the G.G. identities, but it would not necessarily be equal to the limit in \eqref{gwpx9}.
Instead, the Gibbs measure $G_{\kappa N,\xi}$ would be replaced by one corresponding to the perturbed Hamiltonian, whose covariance function only satisfies the symmetry condition \eqref{xi_symmetric} \textit{asymptotically}.
Since Gibbs measures in spin glass theory are known to exhibit temperature chaos \cite{chen_panchenko17, benarous-subag-zeitouni20}, this is not enough to ensure $\LL$ is still symmetric in the sense of Definition~\ref{def:L:symmetry}.
Without this property, Lemma~\ref{lemma:symmetric} is not applicable.

    \item \label{approach2}
    One could instead try to obtain the G.G. identities directly, without any perturbation.
    This would allow symmetry to be preserved, but the task of proving the G.G. identities (or some version thereof) remains a major open problem even for SK model \cite{talagrand10}.
    Therefore, this approach does not presently seem viable.
    
\end{enumerate}

It would thus appear that the G.G. identities needed for the existence of a synchronization are at odds with the symmetry needed to say more about the map itself.
We thus employ a different strategy: we introduce a \textit{generic} version of the Potts spin glass for which we can obtain the G.G. identities via differentiability, without averaging over perturbations. 

Inspired by \cite{panchenko10} and \cite[Sec.~3.7]{panchenko13a}, the generic model works by augmenting the covariance function $\xi$ with countably many polynomials of the form $\xi_\theta$ from \eqref{def:xi:theta}.
If these polynomials are selected so that they span a dense subset of all continuous functions, and if the array $(\xi_\theta(R_{\ell,\ell'}))_{\ell,\ell'\geq1}$ satisfies the $1$-dimensional G.G. identities for each $\theta$, then $(\RR_{\ell,\ell'})_{\ell,\ell'\geq1}$ will satisfy the $\kappa$-dimensional G.G. identities.
The same principle underlies approach~\ref{approach1}, but here the new terms added to the Hamiltonian will be of the same order as the Hamiltonian itself.
Therefore, the limiting free energy of the augmented model will be different from the original, but still fall under the purview of Theorem~\ref{gen_con_par_thm}. 
Each new term in the Hamiltonian will be modulated by an inverse temperature parameter $\beta_\theta$ which is fixed and \textit{does not require averaging}. 
We will ultimately choose these parameters to be sufficiently small, but all nonzero, and preserve symmetry of $\xi$ in \eqref{xi_symmetric}. The next section defines this generic model.

\subsection{The generic model and differentiability}\label{subsec:generic}

Recall the set $\Theta$ in \eqref{eq:def:Theta} and the function $\xi_{\theta}\colon\R^{\kappa\times\kappa}\to\R$ in \eqref{def:xi:theta}. 
Specializing the notation from \eqref{vthet_def}, we denote 
\eeq{ \label{vartheta_theta_def}
\vartheta_\theta(R)\coloneqq\vartheta_{\xi_\theta}(R)=\iprod{R}{\nabla \xi_{\theta}(R)} -\xi_{\theta}(R).
}
The following result provides several important properties of the functions $\xi_{\theta}$, $\nabla\xi_\theta$, $\vartheta_\xi$. 
Under assumption~\ref{xi_power}, parts~\ref{prop:xi:theta_b},~\ref{prop:xi:theta_c}, and~\ref{prop:xi:theta_d} obviously remain true when $\xi_\theta$ is replaced by $\xi$.

\begin{proposition}\label{prop:xi:theta}
For any $\theta\in\Theta$ 
and $N\geq1$, the following hold.
\begin{enumerate}[label=\textup{(\alph*)}]
    \item \label{prop:xi:theta_a}
    $\vartheta_\theta = (\deg(\theta)-1)\xi_\theta$. 

    \item \label{prop:xi:theta_b}
    For any $Q,R\in \Gamma_{\ka}$ such that $Q\preceq R$, we have 
    \eeq{ \label{monotonicity}
        0\leq \xi_{\theta}(Q)\leq\xi_{\theta}(R), 
        \quad 0\preceq\nabla\xi_{\theta}(Q)\preceq\nabla\xi_{\theta}(R), 
        \quad 0\leq \vartheta_\theta(Q)\leq\vartheta_\theta(R).
    }
    
    \item \label{prop:xi:theta_c} There exists a centered Gaussian process $H_{N,\theta}\colon\Sigma^N\to\R$ with covariance 
    \eeq{\label{eq:H:theta:cov}
   \E[H_{N,\theta}(\sigma)H_{N,\theta}(\sigma')]
     = N\xi_\theta\Big(\frac{\sigma\sigma'^{\sT}}{N}\Big).
     }

    \item \label{prop:xi:theta_d}
    There exist centered Gaussian processes $Z_{N,\theta}\colon\Sigma^N\to\R^\kappa$ and $Y_{N,\theta}\colon\Sigma^N\to\R$ with covariances
    \begin{equation*}
\begin{aligned}
\E\Big[Z_{N,\theta}(\sigma)Z_{N,\theta}(\sigma')^\sT\Big] &= 
\nabla\xi_\theta\Big(\frac{\sigma\sigma'^\sT}{N}\Big), \qquad
\E[Y_{N,\theta}(\sigma)Y_{N,\theta}(\sigma')] = 
\vartheta_\theta\Big(\frac{\sigma\sigma'^\sT}{N}\Big).
\end{aligned}
\end{equation*}
\end{enumerate}
\end{proposition}

The proof of Proposition~\ref{prop:xi:theta} is given in Appendix~\ref{sec:app:generic}.
Note that part~\ref{prop:xi:theta_b} was used in \eqref{ghce3} to define the Parisi functional $\PP_\xi$.
Part~\ref{prop:xi:theta_c} guarantees the existence of the Gaussian process $\big(H_{N,\xi}(\sigma)\big)_{\sigma\in \Sigma^N}$ in \eqref{general_cov}, under assumption~\ref{xi_power}.
Finally, part~\ref{prop:xi:theta_d} will be needed for the A.S.S. scheme in Section~\ref{ass_sec}.

While~\ref{xi_power} allows the covariance function $\xi=\sum_{\theta\in\Theta_0}\alpha_\theta^2\xi_\theta$ to use any countable subset $\Theta_0$ of the parameter space $\Theta$, here we consider a particular countable subset: 
\eq{
\Theta_\Q = \Big\{(p,m,n_1,\dots,n_m,w_1,\dots,w_m):\, p,m,n_1,\dots,n_m\geq 1, w_1,\dots,w_m\in\big(\Q\cap[-1,1]\big)^\kappa\Big\}.
}
Given $\xi$ and a family of parameters $\bbeta=(\beta_\theta)_{\theta\in\Theta_\Q}$, we define the modified covariance function $\xi_{\bbeta}\colon\R^{\kappa\times\kappa}\to\R$ by
\eeq{\label{eq:def:xi:bbeta}
\xi_{\bbeta}(R)\coloneqq\xi(R)+\sum_{\theta\in \Theta_\Q}\beta_{\theta}^2\xi_{\theta}(R).
}
As in Remark~\ref{rmk_xi}, the following decay condition on $(\beta_{\theta})_{\theta\in \Theta_\Q}$ guarantees that $\xi_{\bbeta}$ is smooth on the set $\{R\in\R^{\kappa\times\kappa}:\|R\|_\infty\leq1\}$
\begin{equation} \tag{B1} \label{beta_decay1}
\sum_{\theta\in\Theta_\Q}\beta_\theta^2(1+\eps)^{\deg(\theta)}<\infty \quad \text{for some $\eps>0$}.
\end{equation}
Nevertheless, the modified function $\xi_{\bbeta}$ may fail to be convex even if $\xi$ is convex, in which case Theorem~\ref{gen_con_par_thm} would not apply to $\xi_{\bbeta}$.
To avoid this scenario, we impose an additional decay condition on $(\beta_\theta)_{\theta\in\Theta_\Q}$ as follows.

For twice continuously differentiable $f\colon\R^{\kappa\times\kappa}\to \R$,
denote the Hessian of $f$ at $R$ by
\begin{equation*}
\nabla^2 f(R)\coloneqq\Big(\frac{\partial^2 f}{\partial R_{k_1,k_1^\prime}\partial R_{k_2,k_2^\prime}}(R)\Big)_{(k_1,k_1^\prime),(k_2,k_2^\prime)\in \{1,\dots,\kappa\}^2}.
\end{equation*}
Thinking of $\nabla^2 f(R)$ as a (self-adjoint) linear operator $\R^{\kappa\times\kappa}\to \R^{\kappa\times\kappa}$, we can speak of its minimum eigenvalue
\eq{
\lambda_{\min}(\nabla^2 f(R)) \coloneqq \min_{\|Q\|_2=1}\iprod{\nabla^2 f(R)Q}{Q},
}
and its spectral radius
\eq{
\|\nabla^2 f(R)\|_{\op} \coloneqq \max_{\|Q\|_2=1}\big|\iprod{\nabla^2 f(R)Q}{Q}\big|.
}
Now define the following constant for each $\theta\in\Theta$:
    \begin{equation}\label{eq:xi:theta:regularity}
    C_{\theta}\coloneqq\sup_{\|R\|_1\leq 1}\big\{\|\nabla \xi_{\theta}(R)\|_{\infty}\vee\|\nabla^2 \xi_{\theta}(R)\|_{\op}\big\},
    \end{equation}
which is finite by Remark~\ref{rmk_xi}.
Convexity assumption~\ref{xi_convex} means $\lambda_{\min}(\nabla^2\xi(R))\geq0$ for all $R\in\nonpsd$.
The following lemma shows how this condition can be maintained after modification \eqref{eq:def:xi:bbeta} provided we assume a strict inequality for the original covariance function $\xi$.

\begin{lemma} \label{lem_xi_still_good}
Assume $\xi$ satisfies~\ref{xi_power} and
\begin{equation} \tag{A$2'$} \label{xi_strongly_convex}
\inf_{R\in\nonpsd}\lambda_{\min}(\nabla^2 \xi(R)) > 0.
\end{equation}
Assume $\bbeta$ satisfies \eqref{beta_decay1} and
\begin{equation} \tag{B2} \label{beta_decay2}
\sum_{\theta\in\Theta_\Q}\beta_\theta^2 C_\theta < \inf_{R\in\nonpsd}\lambda_{\min}(\nabla^2 \xi(R)).
\end{equation}
Then $\xi_{\bbeta}$ defined in \eqref{eq:def:xi:bbeta} satisfies~\ref{xi_power} and~\ref{xi_convex}.
\end{lemma}

\begin{proof}
That $\xi_{\bbeta}$ satisfies~\ref{xi_power} is immediate from \eqref{beta_decay1}.
For~\ref{xi_convex}, we simply observe that for any $R\in\nonpsd$,
\eq{
    \lambda_{\min}\big(\nabla^2 \xi_{\bbeta}(R)\big)\geq \lambda_{\min}\big(\nabla^2 \xi(R)\big)-\sum_{\theta\in \Theta_\Q}\beta_{\theta}^2\|\nabla^2\xi_\theta(R)\|_{\op}
    \stackref{beta_decay2}{\geq} 0.
}
Hence $\xi_{\bbeta}$ is convex on the set $\nonpsd$.
\end{proof}

The following result is the reason for introducing the modified model \eqref{eq:def:xi:bbeta}.
It is proved in Section~\ref{sec:diff}.

\begin{proposition}\label{prop:diff}
      Assume $\xi$ satisfies~\ref{xi_power}, \eqref{xi_strongly_convex}, and
    $\bbeta$ satisfies \eqref{beta_decay1}, \eqref{beta_decay2}.
      Then for any $d\in\DD$ and $\theta\in\Theta_\Q$,
      when only the value of $\beta_\theta$ is varied,
      the function $\beta_{\theta}\mapsto \inf_{\pi\in \Pi_{d}}\PP_{\xi_{\bbeta}}(\pi)$ is differentiable on an open interval.
\end{proposition}

To see the utility of Proposition~\ref{prop:diff}, let us consider a Hamiltonian corresponding to the modified covariance function $\xi_{\bbeta}$ from \eqref{eq:def:xi:bbeta}.
Namely,
\eeq{ \label{HNbbeta_expanded}
H_{N,\xi_{\bbeta}}(\sigma) = H_{N,\xi}(\sigma) + \sum_{\theta\in\Theta_\Q}\beta_\theta H_{N,\theta}(\sigma),
}
where $H_{N,\xi}$ is as in \eqref{general_cov}, and each $H_{N,\theta}$ is an independent Gaussian process satisfying \eqref{eq:H:theta:cov}.
It is well-known \cite{panchenko08, chatterjee09, auffinger-chen18a} that the differentiability of the Parisi formula guarantees the concentration of $H_{N,\theta}(\sigma)$ under the Gibbs measure.
This concentration in turn leads to the G.G. identities in Definition~\ref{def:GG}, provided that $\beta_\theta\neq 0$ for each $\theta\in\Theta_\Q$.
In this way we will be able to obtain synchronization from Theorem~\ref{sync_thm} without any perturbation.
The tradeoff is that we need the model to have a generic covariance function $\xi_{\bbeta}$, instead of the original function $\xi$ whose symmetry \eqref{xi_symmetric} was meant to enable the crucial Lemma~\ref{lemma:symmetric}.
But as the next section explains, this symmetry can be preserved by selecting $\bbeta$ in a special way.

\subsection{Parisi formula for symmetric generic models} \label{subsec:symmetric:Parisi} 
Section~\ref{sec_overview_arrays} outlined how to obtain the desired lower bound \eqref{eq:goal} under two assumptions: 
\begin{enumerate}[label=\textup{(\roman*)}]

\item \label{bnru7_a} in the large $N$-large limit, the array of overlap matrices satisfies the G.G. identities;

\item \label{bnru7_b} each matrix in that array is invariant (in distribution) under permutation of its rows and columns.

\end{enumerate}
Section~\ref{subsec:generic} described how to secure assumption~\ref{bnru7_a} for the generic model \eqref{eq:def:xi:bbeta} with parameters $\bbeta=(\beta_\theta)_{\theta\in\Theta_\Q}$ satisfying $\beta_\theta\neq 0$ for every $\theta$.
Now we address assumption~\ref{bnru7_b} for the generic model.

Given a permutation $\pmu\in S_{\kappa}$ and a vector $w=(w(k))_{k=1}^\kappa\in\R^\kappa$, define the permuted vector $\pmu\act w\coloneqq(w(\pmu^{-1}(k)))_{k=1}^\kappa$.
We now extend this action to the parameter space $\Theta$:
\eeq{ \label{actthet}
\text{for }\theta &= (p,m,n_1,\ldots, n_m,w_1,\ldots, w_m)\in\Theta, \\
\text{define }\pmu\act \theta&= (p,m,n_1,\ldots, n_m,\pmu\act w_1,\ldots, \pmu\act w_m)\in \Theta.
}
We are then interested in $\bbeta$ satisfying the following symmetry condition:
\begin{equation} \tag{B3} \label{beta_symmetric}
\beta_{\theta}=\beta_{\pmu\act \theta} \quad \text{for every $\theta\in \Theta$ and $\pmu\in S_{\kappa}$}.
\end{equation}

\begin{lemma}\label{lem:symmbeta}
If $\xi$ satisfies \eqref{xi_symmetric} and $\bbeta$ satisfies \eqref{beta_symmetric},
then $\xi_{\bbeta}$ satisfies \eqref{xi_symmetric}.
\end{lemma}

\begin{proof}
For any given $\theta = (p,m,n_1,\ldots, n_m,w_1,\ldots, w_m)$, we have
\eq{
\xi_\theta(\pmu\act R) 
&\stackref{def:xi:theta,pmuR_def}{=} \prod_{j=1}^m\Big(\sum_{k,k'=1}^\kappa R_{\pmu^{-1}(k),\pmu^{-1}(k')}^pw_j(k)w_j({k'})\Big)^{n_j} \\
&\stackrefp{def:xi:theta,pmuR_def}{=} \prod_{j=1}^m\Big(\sum_{k,k'=1}^\kappa R_{k,k'}^pw_j(\pmu(k))w_j(\pmu(k'))\Big)^{n_j} 
\stackref{def:xi:theta,actthet}{=} \xi_{\pmu^{-1}\act\theta}(R).
}
The assumption \eqref{beta_symmetric} means $\beta_\theta = \beta_{\pmu^{-1}\act\theta}$ no matter the choice of $\pmu$, hence
\eq{
\sum_{\theta\in\Theta_\Q}\beta_\theta^2\xi_\theta(\pmu\act R) 
= \sum_{\theta\in\Theta_\Q}\beta_{\pmu^{-1}\act\theta}^2\xi_{\pmu^{-1}\act\theta}(R) = \sum_{\theta\in\Theta_\Q}\beta_\theta^2\xi_\theta(R).
}
Since $\xi(\pmu\act R)=\xi(R)$ by assumption, we have now shown that the same is true for the sum $\xi_{\bbeta} = \xi + \sum_{\theta\in\Theta_\Q}\beta_\theta^2\xi_\theta$.
\end{proof}

By implementing the strategy discussed after Proposition~\ref{prop:diff}, we will prove the following result in Section~\ref{sec_lower_bound}.
Note that $\PPP_{\xi_{\bbeta}}(\pi,0)$ is the Parisi functional \eqref{eq:Parisi:ftl:lagrange} without the Lagrange multiplier $\lambda$ (see Remark~\ref{rmk_lambda} for further discussion).

\begin{proposition}\label{prop:symmetric:parisi}
   Assume $\xi$ satisfies~\ref{xi_power}, \eqref{xi_strongly_convex}, \eqref{xi_symmetric}, and $\bbeta$ satisfies \eqref{beta_decay1}, \eqref{beta_decay2}, \eqref{beta_symmetric}.
   Provided $\beta_\theta\neq0$ for every $\theta\in\Theta_\Q$, there exists $\pi\in\Pi^\star$ such that
\eeq{ \label{prop:symmetric:parisi_eq}
\liminf_{N\to\infty} \eff_{\kappa N, \xi_{\bbeta}}(d_{\bal})
\geq \PPP_{\xi_{\bbeta}}(\pi,0).
}
\end{proposition}

From here the end goal \eqref{eq:goal} is not too much further.
There are just two missing items: 
\begin{enumerate}[label=\textup{\arabic*.},ref=\textup{\arabic*}]

\item \label{missing_piece_1} 
Construct $\bbeta$ that satisfies the hypotheses of Proposition~\ref{prop:symmetric:parisi}.

\item \label{missing_piece_2}
Justify sending $\bbeta\to0$.

\end{enumerate}
Item~\ref{missing_piece_1} is taken care of by the following lemma.
Recall the constant $C_\theta$ defined in \eqref{eq:xi:theta:regularity}.

\begin{lemma}\label{lem:symm:inverse}
For any $\eps>0$, there exists $\bbeta=(\beta_\theta)_{\theta\in\Theta_\Q}$ satisfying \eqref{beta_decay1}, \eqref{beta_symmetric}, $\beta_\theta\neq0$ for every $\theta\in\Theta_\Q$, and
\eeq{ \label{cn73}
    \sum_{\theta\in \Theta_\Q}\beta_{\theta}^2(C_{\theta}\vee 1)\le \eps.
}
\end{lemma}
\begin{proof}
The operation $(\pmu,\theta)\mapsto\pmu\act\theta$ defined in \eqref{actthet} is a group action by the symmetric group $S_\kappa$ on $\Theta_\Q$.
Since $\Theta_{\Q}$ is countable, we can enumerate the orbits of $\Theta_{\Q}$ under this action: $\OO_1$, $\OO_2$, and so on.
Both the constant $C_\theta$ in \eqref{eq:xi:theta:regularity} and the degree in \eqref{deg_def} are invariant
under the action, and so we can write $C_i=C_{\theta}$ and $\deg(\OO_i)= \deg(\theta)$ for any $\theta\in \OO_i$.
Given $\eps>0$, define
\eeq{ \label{djc82}
\beta_i \coloneqq \sqrt{\frac{1}{2^i(1+\eps)^{\deg(\OO_i)}}}\wedge \sqrt{\frac{\eps}{2^ik!(C_i\vee1)}} > 0,
}
and then set $\beta_\theta = \beta_i$ for all $\theta\in\OO_i$.
The collection $\bbeta = (\beta_\theta)_{\theta\in\Theta_\Q}$ satisfies \eqref{beta_symmetric} by construction.
Inequality \eqref{beta_decay1} holds because of the first term on the right-hand side of \eqref{djc82}, while \eqref{cn73} holds because of the second term and the fact $|\OO_i|\leq\kappa!$.
\end{proof}

Meanwhile, item~\ref{missing_piece_2} will be possible thanks to the following proposition, which shows that both sides of \eqref{prop:symmetric:parisi_eq} are continuous with respect to the covariance function $\xi$.

\begin{proposition}\label{prop:continuity}
      Assume $\xi$ and $\tilde\xi$ satisfy~\ref{xi_power}.
      Then the following statements hold.
      \begin{enumerate}[label=\textup{(\alph*)}]

\item \label{prop:continuity_a}
For any $d\in \DD_N$, we have
\begin{equation*}
\big|\eff_{N,\xi}(d)-\eff_{N,\tilde\xi}(d)\big|\leq \sup_{\|R\|_1\leq 1} \big|\xi(R)-\tilde\xi(R)\big|.
\end{equation*}
      
    \item \label{prop:continuity_b}
    For any $d\in \DD$ and any $\pi,\tilde\pi\in \Pi_{d}$, we have 
\eeq{ \label{dje4}
        |\PP_{\xi}(\pi)-\PP_{\tilde\xi}(\tilde\pi)|
        &\leq 2\sup_{\|R\|_1\leq 1}\big\|\nabla \xi(R)-\nabla \tilde\xi(R)\big\|_{\infty} \\
        &\phantom{\leq}+\sup_{\|R\|_1\leq 1}\big\|\nabla^2\xi(R)\big\|_{\infty} \int_{0}^{1}\|\pi(t)-\tilde\pi(t)\|_1\ \de t.
}

\end{enumerate}
\end{proposition}

Proposition~\ref{prop:continuity} is proved in Section~\ref{sec:continuity}.

\subsection{Proofs of main results} \label{sec:proof_main} Finally, we prove our main theorems.
\begin{proof}[Proof of Theorem~\ref{thm:main}]
Consider any $d^N\in\DD_N$ such that $d^N\to d_\bal$ as $N\to\infty$.
We wish to establish \eqref{thm:main_eq}.
Since $\Pi^\star$ is a subset of $\Pi_{d_\bal}$, Theorem~\ref{gen_con_par_thm} provides an upper bound:
\eeq{ \label{gdi10}
\lim_{N\to\infty} \eff_{N,\xi}(d^N)
=\lim_{\eps\searrow 0}\limsup_{N\to\infty}\eff_{N,\xi}(d_{\bal},\eps)
=\lim_{\eps\searrow 0}\liminf_{N\to\infty}\eff_{N,\xi}(d_{\bal},\eps)
&=\inf_{\pi\in\Pi_{d_\bal}}\PP_\xi(\pi) \\
&\leq\inf_{\pi\in\Pi^\star}\PP_\xi(\pi).
}
Since the first line of \eqref{gdi10} is entirely equalities, the matching lower bound only needs to be established along a subsequence.
More specifically, we will prove \eqref{eq:goal}.

The covariance function $\xi$ is assumed to satisfy weak convexity~\ref{xi_convex}.
To upgrade to strong convexity, we add a small multiple of the usual Potts Hamiltonian:
for $\eps>0$, define
\eeq{ \label{def_xi_delta}
\xi^{\eps}(R)\coloneqq\xi(R)+\eps\cdot\tr(R^{\sT}R)/2.
}
Let us note three properties of $\xi^\eps$.
First, by Remark~\ref{rmk_Potts_case}, $\xi^\eps$ satisfies~\ref{xi_power} provided $\xi$ satisfies~\ref{xi_power}.
Second, by assumption~\ref{xi_convex} for $\xi$, the new function $\xi^\eps$ satisfies \eqref{xi_strongly_convex} with
\eeq{ \label{dxip0}
\lambda_{\min}(\nabla^2\xi^\eps(R))\geq\eps \quad \text{for every $R\in\nonpsd$}.
}
Third and finally, since $\xi$ is assumed to satisfy the symmetry condition \eqref{xi_symmetric} and the map $R\mapsto\tr(R^\sT R)$ does as well (the expression in \eqref{Potts_xi} is clearly invariant under any permutation of the row and column indices), the function $\xi^\eps$ satisfies \eqref{xi_symmetric}.
Therefore, Lemma~\ref{lem:symm:inverse} guarantees the existence of $\bbeta=(\beta_\theta)_{\theta\in\Theta_\Q}$ such that \eqref{beta_decay1} and \eqref{beta_symmetric} hold, $\beta_\theta\neq0$ for every $\theta\in\Theta_\Q$, and
\eeq{ \label{dxip1}
\sum_{\theta\in \Theta_\Q}\beta_{\theta}^2(C_{\theta}\vee 1)\leq \eps.
}
Applying the modification \eqref{eq:def:xi:bbeta} to $\xi^\eps$ in place of $\xi$, we obtain $\xi^\eps_{\bbeta}=\xi^\eps+\sum_{\theta\in\Theta_\Q}\beta_\theta^2\xi_\theta$.
The combination of \eqref{dxip0} and \eqref{dxip1} shows that \eqref{beta_decay2} is satisfied with $\xi^\eps$ in place of $\xi$.
We can thus apply Proposition~\ref{prop:symmetric:parisi} to determine that
\begin{equation}\label{eq:goal:perturbed}
\liminf_{N\to\infty} \eff_{\kappa N, \xi_{\bbeta}^{\eps}}(d_{\bal})
\geq \PPP_{\xi^\eps_{\bbeta}}(\pi,0)
\stackref{def:parisi:ftl}{\geq} \inf_{\pi\in \Pi^\star}\PP_{\xi_{\bbeta}^{\eps}}(\pi).
\end{equation}
The remainder of the proof is to justify sending $\eps$ and $\bbeta$ to zero.

We first address the left-hand side of \eqref{eq:goal:perturbed}. 
By Proposition~\hyperref[prop:continuity_a]{\ref*{prop:continuity}\ref*{prop:continuity_a}}, we have 
\eq{
&\big|\eff_{\kappa N, \xi}(d_{\bal})-\eff_{\kappa N, \xi^{\eps}_{\bbeta}}(d_{\bal})\big|
\leq \sup_{\|R\|_1\leq 1} \big|\xi(R)-\xi^{\eps}_{\bbeta}(R)\big|.
}
Whenever $\|R\|_1\leq1$, we have
\eq{
\big|\xi(R)-\xi^{\eps}_{\bbeta}(R)\big|
&\stackrefp{def_xi_delta,eq:def:xi:bbeta,rmk_xi_1}{\leq} \big|\xi(R)-\xi^{\eps}(R)\big| + \big|\xi^\eps(R)-\xi^{\eps}_{\bbeta}(R)\big| \\
&\stackref{def_xi_delta,eq:def:xi:bbeta,rmk_xi_1}{\leq}
\frac{\eps}{2}+\sum_{\theta\in \Theta_\Q}\beta_{\theta}^2
\stackref{dxip1}{\leq}\frac{3\eps}{2}.
}
The two previous displays together show
\eeq{ \label{ghfs7}
\big|\eff_{\kappa N, \xi}(d_{\bal})-\eff_{\kappa N, \xi^{\eps}_{\bbeta}}(d_{\bal})\big| \leq 3\eps/2.
}
Meanwhile, regarding the right-hand side of \eqref{eq:goal:perturbed},
Proposition~\hyperref[prop:continuity_b]{\ref*{prop:continuity}\ref*{prop:continuity_b}} gives
\eq{ 
\Big|\inf_{\pi\in \Pi^\star}\PP_{\xi}(\pi)-\inf_{\pi\in \Pi^\star}\PP_{\xi^{\eps}_{\bbeta}}(\pi)\Big|
&\leq 2\sup_{\|R\|_1\leq 1}\big\|\nabla \xi(R)-\nabla \xi^{\eps}_{\bbeta}(R)\big\|_{\infty}.
}
Whenever $\|R\|_1\leq 1$, another application of the triangle inequality leads to
\eq{
\big\|\nabla \xi(R)-\nabla \xi^{\eps}_{\bbeta}(R)\big\|_{\infty}
&\stackrefp{def_xi_delta,eq:xi:theta:regularity}{\leq} \big\|\nabla \xi(R)-\nabla \xi^{\eps}(R)\big\|_{\infty}
+ \big\|\nabla \xi^\eps(R)-\nabla \xi^{\eps}_{\bbeta}(R)\big\|_{\infty} \\
&\stackref{def_xi_delta,eq:xi:theta:regularity}{\leq} \eps +\sum_{\theta\in \Theta_\Q}\beta_{\theta}^2C_\theta
\stackref{dxip1}{\leq}2\eps.
}
Putting the two previous displays together, we have
\eeq{ \label{ghfs8}
\Big|\inf_{\pi\in \Pi^\star}\PP_{\xi}(\pi)-\inf_{\pi\in \Pi^\star}\PP_{\xi^{\eps}_{\bbeta}}(\pi)\Big| \leq 4\eps.
}
Finally, the combination of \eqref{eq:goal:perturbed}, \eqref{ghfs7}, and \eqref{ghfs8} yields
\eq{
\liminf_{N\to\infty} \eff_{\kappa N, \xi}(d_{\bal})\geq \inf_{\pi\in \Pi^\star}\PP_{\xi}(\pi) - 6\eps.
}
Letting $\eps\searrow0$ completes the proof.
\end{proof}

\begin{proof}[Proof of Theorem~\ref{thm:main:Potts}]
This is immediate from Theorem~\ref{thm:main} since the associated covariance function $\xi(R)=\beta^2\tr(R^{\sT}R)$ from \eqref{Potts_xi} satisfies~\ref{xi_power},~\ref{xi_convex}, \eqref{xi_symmetric}.
\end{proof}


\section{Parisi functional preliminaries} \label{sec_preliminaries}

\subsection{Prelimit of the Parisi functional} \label{prelimit_section}
In this section we define the functional $\Psi_\xi$ that appears in \eqref{fnc73}.
More precisely, there is a family of functionals $(\Psi_{\xi,N})_{N\geq1}$ that arise naturally in the A.S.S. scheme (Proposition~\ref{ass_prop}).
These can be thought of as prelimiting versions of the Parisi functional $\PP_\xi$.
Understanding their behavior as $N\to\infty$ is an important step in proving that the Parisi formula is a lower bound for the limiting free energy.

Our definition requires three steps: 
\begin{enumerate}[label=\textup{\arabic*.}]

\item Give a abstract notion of an overlap map $R$ which generalizes \eqref{def:overlap}.

\item Using this map and a random measure from which to sample, generate a random array of overlap matrices.

\item Define a functional $\LL\mapsto\Psi_{\xi,N}(\LL)$, where $\LL$ is the law of the random array.

\end{enumerate}

\subsubsection{Allowable overlap maps}
Recall from \eqref{def_nonpsd} that $\nonpsd$ denotes the set of $\kappa\times\kappa$ matrices whose entries are nonnegative and have sum at most 1.
Fix $d\in\DD$ and a positive integer $N$.

\begin{assumption} \label{map_assumption}
$(\XX,\FF)$ is a measurable space, and $R\colon\XX\times\XX\to\nonpsd$ is a map with the following properties.
\begin{enumerate}[label=\textup{(\roman*)}]

\item \label{map_assumption_i}
$R$ is measurable.

\item \label{map_assumption_ii}
For every $x\in\XX$, $R(x,x)=\diag(d)$.

\item \label{map_assumption_iii}
There exist centered Gaussian processes $Z\colon\XX\to\R^\kappa$ and $Y\colon\XX\to\R$ with covariances
\eeq{ \label{A3_cov}
\begin{aligned}
\E[Z(x)Z(x')^\sT] &= 
\nabla\xi(R(x,x')) \\
\E[Y(x)Y(x')] &= 
\vartheta_\xi(R(x,x'))
\end{aligned}
\qquad \text{for $x,x'\in\XX$}.
}
Furthermore, these processes are almost surely measurable functions on $\XX$.
\end{enumerate}
\end{assumption}

\subsubsection{The overlap array} \label{prelimit_section_3}

Given a random probability measure $\GG$ on $(\XX,\FF)$ which is independent of the processes from \eqref{A3_cov}, let
$(x^\ell)_{\ell\geq1}$ be i.i.d.~samples from $\GG$.
Apply $R$ to each pair of samples, and set
\eq{ 
\RR_{\ell,\ell'} = R(x^\ell,x^{\ell'})
    +\one_{\{\ell=\ell'\}}(\diag(d)-R(x^\ell,x^{\ell'})).
}
This defines a random array $\RR = (\RR_{\ell,\ell'})_{\ell,\ell'\geq1}$.
Denote the law of $\RR$ under $\E(\GG^{\otimes\infty})$ by $\mathsf{Law}(\RR;\E(\GG^{\otimes\infty}))$, where the dependence on the map $R$ is implicit.

\subsubsection{The functional} \label{prefnc_sec}
Given a nonempty subset $S\subseteq\Sigma^N$, we now define a functional $\LL\mapsto\Psi_{N,\xi}(\LL;S)$ that 
accepts as input any law $\LL = \mathsf{Law}(\RR;\GG)$ realized as above.

Let $Z_1,\dots,Z_N$ be independent copies of the process $Z$ from \eqref{A3_cov}.
Let $Y$ be as in \eqref{A3_cov}.
%
Let $\langle\cdot\rangle_\GG$ denote expectation with respect to $\GG$. 
Finally, let $\E(\cdot)$ denote expectation over both realizations of $\GG$ and the Gaussian processes from Assumption~\ref{map_assumption}.
Now define
\begin{subequations}
\label{defining_Psi_parts}
\begin{align}
\label{psi_1_def}
\Psi_{N,\xi}^{(1)}(\LL;d,S) &\coloneqq \frac{1}{N}\E\log\sum_{\sigma\in S}
\Big\langle\exp\Big(\sum_{i=1}^N \iprod[\big]{Z_i(x)}{\sigma_i}\Big)\Big\rangle_\GG, \\
\label{psi_2_def}
\Psi_{N,\xi}^{(2)}(\LL;d) &\coloneqq \frac{1}{N}\E\log\big\langle \exp\big(\sqrt{N}\, Y(x)\big)\big\rangle_\GG,
\end{align}
\end{subequations}
and then the functional of interest is given by
\eeq{ \label{Psi_def}
\Psi_{N,\xi}(\LL;S) \coloneqq \Psi_{N,\xi}^{(1)}(\LL;S)-\Psi_{N,\xi}^{(2)}(\LL).
}

Especially important is the fact that $\LL\mapsto\Psi_{N,\xi}(\LL;S)$ is continuous.
To state this precisely, we need to introduce some notation.
Given any law $\LL=\Law(\RR;\GG)$ realized as above, 
let $\LL^{(n)}$ denote the law of the finite subarray $\RR^{(n)} = ( \RR_{\ell,\ell'})_{1\leq \ell,\ell'\leq n}$; this is a Borel probability measure on $n\times n$ arrays whose elements belong to $\nonpsd$.
Let $\Prob((\nonpsd)^{n\times n})$ denote the set of all Borel probability measures on this space.
Because $\nonpsd$ is compact, one can metrize the topology of weak convergence on $\Prob((\nonpsd)^{n \times n})$ by, say, a Wasserstein distance with respect to the norm
$\|\RR^{(n)}\| \coloneqq \sum_{\ell,\ell'=1}^n \|\RR_{\ell,\ell'}\|_1$.
We can thus speak of continuity with respect to weak convergence.

\begin{proposition} \label{uniform_continuity}
Assume $\xi$ satisfies~\ref{xi_power}.
For any $\eps>0$, there is $n$ large enough and some continuous function $\phi_N\colon \Prob((\nonpsd)^{n\times n})\to\R$ such that
\eq{
|\Psi_{N,\xi}(\LL;S)-\phi_N(\LL^{(n)})| \leq \eps \quad \text{whenever $\Psi_{N,\xi}(\LL;S)$ is defined}.
}
\end{proposition}

It follows from this proposition that $\LL\mapsto\Psi_{N,\xi}(\LL;S)$ is continuous.
In the statement below, weak convergence is understood in the product sense: a sequence $(\LL_m)_{m\geq1}$ converges weakly to $\LL$ if for every $n$, $(\LL_m^{(n)})_{m\geq1}$ converges weakly to $\LL^{(n)}$.

\begin{corollary} \label{extension_cor}
Assume $\xi$ satisfies~\ref{xi_power}.
If $(\LL_j)_{j\geq1}$ is a weakly convergent sequence of laws such that $\Psi_{N,\xi}(\LL_j;S)$ is defined for every $j$,
then $\lim_{j\to\infty}\Psi_{N,\xi}(\LL_j;S)$ exists and depends only on the limit of $(\LL_j)_{j\geq1}$.  
\end{corollary}

A proof of Corollary~\ref{extension_cor} can be found in \cite[Cor.~2.7]{bates-sohn22a}.
Proposition~\ref{uniform_continuity} follows from a standard truncation argument that relies only on the compactness of $\nonpsd$ and the continuity of $\nabla\xi$ and $\vartheta$.
Examples of this argument can be found in \cite[Lem.~3]{panchenko14}, \cite[Thm.~1.3]{panchenko13a}, and \cite[Prop.~2.6]{bates-sohn22a}.
Therefore, we omit the proof.

\subsection{Ruelle probability cascades} \label{subsec:prelimit}
In this section we restrict the functional $\Psi_{N,\xi}$ from \eqref{Psi_def} to overlap arrays generated by a Ruelle probability cascade (RPC).
The RPCs are a certain class of random probability measures with ultrametric support.
Each RPC is associated to a 
partition of the unit interval:
\eeq{ \label{eq:m2}
0 = m_0 < m_1 < \cdots < m_{s-1} < m_s = 1.
}
For concreteness, we take the support of the RPC to be the countable set $\N^{s-1}$, which can be regarded as the leaf set of a tree of depth $s-1$:
\begin{itemize}
\item The root vertex is the empty sequence $\varnothing$, and by convention $\N^0 = \{\varnothing\}$.
\item Each $\alpha=(\alpha_1,\dots,\alpha_{r})\in\N^{r}$ is said to have depth $|\alpha| = r$ and has children $\{(\alpha_1,\dots,\alpha_{r},a):\, a\in\N\}$ at depth $r+1$.
\item The path to a given leaf $\alpha=(\alpha_1,\dots,\alpha_{s-1})\in\N^{s-1}$ is the set
\eq{
\mathfrak{p}(\alpha) = \{\varnothing,\alpha_1,(\alpha_1,\alpha_2),\dots,(\alpha_1,\dots,\alpha_{s-1})\}.
}
\end{itemize}
Given two leaves $\alpha = (\alpha_1,\dots,\alpha_{s-1})$ and $\alpha' = (\alpha'_1,\dots,\alpha'_{s-1})$ in $\N^{s-1}$, let $r(\alpha,\alpha')$ denote the smallest depth at which their ancestral paths disagree:
\eq{ 
r(\alpha,\alpha') \coloneqq
\begin{cases} \inf\{r:\,  \alpha_r \neq \alpha_r'\} &\text{if $\alpha\neq\alpha'$} \\
s &\text{if $\alpha=\alpha'$}.
\end{cases}
}
The RPCs are characterized by the following fact.

\begin{theirthm} {\textup{\cite[Thm.~15.2.1]{talagrand06b}}} \label{rpcthm}
For any sequence of the form \eqref{eq:m2}, there is a random measure $\nu$ on $\N^{s-1}$ with the following property.
Given any sequence $0 \leq q_1 \leq \cdots \leq q_s = 1$, define the measure $\mu = \sum_{r=1}^s (m_r-m_{r-1})\delta_{q_r}$.
If $(\alpha^1,\alpha^2,\dots)\sim\E(\nu^{\otimes\infty})$, then the law of the array $(q_{r(\alpha^\ell,\alpha^{\ell'})})_{\ell,\ell'\geq1}$ is equal to $\bar\LL_\mu$ from Theorem~\ref{representation_thm}.
\end{theirthm}




The measure $\nu$ is called the RPC associated to \eqref{eq:m2}, and we denote its probability mass function by $(\nu_\alpha)_{\alpha\in\N^{s-1}}$.
We will now use RPCs to give a different formulation of the Parisi functional from Section~\ref{subsec:Parisi}.

Let $\xi$ satisfy~\ref{xi_power}.
Consider any $d\in\DD$ and any discrete path $\pi\in\Pi_{d}^\disc$, meaning there is a partition of the form \eqref{eq:m2}
and a sequence in $\Gamma_\kappa(d)$,
\begin{subequations} \label{pi_rep}
\eeq{ \label{mat_seq}
0 \preceq \gamma_1 \preceq \cdots \preceq \gamma_{s-1} \preceq \gamma_{s} = \mathrm{diag}(d),
}
such that
\eeq{ \label{pi_rep_b}
\pi(t) = \gamma_r \quad \text{for $t\in(m_{r-1},m_{r}]$, $r\in\{1,\dots,s\}$}.
}
\end{subequations}
Let $Z_1,\dots,Z_N\colon\N^{s-1}\to\R^\kappa$ and $Y\colon\N^{s-1}\to\R$ be independent centered Gaussian processes with covariances
\begin{equation} \label{bni}
\begin{aligned}
\E[Z_i(\alpha)Z_i(\alpha')^\sT] &= 
\nabla\xi(\gamma_{r(\alpha,\alpha')}) \\
\E[Y(\alpha)Y(\alpha')] &= 
\vartheta_\xi(\gamma_{r(\alpha,\alpha')})
\end{aligned}
\qquad \text{for $\alpha,\alpha'\in\N^{s-1}$}.
\end{equation}
We assume these processes are also independent of the weights $(\nu_\alpha)_{\alpha\in\N^{s-1}}$;
their existence is addressed in Remark~\ref{rmk_RPCexist}.
Now define, for $\lambda\in\R^\kappa$ and nonempty $S\subseteq\Sigma^N$, the quantity 
\eeq{
\PPP_{N,\xi}^{(1)}(\pi,\lambda;S) \label{pre_par_a}
&\coloneqq \frac{1}{N}\E\log\sum_{\alpha\in\N^{s-1}}\sum_{\sigma\in S}\nu_\alpha\exp\Big(\sum_{i=1}^N\iprod[\big]{Z_{i}(\alpha)+\lambda}{\sigma_i}\Big).
}
%
%
%
%
Also recall the following quantity from \eqref{PPP2_def}:
\eeq{
\PPP_{\xi}^{(2)}(\pi) \label{pre_par_b}
&= \frac{1}{2}\int_0^1\vartheta_\xi(\pi(t))\ \dd t - \frac{1}{2}\vartheta_\xi(\diag(d)), \quad \pi\in\Pi_d.
}
Apart from the Lagrange multiplier $\lambda\in\R^\kappa$ in \eqref{pre_par_a}, the functionals $\PPP_{N,\xi}^{(1)}$ and $\PPP_{N,\xi}^{(2)}$ are simply $\Psi_{N,\xi}^{(1)}$ and $\Psi_{N,\xi}^{(2)}$ from \eqref{defining_Psi_parts} applied to the RPC setting.
This is captured by the following lemma, which is proved in Appendix~\ref{sec_app_rpc}.

\begin{lemma} \label{PPPPsi_lemma}
Assume $\xi$ satisfies~\ref{xi_power}.
Let $\nu$ be the RPC associated to \eqref{eq:m2}.
With $\pi\in\Pi_d^\disc$ given by \eqref{pi_rep}, let $\LL$ denote the law of the array $(\gamma_{r(\alpha^\ell,\alpha^{\ell'})})_{\ell,\ell'\geq1}$ under $\E(\nu^{\otimes\infty})$. 
Then
\begin{subequations} \label{PPPPsi}
\begin{align}
\PPP_{N,\xi}^{(1)}(\pi,0;S) \label{PPPPsi_a}
&= \Psi_{N,\xi}^{(1)}(\LL; S), \\
\PPP_{\xi}^{(2)}(\pi) \label{PPPPsi_b}
&= -\Psi_{N,\xi}^{(2)}(\LL).
\end{align}
\end{subequations}
In particular, $\PPP_{N,\xi}^{(1)}(\pi,0;S)+\PPP_{\xi}^{(2)}(\pi) = \Psi_{N,\xi}(\LL;S)$.
\end{lemma}

\begin{remark} \label{rmk_wd}
Because \eqref{mat_seq} allows for equalities, the representation \eqref{pi_rep_b} of the path $\pi$ is not unique.
For instance, one could insert duplicate copies of some $\gamma_r$ and then refine the partition \eqref{eq:m2}.
Consequently, it is not immediately clear that $\PPP_{N,\xi}^{(1)}$ is well-defined, but this actually follows from Lemma~\ref{PPPPsi_lemma}.
Indeed, no matter the representation of $\pi$, the following measure remains the same:
\eq{
\mu = \sum_{r=1}^s(m_r-m_{r-1})\delta_{q_r}, \quad \text{where} \quad
q_r = \sup\{t\in(0,1]:\, \pi(t)=\gamma_r\}.
}
Under this definition of $q_r$, we always have $\pi(q_r) = \gamma_r$ by left-continuity of $\pi$, and so the array $(\gamma_{r(\alpha^\ell,\alpha^{\ell'})})_{\ell,\ell'\geq1}$ in Lemma~\ref{PPPPsi_lemma} is equal to $(\pi(q_{r(\alpha^\ell,\alpha^{\ell'})}))_{\ell,\ell'\geq1}$. 
Since the law of $(q_{r(\alpha^\ell,\alpha^{\ell'})})_{\ell,\ell'\geq1}$ is equal to $\bar\LL_\mu$ by Theorem~\ref{rpcthm}, it follows that the law $\LL$ in Lemma~\ref{PPPPsi_lemma} does not depend on the representation of $\pi$.
\end{remark}

Notice that the left-hand side of \eqref{PPPPsi_b} does not depend on $N$.
Our next lemma shows the same is true of \eqref{PPPPsi_a} when $S$ is the entire product set $\Sigma^N$.
More importantly, the functional $\PPP_{\xi}^{(1)}$ from \eqref{eq:def:Psi} is recovered.
Once again, the proof is postponed to Appendix~\ref{sec_app_rpc}.

\begin{lemma} \label{lem_parisi_recovered}
Assume $\xi$ satisfies~\ref{xi_power}.
For any $N\geq1$, $\pi\in\Pi_{d}^\disc$, and $\lambda\in\R^\kappa$, we have
\eeq{ \label{qk38}
\PPP_{N,\xi}^{(1)}(\pi,\lambda;\Sigma^N) = \PPP_{\xi}^{(1)}(\pi,\lambda).
}
\end{lemma}

\begin{remark} \label{rmk_dual}
It would appear from Lemmas~\ref{PPPPsi_lemma} and~\ref{lem_parisi_recovered} that the functional $\Psi_{N,\xi}$ from \eqref{Psi_def} is easily related to $\PPP_\xi$ from \eqref{eq:Parisi:ftl:lagrange}.
The complication is that the A.S.S. scheme presented in Section~\ref{ass_sec} will require we evaluate $\LL\mapsto\Psi_{N,\xi}(\LL;S)$ using the constrained set $S=\Sigma^N(d^N)$ for some $d^N\in\DD_N$.
In contrast, the identity \eqref{qk38} only holds in unconstrained case $S=\Sigma^N$.
This gap can only be bridged once $N\to\infty$, for then the constraint $d^N$ can be replaced by an optimization over the dual variable $\lambda$; see Lemma~\ref{lem:duality}.
The tradeoff is that an additional optimization must be performed over $\lambda$, as in \eqref{def:parisi:ftl}.
\end{remark}

\begin{remark} \label{rmk_RPCexist}
In order for the quantities in \eqref{PPPPsi} to be defined, we need the existence of processes $Z_1,\dots,Z_N$, $Y$ satisfying \eqref{bni}.
Fortunately, it is easy to construct them.
Indeed, let $(\eta_{i,\alpha}:\, i\geq1, \alpha\in\N^0\cup\cdots\cup\N^{s-1})$ be i.i.d.~standard normal random vectors in $\R^\kappa$, and then define
\eeq{ \label{Zialpha}
Z_{i}(\alpha) = \sum_{r=0}^{s-1}\sqrt{\nabla\xi(\gamma_{r+1})-\nabla\xi(\gamma_{r})\one\{r>0\}}\, \eta_{i,(\alpha_1,\dots,\alpha_{r})}.
}
Here we are using \eqref{monotonicity} to ensure that the matrix inside the square root is positive-semidefinite and thus admits a square root.
Since $(\alpha_1,\dots,\alpha_r) = (\alpha'_1,\dots,\alpha'_r)$ if and only if $r<r(\alpha,\alpha')$, we have
\eq{
\E[Z_{i}(\alpha)Z_{i}(\alpha')^{\sT}]
= \sum_{r=0}^{r(\alpha,\alpha')-1}[\nabla\xi(\gamma_{r+1})-\nabla\xi(\gamma_r)\one\{r>0\}] = \nabla\xi(\gamma_{r(\alpha,\alpha')}).
}
Similarly, let $(\eta_{\alpha}:\alpha\in\N^0\cup\cdots\cup\N^{s-1})$ be i.i.d.~standard normal random variables, and then
\eeq{ \label{xm334}
Y(\alpha) = \sum_{r=0}^{s-1}\sqrt{\vartheta_\xi(\gamma_{r+1})-\vartheta_\xi(\gamma_{r})}\, \eta_{(\alpha_1,\dots,\alpha_{r})}.
}
The same kind of telescoping calculation as above yields $\E[Y(\alpha)Y(\alpha')]=\vartheta_\xi(\gamma_{r(\alpha,\alpha')})$.
\end{remark}

\section{Differentiability of the Parisi formula}
\label{sec:diff}
In this section, we prove Proposition~\ref{prop:diff}. 
The following lemma provides the key estimate.
Recall the functional $\PPP_{\xi}$ from \eqref{eq:Parisi:ftl:lagrange}.

\begin{lemma}\label{lem:second:derivative}
Assume $\xi$ satisfies~\ref{xi_power} and $\bbeta=(\beta_\theta)_{\theta\in\Theta_\Q}$ satisfies \eqref{beta_decay1}.
    For any $\theta\in \Theta_\Q$, there exists a constant $C$ depending only on $\theta$ and $\ka$, such that
    \eq{
      \Big|\Big(\frac{\partial}{\partial\beta_{\theta}}\Big)^2\PPP_{\xi_{\bbeta}}(\pi,\la)\Big|\leq C(1+\beta_{\theta}^2) \quad \text{for any $\pi\in \Pi_{d}^{\disc}$, $d\in\DD$, $\lambda\in\R^\kappa$}.
    }
\end{lemma}


\begin{proof}
During the proof, the value of $C$ may change from line to line, but it will only depend on $\theta$ and $\kappa$.
Assume $\pi$ is given by \eqref{eq:pi:discrete}.
Since $\PPP_{\xi_{\bbeta}}(\pi,\lambda)=\PPP_{\xi_{\bbeta}}^{(1)}(\pi,\lambda)+\PPP_{\xi_{\bbeta}}^{(2)}(\pi)$, it suffices to prove the desired differential inequality for each term on the right-hand side.
We begin with the second term.

Recalling the definition of $\PPP_{\xi_{\bbeta}}^{(2)}$ from \eqref{PPP2_def}, we have
\eq{
\PPP_{\xi_{\bbeta}}^{(2)}(\pi) = \frac{1}{2}\sum_{r=1}^s (m_r-m_{r-1})\vartheta_{\xi_{\bbeta}}(\gamma_r) - \frac{1}{2}\vartheta_{\xi_{\bbeta}}(\diag(d)).
}
The right-hand side can be rewritten using summation by parts.
Since $m_0=0$, $m_s=1$, and $\gamma_s = \diag(d)$, this results in
\eq{
\PPP_{\xi_{\bbeta}}^{(2)}(\pi) = -\frac{1}{2}\sum_{r=2}^{s} m_{r-1}[\vartheta_{\xi_{\bbeta}}(\gamma_{r})-\vartheta_{\xi_{\bbeta}}(\gamma_{r-1})].
}
Now recall from  \eqref{vthet_def} that $\vartheta_{\xi_{\bbeta}}(R)=\iprod{R}{\nabla \xi_{\bbeta}(R)} -\xi_{\bbeta}(R)$, where $\xi_{\bbeta}=\xi+\sum_{\theta\in\Theta_\Q}\beta_{\theta}^2\xi_{\theta}$ as in \eqref{eq:def:xi:bbeta}.
We thus have
\eq{
\Big(\frac{\partial}{\partial\beta_\theta}\Big)^2\vartheta_{\xi_{\bbeta}}(R)
= 2\iprod{R}{\nabla\xi_\theta(R)}-2\xi_\theta(R)
\stackref{vartheta_theta_def}{=} 2\vartheta_\theta(R)
\stackrel{\footnotesize \text{Prop.~\hyperref[prop:xi:theta_a]{\ref*{prop:xi:theta}\ref*{prop:xi:theta_a}}}}{=} 2(\deg(\theta)-1)\xi_\theta(R).
}
Recall from Proposition~\hyperref[prop:xi:theta_b]{\ref*{prop:xi:theta}\ref*{prop:xi:theta_b}} that $\xi_\theta(\gamma_r)\geq\xi_\theta(\gamma_{r-1})\geq0$ for each $r$.
Therefore, the two previous displays together yield
\eeq{ \label{dxk4}
\Big|\Big(\frac{\partial}{\partial\beta_\theta}\Big)^2\PPP_{\xi_{\bbeta}}^{(2)}(\pi)\Big|
&\leq (\deg(\theta)-1)\sum_{r=2}^s [\xi_\theta(\gamma_r)-\xi_\theta(\gamma_{r-1})] \\
&\leq (\deg(\theta)-1)\xi_\theta(\gamma_s)
\leq (\deg(\theta)-1)\sup_{\|R\|_1\leq1}|\xi_\theta(R)| \leq C.
}

Next we consider the quantity $\PPP_{\xi_{\bbeta}}^{(1)}(\pi,\lambda)$ from \eqref{eq:def:Psi}.
By Lemma~\ref{lem_parisi_recovered} with $N=1$,
\eeq{ \label{ghw8ds}
\PPP_{\xi_{\bbeta}}^{(1)}(\pi,\lambda)
= \E \log\sum_{\alpha\in \N^{s-1}}\sum_{\sigma\in \Sigma}\nu_{\alpha}\exp\iprod[\big]{Z_{\xi_{\bbeta}}(\alpha)+\lambda}{\sigma},
}
where $(\nu_{\alpha})_{\alpha\in \N^{s-1}}$ is the probability mass function of the Ruelle probability cascade associated to \eqref{eq:m2},
and $Z_{\xi_{\bbeta}}\colon\N^{s-1}\to\R^\kappa$ is an independent centered Gaussian process with covariance
\eq{
\E[Z_{\xi_{\bbeta}}(\alpha)Z_{\xi_{\bbeta}}(\alpha')^\sT] = \nabla\xi_{\bbeta}(\gamma_{r(\alpha,\alpha')}).
}
We can realize this process as $Z_{\xi_{\bbeta}}(\alpha) = Z_{\xi}(\alpha) + \sum_{\theta\in\Theta_\Q}\beta_\theta Z_\theta(\alpha)$, where all the processes on the right-hand side are independent and centered, with covariances
\eeq{  \label{dzo2}
\E[Z_{\xi}(\alpha)Z_{\xi}(\alpha')^\sT] = \nabla\xi(\gamma_{r(\alpha,\alpha')}), \qquad
\E[Z_{\theta}(\alpha)Z_{\theta}(\alpha')^\sT] = \nabla\xi_\theta(\gamma_{r(\alpha,\alpha')}).
}
These processes exist by Remark~\ref{rmk_RPCexist}.
Then consider the following probability measure on $\N^{s-1}\times\Sigma$:
\eq{ 
    G(\alpha,\sigma)\propto \nu_{\alpha}\exp\iprod[\big]{Z_{\xi_{\bbeta}}(\alpha)+\lambda}{\sigma}.
}
In what follows, we will write $\big((\alpha^\ell,\sigma^\ell)\big)_{\ell\geq1}$ to denote i.i.d.~samples from $G$,
and $\langle\cdot\rangle_G$ to denote expectation with respect to $G^{\otimes\infty}$.
Differentiating \eqref{ghw8ds} with respect to $\beta_\theta$, we have
\begin{equation*}
    \frac{\partial}{\partial \beta_{\theta}}\PPP^{(1)}_{\xi_{\bbeta}}(\pi,\lambda)=\E \big\langle g_1\big\rangle_G, \quad \text{where} \quad g_\ell = \iprod[\big]{Z_\theta(\alpha^\ell)}{\sigma^\ell}.
\end{equation*}
By Gaussian integration by parts \cite[Lem.~1.1]{panchenko13a}, this can be rewritten as
\begin{equation}\label{eq:differentiate:once}
    \frac{\partial}{\partial \beta_{\theta}}\PPP^{(1)}_{\xi_{\bbeta}}(\pi,\la)=\beta_{\theta}\E \big\langle f_{1,1}-f_{1,2}\big\rangle_{G},
\end{equation}
where $f_{\ell,\ell'}= \CC\big((\alpha^{\ell},\sigma^{\ell}),(\alpha^{\ell^\prime},\sigma^{\ell^\prime})\big)$, and $\CC\colon(\N^{s-1}\times\Sigma)^2\to\R$ is given by
\begin{equation*}
 \CC\big((\alpha,\sigma),(\alpha^\prime,\sigma^{\prime})\big)
 =\E\big[\iprod[\big]{Z_\theta(\alpha)}{\sigma}\cdot \iprod[\big]{Z_\theta(\alpha')}{\sigma'}\big] 
 \stackref{dzo2}{=} 
 \iprod{\nabla\xi_{\theta}(\gamma_{r(\alpha,\alpha^\prime)})\sigma}{\sigma'}.
\end{equation*}
Because $\sigma,\sigma'\in\Sigma$ are both standard basis vectors, we have
\begin{equation*}
|\iprod{\nabla\xi_{\theta}(\gamma_{r(\alpha,\alpha^\prime)})\sigma}{\sigma'}|
\leq \|\nabla\xi_{\theta}(\gamma_{r(\alpha,\alpha^\prime)})\|_{\infty}
\leq \sup_{\|R\|_1\leq1}\|\nabla\xi_\theta(R)\|_\infty \leq C.
\end{equation*}
Therefore, for any $\ell,\ell^\prime\geq 1$ we have 
\eq{
|f_{\ell,\ell^\prime}|\leq C.
}
By differentiating \eqref{eq:differentiate:once} using the product rule, we obtain
\begin{equation*}
 \Big(\frac{\partial}{\partial \beta_{\theta}}\Big)^2\PPP_{\xi_{\bbeta}}^{(1)}(\pi,\lambda)
 =\E \big\langle f_{1,1}-f_{1,2}\big\rangle_{G}+\beta_{\theta}\frac{\partial}{\partial \beta_{\theta}}\E \big\langle f_{1,1}-f_{1,2}\big\rangle_{G}.
\end{equation*}
Observe that $f_{\ell,\ell'}$ has no dependence on $\beta_\theta$, and so direct calculation yields
\eq{
\frac{\partial}{\partial \beta_{\theta}}\E \big\langle f_{1,1}\big\rangle_{G}
 &= \E\big\langle f_{1,1}\cdot(g_1-g_2)\big\rangle_G, \\
\frac{\partial}{\partial \beta_{\theta}}\E \big\langle f_{1,2}\big\rangle_{G}
 &= \E\big\langle f_{1,2}\cdot(g_1+g_2-2g_3)\big\rangle_G
 = 2\E\big\langle f_{1,2}\cdot(g_1-g_3)\big\rangle_G.
}
Applying Gaussian integration by parts once more (this time using \cite[Lem.~1.2]{panchenko13a}), we have
\eq{ 
\E\big\langle f_{1,1}\cdot g_1\big\rangle_{G}
&= \beta_\theta\E\big\langle f_{1,1}\cdot(f_{1,1}-f_{1,2})\big\rangle_G, \\
\E\big\langle f_{1,1}\cdot g_2\big\rangle_G
&= \beta_\theta\E\big\langle f_{1,1}\cdot(f_{2,1}+f_{2,2}-2f_{2,3})\big\rangle_G, \\
\E\big\langle f_{1,2}\cdot g_1\big\rangle_G
&= \beta_\theta\E\big\langle f_{1,2}\cdot(f_{1,1}+f_{1,2}-2f_{1,3})\big\rangle_G, \\
\E\big\langle f_{1,2}\cdot g_3\big\rangle_G
&= \beta_\theta\E\big\langle f_{1,2}\cdot(f_{3,1}+f_{3,2}+f_{3,3}-3f_{3,4})\big\rangle_G.
}
The four previous displays together show 
\eeq{ \label{dxk5}
\Big|\Big(\frac{\partial}{\partial\beta_\theta}\Big)^2\PPP_{\xi_{\bbeta}}^{(1)}(\pi,\lambda)\Big| \leq C(1+\beta_\theta^2).
}
The combination of \eqref{dxk4} and \eqref{dxk5} yields the desired inequality.
\end{proof}


\begin{lemma}\label{lem:convex}
Assume $\xi$ satisfies~\ref{xi_power}, \eqref{xi_strongly_convex}, and $\bbeta$ satisfies \eqref{beta_decay1}, \eqref{beta_decay2}.
Then for any $d\in\DD$, the following statements hold.
\begin{enumerate}[label=\textup{(\alph*)}]

\item \label{lem:convex_a} 
$\inf_{\pi\in \Pi_d}\PP_{\xi_{\bbeta}}(\pi)\geq 0$.

\item \label{lem:convex_b}
For any $\theta\in \Theta_\Q$, the function $\beta_{\theta}\mapsto \inf_{\pi\in \Pi_d}\PP_{\xi_{\bbeta}}(\pi)$ is convex on an open interval.

\end{enumerate}
\end{lemma}

\begin{proof}
By Theorem~\ref{gen_con_par_thm} we have
\eeq{ \label{4jf0}
\lim_{\eps\searrow0}\limsup_{N\to\infty} \eff_{N,\xi_{\bbeta}}(d,\eps) = \inf_{\pi\in\Pi_{d}}\PP_{\xi_{\bbeta}}(\pi).
} 
Given any $\eps>0$, the set $\Sigma^N(d,\eps)$ from \eqref{eq:def:Sigma:N:d:eps} is nonempty for all large $N$.
For any $\sigma\in\Sigma^N(d,\eps)$, we have the trivial lower bound
\eeq{ \label{4jf1}
\eff_{N,\xi_{\bbeta}}(d,\eps)
\geq\frac{1}{N}\E H_{N, \xi_{\bbeta}}(\sigma)=0.
}
Part~\ref{lem:convex_a} now follows from \eqref{4jf0} and \eqref{4jf1}.

For part~\ref{lem:convex_b}, note that \eqref{4jf0} remains true for all $\beta_\theta$ in some open interval, since \eqref{beta_decay1} and \eqref{beta_decay2} remain true for all $\beta_\theta$ in some open interval. 
From the decomposition \eqref{HNbbeta_expanded},
a standard application of H\"{o}lder's inequality shows that whenever $\Sigma^N(d,\eps)$ is nonempty, the map $\beta_{\theta}\mapsto \sF_{N, \xi_{\bbeta}}(d,\eps)$ is convex for any $\theta\in \Theta_\Q$. 
Therefore, $\beta_{\theta}\mapsto \inf_{\pi\in \Pi_d}\PP_{\xi_{\bbeta}}(\pi)$ is a limit of convex functions (via \eqref{4jf0}) and hence convex.
\end{proof}

Finally, we show that Lemma~\ref{lem:second:derivative} and Lemma~\ref{lem:convex} imply Proposition~\ref{prop:diff}.
\begin{proof}[Proof of Proposition~\ref{prop:diff}]
  Fix $\theta\in \Theta_\Q$ and $d\in\DD$.
    For clarity, we make the dependence of $\PP_{\xi_{\bbeta}}(\pi,\la)$ on $\beta_{\theta}$ explicit by writing $\eff_{\pi,\la}(\beta_{\theta})\coloneqq\PP_{\xi_{\bbeta}}(\pi,\la)$, 
    as well as 
    \begin{equation*}
        \eff(\beta_\theta) 
        \coloneqq \inf_{\pi\in \Pi_{d}}\PP_{\xi_{\bbeta}}(\pi) 
        =\inf_{\pi\in\Pi_d,\lambda\in\R^\kappa}\eff_{\pi,\lambda}(\beta_\theta).
    \end{equation*}
    By Lemma~\hyperref[lem:convex_a]{\ref*{lem:convex}\ref*{lem:convex_a}}, $\eff(\beta_{\theta})$ is a finite number.
    So for any $\eps>0$, there exist $\pi_{\eps}\in \Pi_{d}$ and $\la_{\eps}\in \R^{\ka}$ such that
    \eeq{ \label{dgx3}
        \eff_{\pi_{\eps},\la_{\eps}}(\beta_{\theta})\leq \eff(\beta_{\theta})+\eps.
    }
    Since $\beta_\theta\mapsto\eff(\beta_\theta)$ is convex by Lemma~\hyperref[lem:convex_b]{\ref*{lem:convex}\ref*{lem:convex_b}}, the differentiability we seek is equivalent to the subdifferential $\partial \eff(\beta_{\theta})$ consisting of a single point. 
    So consider $x \in \partial \eff(\beta_{\theta})$. 
    For all small enough $h>0$, we have
    \eeq{ \label{dgx2}
    x\leq \frac{\eff(\beta_{\theta}+h)-\eff(\beta_{\theta})}{h}. 
    }
    Combining \eqref{dgx3} and \eqref{dgx2} with the trivial inequality $\eff(\beta_\theta+h) \leq \eff_{\pi_\eps,\lambda_\eps}(\beta_\theta+h)$, we see
    \begin{equation}\label{eq:subdifferential:upper}
        x\leq \frac{\eff_{\pi_{\eps},\la_{\eps}}(\beta_{\theta}+h)-\eff_{\pi_{\eps},\la_{\eps}}(\beta_{\theta})+\eps}{h}\leq \frac{\partial}{\partial \beta_{\theta}}\eff_{\pi_{\eps},\la_{\eps}}(\beta_{\theta})+Ch\big(1+(|\beta_{\theta}|+h)^2\big)+\frac{\eps}{h},
    \end{equation}
    where the second inequality follows from Taylor's theorem together with Lemma~\ref{lem:second:derivative}. 
    By analogous reasoning, 
    \begin{equation}\label{eq:subdifferential:lower}
    x 
    \geq \frac{\partial}{\partial \beta_{\theta}}\eff_{\pi_{\eps},\la_{\eps}}(\beta_{\theta})-Ch\big(1+(|\beta_{\theta}|+h)^2\big)-\frac{\eps}{h}.
    \end{equation}
    Finally, take $h=\sqrt{\eps}$ and send $\eps\to 0$. 
    The inequalities \eqref{eq:subdifferential:upper} and \eqref{eq:subdifferential:lower} together show that  $x$ can take at most one value.
\end{proof}

\section{Continuity and duality} \label{sec:continuity_and_duality}

The purpose of this section is threefold.
First we prove Proposition~\ref{prop:continuity}, part~\ref{prop:continuity_b} of which gives Lipschitz continuity of the Parisi functional.
Next we use this continuity in the proof of a duality statement (Lemma~\ref{pr456}) that is needed in Section~\ref{sec_lower_bound} to relate the functional appearing the Aizenman--Sims--Starr Scheme to the Parisi functional. 
Finally, we solve this duality in the balanced case (Lemma~\ref{pr457}), thereby eliminating the infimum appearing in \eqref{def:parisi:ftl} and fulfilling the promise of Remark~\ref{rmk_lambda}.

\subsection{Continuity of free energy and the Parisi functional} \label{sec:continuity}

Both parts of Proposition~\ref{prop:continuity} follow from interpolation arguments.
We begin with part~\ref{prop:continuity_a}.

\begin{proof}[Proof of Proposition {\hyperref[prop:continuity_a]{\ref*{prop:continuity}\ref*{prop:continuity_a}}}]
For $t\in [0,1]$, define the following interpolating free energy:
\begin{equation*}
\phi(t)=\frac{1}{N}\E \log\sum_{\sigma\in \Sigma^N(d)}\exp \HH_{t}(\sigma),
\quad\textnormal{where}\quad \HH_{t}(\sigma)=\sqrt{t}\, H_{N,\xi}(\sigma)+\sqrt{1-t}\, H_{N,\tilde\xi}(\sigma).
\end{equation*}
We assume $H_{N,\xi}$ is independent of $H_{N,\tilde\xi}$.
Let $\langle\cdot\rangle_t$ denote expectation with respect to the Gibbs measure $G_{t}(\sigma)\propto \exp\HH_t(\sigma)$ on $\Sigma^N(d)$.
By differentiating the expression for $\phi(t)$ with respect to $t$, we have
\eq{
\phi'(t) = \frac{1}{N}\E\Big\langle\frac{\partial \HH_t(\sigma)}{\partial t}\Big\rangle_t.
}
By Gaussian integration by parts \cite[Lem.~1.1]{panchenko13a}, this can be rewritten as
\eq{
&\phi^\prime(t)= N^{-1}\E\big\langle \CC(\sigma^1,\sigma^1)-\CC(\sigma^1,\sigma^2)\big\rangle_t,
}
where $\sigma^1,\sigma^2$ denote independent samples from $G_t$, and $\CC\colon\Sigma^N(d)\times\Sigma^N(d)\to\R$ is given by
\eq{
\CC(\sigma,\sigma^\prime)
=\E\Big[\frac{\partial \HH_{t}}{\partial t}(\sigma)\HH_{t}(\sigma')\Big]
\stackref{general_cov}{=}\frac{1}{2}\Big[\xi\Big(\frac{\sigma\sigma'^\sT}{N}\Big) - \tilde\xi\Big(\frac{\sigma\sigma'^\sT}{N}\Big)\Big].
}
Since $\|\sigma\sigma'^\sT\|_1\leq N$ for any $\sigma,\sigma'\in\Sigma^N$, we deduce that
\begin{equation*}
\sup_{t\in (0,1)}|\phi^\prime(t)|\leq \sup_{\|R\|_1\leq 1}|\xi(R)-\tilde\xi(R)|.
\end{equation*}
In summary,
\begin{equation*}
\big|\eff_{N,\xi}-\eff_{N,\tilde\xi}\big|=|\phi(1)-\phi(0)\big|\leq \sup_{\|R\|_1\leq 1}|\xi(R)-\tilde\xi(R)|. \qedhere
\end{equation*}
\end{proof}

We now turn our attention to Proposition~\hyperref[prop:continuity_b]{\ref*{prop:continuity}\ref*{prop:continuity_b}}, 
which will be an easy consequence of the following lemma.
In fact, this lemma is what allows the Parisi functional to be extended to arbitrary paths in the first place.

\begin{lemma} \label{pconlem}
Assume $\xi$ and $\tilde\xi$ satisfy~\ref{xi_power}.
Then for any $d\in\DD$, discrete paths $\pi,\tilde\pi\in\Pi_{d}^{\disc}$, $\lambda\in\R^\kappa$, and nonempty $S\subseteq\Sigma^N$, we have
\eeq{ \label{ne3w}
&|\PPP_{N,\xi}^{(1)}(\pi,\lambda;S) - \PPP_{N,\tilde \xi}^{(1)}(\tilde\pi,\lambda;S)| \\
&\leq \sup_{\|R\|_1\leq1}\|\nabla\xi(R)-\nabla\tilde\xi(R)\|_\infty + \frac{1}{2}\sup_{\|R\|_1\leq1}\|\nabla^2\xi(R)\|_\infty
\int_0^1\|\pi(t) - \tilde\pi(t)\|_1\ \dd t.
}
Similarly,
\eeq{ \label{ne4w}
&|\PPP_{\xi}^{(2)}(\pi) - \PPP_{\tilde\xi}^{(2)}(\tilde\pi)| \\
&\leq \sup_{\|R\|_1\leq1}\|\nabla\xi(R)-\nabla\tilde\xi(R)\|_\infty + \frac{1}{2}\sup_{\|R\|_1\leq1}\|\nabla^2\xi(R)\|_\infty
\int_0^1\|\pi(t) - \tilde\pi(t)\|_1\ \dd t.
}
\end{lemma}

\begin{proof}
As in \eqref{pi_rep}, we assume $\pi$ and $\tilde\pi$ are of the form
\eq{
\pi(t) = \gamma_r \quad \text{and} \quad \tilde\pi(t) = \tilde\gamma_r \quad \text{for $t\in(m_{r-1},m_{r}]$, $r\in\{1,\dots,s\}$}.
}
There is no loss of generality in using the same partition $0=m_0<m_1<\cdots<m_s=1$ for both paths, thanks to Remark~\ref{rmk_wd}.

First we prove \eqref{ne3w}.
Let $(\tilde Z_{i}(\alpha))_{1\le i\le N,\alpha\in\N^{s-1}}$ be an independent copy of the process $(Z_{i}(\alpha))_{1\le i\le N,\alpha\in\N^{s-1}}$ from \eqref{bni}.
We assume both processes are independent of the RPC weights $(\nu_\alpha)_{\alpha\in\N^{s-1}}$.
Then define an interpolating process
\eeq{ \label{vq12t8}
\HH_t(\alpha,\sigma) = \sum_{i=1}^N \big(\sqrt{t}\,\iprod[\big]{Z_{i}(\alpha)}{\sigma_i} + \sqrt{1-t}\,\iprod{\tilde Z_{i}(\alpha)}{\sigma_i} + \iprod{\lambda}{\sigma_i}\big), \quad t\in[0,1],
}
as well as the corresponding energy
\eeq{ \label{bz4hsk}
\phi(t) = \frac{1}{N}\E\log\sum_{\alpha\in\N^{s-1}}\sum_{\sigma\in S}\nu_\alpha\exp \HH_t(\alpha,\sigma).
}
We then have $\phi(1) = \PPP_{N,\xi}^{(1)}(\pi)$ while $\phi(0)=\PPP_{N,\tilde\xi}^{(1)}(\tilde\pi)$, and so we seek an upper bound on $|\phi(1)-\phi(0)|$.
For \eqref{ne3w} it suffices to show a uniform bound on the derivative of $\phi$:
\eeq{ \label{ne3w_su}
\sup_{t\in(0,1)} |\phi'(t)| \leq \text{R.H.S. of \eqref{ne3w}}.
}
To this end, let $\langle\cdot\rangle_t$ denote expectation with respect to the probability measure $G_t(\alpha,\sigma)\propto\nu_\alpha\exp\HH_t(\alpha,\sigma)$ on $\N^{s-1}\times S$.
By differentiating \eqref{bz4hsk} with respect to $t$, we have
\eeq{ \label{xme3}
\phi'(t) = \frac{1}{N}\E\Big\langle\frac{\partial \HH_t(\alpha,\sigma)}{\partial t}\Big\rangle_t.
}
Define $\CC\colon(\N^{s-1}\times S)^2\to\R$ by
\eeq{ \label{wzi8}
\CC\big((\alpha,\sigma),(\alpha',\sigma')\big)
&\stackrefp{vq12t8}{=} \frac{1}{N}\E\Big[\frac{\partial \HH_t(\alpha,\sigma)}{\partial t} \HH_t(\alpha',\sigma')\Big] \\ 
&\stackref{vq12t8}{=} \frac{1}{2N}\sum_{i=1}^N \E\big[\iprod[\big]{Z_{i}(\alpha)}{\sigma_i}\iprod[\big]{Z_{i}(\alpha')}{\sigma_i'} - \iprod{\tilde Z_{i}(\alpha)}{\sigma_i}\iprod{\tilde Z_{i}(\alpha')}{\sigma_i'}\big] \\
&\stackref{bni}{=} \frac{1}{2N}\sum_{i=1}^N\big[\iprod{\nabla\xi(\gamma_{r(\alpha,\alpha')})\sigma_i}{\sigma_i'}-\iprod{\nabla\tilde\xi(\tilde\gamma_{r(\alpha,\alpha')})\sigma_i}{\sigma_i'}\big],
}
Then Gaussian integration by parts \cite[Lem.~1.1]{panchenko13a} transforms \eqref{xme3} to 
\eeq{ \label{akl2}
\phi'(t) = \E\Big\langle \CC\big((\alpha^1,\sigma^1),(\alpha^1,\sigma^1)\big)-\CC\big((\alpha^1,\sigma^1),(\alpha^2,\sigma^2)\big)\Big\rangle_t,
}
where $(\alpha^1,\sigma^1),(\alpha^2,\sigma^2)$ denote independent samples from $G_t$.
Since both $\sigma_i$ and $\sigma_i'$ are standard basis vectors, it follows from \eqref{wzi8} that
\eeq{ \label{cn23}
&|\CC\big((\alpha,\sigma),(\alpha',\sigma')\big)|
\leq \frac{1}{2}\|\nabla\xi(\gamma_{r(\alpha,\alpha')}) - \nabla\tilde\xi(\tilde\gamma_{r(\alpha,\alpha')})\|_\infty \\
&\leq \frac{1}{2}\big(\|\nabla\xi(\gamma_{r(\alpha,\alpha')}) - \nabla\xi(\tilde\gamma_{r(\alpha,\alpha')})\|_\infty
+ \|\nabla\xi(\tilde\gamma_{r(\alpha,\alpha')}) - \nabla\tilde\xi(\tilde\gamma_{r(\alpha,\alpha')})\|_\infty\big) \\
&\leq \frac{1}{2}\Big(\sup_{\|R\|_1\leq 1}\|\nabla^2\xi(R)\|_\infty \cdot \|\gamma_{r(\alpha,\alpha')}-\tilde\gamma_{r(\alpha,\alpha')}\|_1 + \sup_{\|R\|_1\le 1}\|\nabla\xi(R)-\nabla\tilde\xi(R)\|_\infty\Big).
}
In the special case $\alpha=\alpha'$, we can use the assumption that both $\pi$ and $\tilde\pi$ belong to $\Pi_d^{\disc}$ in order to conclude that 
\eq{
\gamma_{r(\alpha,\alpha)}=\gamma_s=\tilde\gamma_s=\diag(d).
}
That is, the first term in the final line of \eqref{cn23} vanishes when $\alpha=\alpha'$.
Therefore, using \eqref{cn23} in \eqref{akl2} leads to
\eq{
|\phi'(t)| \leq \sup_{\|R\|_1\le 1}\|\nabla\xi(R)-\nabla\tilde\xi(R)\|_\infty
+ \frac{1}{2}\sup_{\|R\|_1\leq 1}\|\nabla^2\xi(R)\|_\infty \E\big\langle\|\gamma_{r(\alpha^1,\alpha^2)}-\tilde\gamma_{r(\alpha^1,\alpha^2)}\|_1\big\rangle_t.
}
By \cite[Thm.~4.4]{panchenko13a}, the marginal of $r(\alpha^1,\alpha^2)$ under $\E(G_t^{\otimes 2})$ is the same with the marginal of $r(\alpha^1,\alpha^2)$ under $\E(\nu^{\otimes 2})$, where $\nu$ is the RPC from  Theorem~\ref{rpcthm}.
In particular,
\eq{
\E\big\langle\|\gamma_{r(\alpha^1,\alpha^2)}-\tilde\gamma_{r(\alpha^1,\alpha^2)}\|_1\big\rangle_t
= \sum_{r=1}^s \|\gamma_r - \tilde\gamma_r\|_1(m_r - m_{r-1})
= \int_0^1\|\pi(t)-\tilde\pi(t)\|_1\ \dd t.
}
Using this calculation in the previous inequality yields \eqref{ne3w_su}.

Now we turn to proving \eqref{ne4w}.
Since $\pi$ and $\tilde\pi$ both belong to $\Pi_d$, we have
\eeq{ \label{sjc43}
|\PPP_{\xi}^{(2)}(\pi)
- \PPP_{\tilde\xi}^{(2)}(\tilde\pi)|
&\stackref{PPP2_def}{\leq} \frac{1}{2}\int_0^1 |\vartheta_\xi(\pi(t)) - \vartheta_{\tilde\xi}(\tilde\pi(t))|\ \dd t. 
}
For any $\gamma$ and $\tilde \gamma$, the triangle inequality gives
\eeq{ \label{cnz1}
|\vartheta_\xi(\gamma)-\vartheta_{\tilde\xi}(\tilde\gamma)|
\leq |\vartheta_\xi(\gamma) - \vartheta_\xi(\tilde\gamma)|
+ |\vartheta_\xi(\tilde \gamma) - \vartheta_{\tilde\xi}(\tilde \gamma)|.
}
When $\|\gamma\|_1,\|\tilde\gamma\|_1\leq 1$, the first term on the right-hand side of \eqref{cnz1} obeys the upper bound
\eeq{ \label{db76}
|\vartheta_\xi(\gamma) - \vartheta_\xi(\tilde \gamma)|
\leq \sup_{\|R\|_1\leq 1} \|\nabla\vartheta_\xi(R)\|_\infty\cdot\|\gamma-\tilde\gamma\|_1.
}
Now recall from \eqref{vthet_def} that $\vartheta_\xi(R)=\iprod{R}{\nabla\xi(R)}-\xi(R)$.
It follows that the $(a,b)$ entry of $\nabla\vartheta_\xi(R)$ is
\eq{
\frac{\partial \vartheta}{\partial R_{a,b}}
= R_{a,b}\sum_{k,k'=1}^\kappa\frac{\partial^2\xi}{\partial R_{a,b}\partial R_{k,k'}},
}
and thus
$\|\nabla\vartheta_\xi(R)\|_\infty
\leq \|R\|_1\cdot\|\nabla^2\xi(R)\|_\infty$.
Using this in \eqref{db76} results in
\eeq{ \label{dn39}
|\vartheta_\xi(\gamma) - \vartheta_\xi(\tilde\gamma)| \leq \sup_{\|R\|_1\leq 1}\|\nabla^2\xi(R)\|_\infty \cdot\|\gamma-\tilde\gamma\|_1.
}
Meanwhile, the second term on the right-hand side of \eqref{cnz1} satisfies
\eeq{ \label{ela1}
|\vartheta_\xi(\tilde\gamma) - \vartheta_{\tilde\xi}(\tilde\gamma)|
&\leq |\iprod{\tilde\gamma}{\nabla\xi(\tilde\gamma)-\nabla\tilde\xi(\tilde\gamma)}| 
+ |\xi(\tilde\gamma)-\tilde\xi(\tilde\gamma)| \\
&\leq \sup_{\|R\|_1 \leq 1}\|\nabla\xi(R)-\nabla\tilde\xi(R)\|_\infty
+ \sup_{\|R\|_1\leq 1}|\xi(R)-\tilde\xi(R)|.
}
Since $\xi(0)=\tilde\xi(0)=0$ by~\ref{xi_power}, the second supremum above can be further controlled as
\eq{
\sup_{\|R\|_1\leq 1}|\xi(R)-\tilde\xi(R)|
\leq \sup_{\|R\|_1\leq 1}\|\nabla\xi(R) - \nabla\tilde\xi(R)\|_\infty.
}
Therefore, \eqref{ela1} can be rewritten as
\eeq{ \label{ej32}
|\vartheta_\xi(\tilde\gamma) - \vartheta_{\tilde\xi}(\tilde\gamma)|
\leq 2\sup_{\|R\|_1\leq1}\|\nabla\xi(R) - \nabla\tilde\xi(R)\|_\infty.
}
Applying both \eqref{dn39} and \eqref{ej32} in \eqref{cnz1} results in
\eq{
|\vartheta_\xi(\gamma)-\vartheta_{\tilde\xi}(\tilde\gamma)|
\leq 2\sup_{\|R\|_1\leq 1}\|\nabla\xi(R) - \nabla\tilde\xi(R)\|_\infty
+ \sup_{\|R\|_1\leq 1}\|\nabla^2\xi(R)\|_\infty \cdot\|\gamma-\tilde\gamma\|_1.
}
Applying this uniform bound to the integrand in \eqref{sjc43} yields \eqref{ne4w} as desired.
\end{proof}

A key consequence of Lemma~\ref{pconlem} is that the map $\Pi_d^\disc\ni\pi\mapsto\PPP_{N,\xi}^{(1)}(\pi,\lambda;S)$ has a unique continuous extension to all of $\Pi_d$.
We continue to write $\PPP_{N,\xi}^{(1)}$ for this extension, including the special case $\PPP_\xi^{(1)}(\pi,\lambda)=\PPP_{1,\xi}^{(1)}(\pi,\lambda;\Sigma)$ used in \eqref{eq:def:Psi}.
Of course, these extensions admit the same estimates \eqref{ne3w} and \eqref{ne4w} for any $\pi,\tilde\pi\in\Pi_d$.
We now record how these estimates lead to the continuity of Parisi functional, as claimed in Proposition~\hyperref[prop:continuity_b]{\ref*{prop:continuity}\ref*{prop:continuity_b}}.

\begin{proof}[Proof of Proposition {\hyperref[prop:continuity_b]{\ref*{prop:continuity}\ref*{prop:continuity_b}}}]
Recall the definition of the Parisi functional from \eqref{def:parisi:ftl}:
\eq{
\PP_{\xi}(\pi)
&\stackrefp{qk38}{=}\inf_{\lambda\in \R^{\kappa}}\big[\PPP_{\xi}^{(1)}(\pi,\lambda)+\PPP_{\xi}^{(2)}(\pi,\lambda)-\iprod{\lambda}{d}\big] \\
&\stackref{qk38}{=}\inf_{\lambda\in \R^{\kappa}}\big[\PPP_{1,\xi}^{(1)}(\pi,\lambda;\Sigma)+\PPP_{\xi}^{(2)}(\pi,\lambda)-\iprod{\lambda}{d}\big].
}
By combining \eqref{ne3w} and \eqref{ne4w}, we have
\eeq{ \label{44rf}
&\big|\PPP_{1,\xi}^{(1)}(\pi,\lambda;\Sigma)+\PPP_{\xi}^{(2)}(\pi,\lambda)
- \PPP_{1,\tilde\xi}^{(1)}(\tilde\pi,\lambda;\Sigma)-\PPP_{\tilde\xi}^{(2)}(\tilde\pi,\lambda)\big| \\
&\leq 2\sup_{\|R\|_1\leq1}\|\nabla\xi(R)-\nabla\tilde\xi(R)\|_\infty + \sup_{\|R\|_1\leq1}\|\nabla^2\xi(R)\|_\infty
\int_0^1\|\pi(t) - \tilde\pi(t)\|_1\ \dd t.
}
Since the right-hand side of \eqref{44rf} does not depend on $\lambda$, it follows that
\eq{
|\PP_\xi(\pi) - \PP_{\tilde\xi}(\tilde\pi)| \leq \text{R.H.S. of \eqref{44rf}}.
}
This is exactly what was claimed in \eqref{dje4}.
\end{proof}

\subsection{Duality of magnetization with Lagrange multiplier}
The following lemma makes precise the discussion in Remark~\ref{rmk_dual}.
It is a generalization of \cite[Lem.~2]{panchenko18a} and is proved similarly.
The proof is included in Appendix~\ref{subsec:lem:duality} for completeness.

\begin{lemma}\label{lem:duality}
Assume $\xi$ satisfies~\ref{xi_power}.
    Assume $d^{N}\in \DD_N$ converges to $d\in \DD$ as $N\to\infty$. 
    Then for any $d^\prime \in \DD$ and $\pi\in \Pi_{d^\prime}^{\disc}$, we have
    \eeq{ \label{mijr}
    \lim_{N\to\infty}\PPP^{(1)}_{N,\xi}(\pi,0;\Sigma^{N}(d^N))= \inf_{\la\in \R^{\ka}}\big[\PPP^{(1)}_{\xi}(\pi,\la)-\iprod{\la}{d}\big].
    }
\end{lemma}

On a technical note, we point out that $d'$ does not need to be related to $d$.
We will, however, only need the case $d'=d$. 
Now we bootstrap to a stronger statement.

\begin{lemma} \label{pr456}
Assume $\xi$ satisfies~\ref{xi_power}.
Assume $d^{N}\in \DD_N$ converges to $d\in \DD$ as $N\to\infty$.
Let $\pi\in\Pi_{d'}$, and consider any sequence $(\pi_N)_{N\geq1}$ in $\Pi_{d'}$ satisfying
\eeq{ \label{b7e3}
\lim_{N\to\infty}\int_0^1 \|\pi_N(t)-\pi(t)\|_1\ \dd t = 0,
}
We then have
\eeq{ \label{xjq32}
\lim_{N\to\infty}\PPP_{N,\xi}^{(1)}(\pi_N,0;\Sigma^N(d^N)) = \inf_{\lambda\in\R^\kappa}\big[\PPP_{\xi}^{(1)}(\pi,\lambda) - \iprod{\lambda}{d}\big].
}
\end{lemma}

\begin{proof}
Our goal is to extend \eqref{mijr} to general paths $\pi\in\Pi_{d'}$ and to allow for $\pi$ on the left-hand side of \eqref{mijr} to be replaced with an approximating sequence $(\pi_N)_{N\geq1}$. 
To this end, take any $\tilde\pi\in\Pi_{d'}^\disc$ and observe that
\eeq{ \label{wkld}
&\Big|\PPP_{N,\xi}^{(1)}(\pi_N,0;\Sigma^N(d^N))
- \inf_{\lambda\in\R^\kappa}\big[\PPP_{\xi}^{(1)}(\pi,\lambda) - \iprod{\lambda}{d}\big]\Big| \\
&\stackrefp{ne3w,mijr}{\leq}\big|\PPP_{N,\xi}^{(1)}(\pi_N,0;\Sigma^N(d^N)) - \PPP_{N,\xi}^{(1)}(\tilde\pi,0;\Sigma^N(d^N))\big| \\
&\qquad\qquad\qquad+\Big|\PPP_{N,\xi}^{(1)}(\tilde\pi,0;\Sigma^N(d^N))-\inf_{\lambda\in\R^\kappa}\big[\PPP_{\xi}^{(1)}(\tilde\pi,\lambda) - \iprod{\lambda}{d}\big]\Big| \\
&\qquad\qquad\qquad+\Big|\inf_{\lambda\in\R^\kappa}\big[\PPP_{\xi}^{(1)}(\tilde\pi,\lambda) - \iprod{\lambda}{d}\big]
- \inf_{\lambda\in\R^\kappa}\big[\PPP_{\xi}^{(1)}(\pi,\lambda) - \iprod{\lambda}{d}\big]\Big| \\
&\stackref{ne3w,mijr}{\leq} o(1)+ \sup_{\|R\|_1\leq1}\|\nabla^2\xi(R)\|_\infty
\int_0^1\big(\|\pi_N(t) - \tilde\pi(t)\|_1+\|\tilde\pi(t)-\pi(t)\|_1\big)\ \dd t \\
&\stackrefpp{b7e3}{ne3w,mijr}{\leq} o(1)+ \sup_{\|R\|_1\leq1}\|\nabla^2\xi(R)\|_\infty
\int_0^1 2\|\tilde\pi(t)-\pi(t)\|_1\ \dd t.
}
Since $\Pi_{d'}^\disc$ is dense in $\Pi_{d'}$ (with respect to the norm \eqref{norm}), we can choose $\tilde\pi$ to make the integral on the final line of \eqref{wkld} arbitrarily small.
The claim \eqref{xjq32} follows.
\end{proof}

\subsection{Duality in the balanced case}
When $d=d_\bal$, we can solve the infimum on the right-hand side of \eqref{xjq32}, at least for paths $\pi$ belonging to the set $\Pi^\star$ from \eqref{def:Pi:star}.
Namely, the infimum is achieved at $\lambda=0$.

\begin{lemma} \label{pr457}
Assume $\xi$ satisfies~\ref{xi_power} and \eqref{xi_symmetric}.
Then for any $\pi\in\Pi^\star$, we have
\eeq{ \label{xjq33}
\inf_{\lambda\in\R^\kappa}\big[\PPP_{\xi}^{(1)}(\pi,\lambda) - \iprod{\lambda}{d_{\bal}}\big]
= \PPP_{\xi}^{(1)}(\pi,0).
}
\end{lemma}

In what follows, it is necessary to recall the permutation action $R\mapsto\pmu\act R$ from \eqref{pmuR_def}.
Namely, the $(k,k')$ entry of $R$ is equal to the $(\pmu(k),\pmu(k'))$ entry of $\pmu\act R$.

\begin{lemma}
Assume $\xi$ satisfies~\ref{xi_power} and \eqref{xi_symmetric}.
Then for any permutation $\pmu\in S_\kappa$,
\eeq{ \label{cb3s}
\pmu\act\nabla\xi(R) = \nabla\xi(\pmu\act R) \quad \text{for all $R\in\R^{\kappa\times\kappa}$}.
}
\end{lemma}

\begin{proof}
First make a general observation not requiring symmetry:
if we define the function $\xi^\pmu(R) = \xi(\pmu\act R)$, then
\eeq{ \label{qpa9}
\text{$(k,k')$ entry of $\nabla\xi^\pmu(R)$}
= \text{$(\pmu(k),\pmu(k')$) entry of $\nabla\xi(\pmu\act R)$}.
}
The assumption that $\xi$ is symmetric means $\xi^\pmu = \xi(R)$.
In this case, we can now chase definitions from the left-hand side of \eqref{cb3s} to the right-hand side of \eqref{cb3s}:
\begin{align*}
\text{$(k,k')$ entry of $\pmu\act\nabla\xi(R)$} \notag
&\stackrefpp{pmuR_def}{qpa9}{=}\text{$(\pmu^{-1}(k),\pmu^{-1}(k'))$ entry of $\nabla\xi(R)$} \\
&\stackref{qpa9}{=}\text{$(k,k')$ entry of $\nabla\xi(\pmu\act R)$}. \qedhere
\end{align*}
\end{proof}

\begin{proof}[Proof of Lemma~\ref{pr457}]
It suffices to verify \eqref{xjq33} for discrete paths $\pi\in\Pi_d^\disc$, since such paths are dense in $\Pi_d$, and 
both sides of \eqref{xjq33}
are continuous in $\pi$ 
thanks to \eqref{ne3w}.
So let us assume $\pi\in\Pi_d^\disc$ is given by \eqref{pi_rep}.
It follows from \cite[Lem.~6]{panchenko18a} that $\lambda\mapsto\PPP_{\xi}^{(1)}(\pi,\lambda)$ is convex.
This implies $\lambda\mapsto\PPP_{\xi}^{(1)}(\pi,\lambda)-\iprod{\lambda}{d_{\bal}}$ is convex and thus achieves a global minimum wherever its gradient is zero.
Therefore, it suffices to show
\eq{ 
\frac{\partial\PPP_{\xi}^{(1)}(\pi,\lambda)}{\partial \lambda_k}\Big|_{\lambda=0} 
= \frac{\partial}{\partial\lambda_k}\iprod{\lambda}{d_\bal}
\quad \text{for each $k\in\{1,\dots,\kappa\}$}.
}
Since $d_\bal = \kappa^{-1}\bone$, this amounts to showing
\eeq{ \label{xi33}
\frac{\partial\PPP_{\xi}^{(1)}(\pi,\lambda)}{\partial \lambda_k}\Big|_{\lambda=0} 
= \frac{1}{\kappa} 
\quad \text{for each $k\in\{1,\dots,\kappa\}$}.
}
With $(\nu_\alpha)_{\alpha\in\N^{s-1}}$ denoting the RPC associated to \eqref{eq:m2}, we have
\eeq{ \label{xme32}
\PPP_{\xi}^{(1)}(\pi,\lambda) 
\stackref{qk38}{=} \PPP_{1,\xi}^{(1)}(\pi,\lambda;\Sigma) 
\stackref{pre_par_a}{=} \E\log\sum_{\alpha\in\N^{s-1}}\sum_{k=1}^\kappa\nu_\alpha\exp \iprod[\big]{Z(\alpha)+\lambda}{\mathbf{e}_k},
}
where $(Z(\alpha))_{\alpha\in\N^{s-1}}$ is a centered $\R^\kappa$-valued Gaussian process with covariance structure
\eeq{ \label{wr4t24}
\E[Z(\alpha) Z(\alpha')^\sT] = \nabla\xi(\gamma_{r(\alpha,\alpha')}), \quad \alpha,\alpha'\in\N^{s-1}.
}
Now differentiate \eqref{xme32} with respect to $\lambda_k$ and evaluate at $\lambda=0$:
\eeq{ \label{je8a}
\frac{\partial\PPP_{\xi}^{(1)}(\pi,\lambda)}{\partial \lambda_k}\Big|_{\lambda=0}
= \E\Big[\frac{\sum_{\alpha\in\N^{s-1}}\nu_\alpha\exp\iprod[\big]{Z(\alpha)}{\mathbf{e}_k}}{\sum_{\alpha\in\N^{s-1}}\sum_{k'=1}^\kappa\nu_\alpha\exp\iprod[\big]{Z(\alpha)}{\mathbf{e}_{k'}}}\Big].
}
If we sum over $k$, then the numerators on the right-hand side add up to the denominator:
\eeq{ \label{xe45}
\sum_{k=1}^\kappa \frac{\partial\PPP_{\xi}^{(1)}(\pi,\lambda)}{\partial \lambda_k}\Big|_{\lambda=0} = 1. 
}
On the other hand, we make the following claim.
\begin{claim} \label{exch_cl}
$(\iprod[\big]{Z(\alpha)}{\mathbf{e}_k})_{\alpha\in\N^{s-1},1\leq k\leq\kappa}$ is exchangeable in $k$.
\end{claim}
\begin{proofclaim}
For any permutation $\pmu\in S_\kappa$, we have
\eq{
\E[\iprod[\big]{Z(\alpha)}{\mathbf{e}_{\pmu(k)}}\iprod[\big]{Z(\alpha')}{\mathbf{e}_{\pmu(k')}}] 
&\stackrefpp{wr4t24}{cb3s}{=} \iprod{\nabla\xi(\gamma_{r(\alpha,\alpha')})\mathbf{e}_{\pmu(k)}}{\mathbf{e}_{\pmu(k')}} \\
&\stackrefpp{pmuR_def}{cb3s}{=} \iprod{(\pmu^{-1}\act\nabla\xi(\gamma_{r(\alpha,\alpha')}))\mathbf{e}_{k}}{\mathbf{e}_{k'}} \\
&\stackref{cb3s}{=} \iprod{(\nabla\xi(\pmu^{-1}\act\gamma_{r(\alpha,\alpha')}))\mathbf{e}_{k}}{\mathbf{e}_{k'}}.
}
Since $\pi\in\Pi^\star$, the matrix $\gamma_{r(\alpha,\alpha')}$ has identical diagonal entries and identical off-diagonal entries.
In particular, it is invariant under the group action \eqref{pmuR_def}: $\pmu^{-1}\act\gamma_{r(\alpha,\alpha')} = \gamma_{r(\alpha,\alpha')}$.
Therefore, the previous display shows
\eq{
\E[\iprod[\big]{Z(\alpha)}{\mathbf{e}_{\pmu(k)}}\iprod[\big]{Z(\alpha')}{\mathbf{e}_{\pmu(k')}}]
= \E[\iprod[\big]{Z(\alpha)}{\mathbf{e}_{k}}\iprod[\big]{Z(\alpha')}{\mathbf{e}_{k'}}].
}
Since $(\iprod[\big]{Z(\alpha)}{\mathbf{e}_k})_{\alpha\in\N^{s-1},1\leq k\leq\kappa}$ is a centered Gaussian process, this equivalence of covariances is enough to establish exchangeability.
\end{proofclaim}

Recall that $(Z(\alpha))_{\alpha\in\N^{s-1}}$ is independent of the RPC weights $(\nu_\alpha)_{\alpha\in\N^{s-1}}$.
Consequently, it follows from Claim~\ref{exch_cl} that the right-hand side of \eqref{je8a} does not depend on $k$.
In light of \eqref{xe45}, this forces \eqref{xi33}.
\end{proof}

\section{Lower bound} \label{sec_lower_bound}
In this section we prove Proposition~\ref{prop:symmetric:parisi}.
The A.S.S. scheme is presented in Section~\ref{ass_sec}, and the G.G. identities are proved in Section~\ref{sec_GGgen}.
These two inputs are then combined with symmetry to prove the proposition in Section~\ref{sec_lbsym}. 

Since the A.S.S. scheme is based on the cavity method, we must work with a cavity Hamiltonian $(H_{N,M,\xi}(\sigma))_{\sigma\in\Sigma^N}$.
Lemma~\ref{lem_cav_Ham} stated below guarantees its existence.
Throughout Section~\ref{sec_lower_bound}, we will write $G_{N,M,\xi}$ for the associated Gibbs measure.
More precisely, given some $d^N\in\DD_N$, $G_{N,M,\xi}$ is the probability measure on $\Sigma^N(d^N)$ defined by
\eeq{ \label{GNM_def}
G_{N,M,\xi}(\sigma) \coloneqq \frac{\exp H_{N,M,\xi}(\sigma)}{\zee_{N,M,\xi}(d^N)} \quad \text{where} \quad 
\zee_{N,M,\xi}(d^N) \coloneqq \sum_{\sigma\in\Sigma^N(d^N)}\exp H_{N,M,\xi}(\sigma).
}
We will write $\Law\big(\RR;\E(G_{N,M,\xi}^{\otimes \infty})\big)$ to denote the law of the array $\RR=(\RR_{\ell,\ell'})_{\ell,\ell'\geq1}$, where 
\eq{ 
\RR_{\ell,\ell'} = \frac{\sigma^\ell(\sigma^{\ell'})^\sT}{N}, \quad
(\sigma^1,\sigma^2,\dots)\sim\E(G_{N,M,\xi}^{\otimes\infty}).
}

\begin{lemma} \label{lem_cav_Ham}
Assume $\xi$ satisfies~\ref{xi_power}.
Then for any positive integers $N$ and $M$, there exists a centered Gaussian process $(H_{N,M,\xi}(\sigma))_{\sigma\in\Sigma^N}$ with covariance
\eeq{ \label{comma_covar}
\E[H_{N,M,\xi}(\sigma)H_{N,M,\xi}(\sigma')] = (N+M)\xi\Big(\frac{\sigma\sigma'^\sT}{N+M}\Big).
}
\end{lemma}

\begin{proof}
By the decomposition in~\ref{xi_power}, it suffices the prove the following for each $\theta$ in the parameter set $\Theta$ from \eqref{eq:def:Theta}:
there exists a centered Gaussian processes $(H_{N,M,\theta}(\sigma))_{\sigma\in\Sigma^N}$ such that
\eeq{\label{HNMtheta_cov}
\E[H_{N,M,\theta}(\sigma)H_{N,M,\theta}(\sigma')] = (N+M)\xi_\theta\Big(\frac{\sigma\sigma'^\sT}{N+M}\Big),
}
where $\xi_\theta$ is defined in \eqref{def:xi:theta}.
To this end, let $(H_{N,\theta}(\sigma))_{\sigma\in\Sigma^N}$ be the Gaussian process from Proposition~\hyperref[prop:xi:theta_c]{\ref*{prop:xi:theta}\ref*{prop:xi:theta_c}}, and then set
\eq{
H_{N,M,\theta}(\sigma) = \Big(\frac{N}{N+M}\Big)^{(\deg(\theta)-1)/2}H_{N,\theta}(\sigma),
}
where $\deg(\theta)$ is defined in \eqref{deg_def}.
This trivially leads to
\eq{ 
\E[H_{N,M,\theta}(\sigma)H_{N,M,\theta}(\sigma')]
&\stackrefp{eq:H:theta:cov}{=} \Big(\frac{N}{N+M}\Big)^{\deg(\theta)-1}\E[H_{N,\theta}(\sigma)H_{N,\theta}(\sigma')] \\
&\stackref{eq:H:theta:cov}{=} \Big(\frac{N}{N+M}\Big)^{\deg(\theta)-1}N\xi_\theta\Big(\frac{\sigma\sigma'^\sT}{N}\Big).
}
Then, by the fact that $\xi_\theta$ is a homogeneous polynomial, we achieve \eqref{HNMtheta_cov}:
\begin{align}
\E[H_{N,M,\theta}(\sigma)H_{N,M,\theta}(\sigma')] \label{HNMtheta_cov3}
&= (N+M)\Big(\frac{N}{N+M}\Big)^{\deg(\theta)}\xi_\theta\Big(\frac{\sigma\sigma'^\sT}{N}\Big) \\
&= (N+M)\xi_\theta\Big(\frac{\sigma\sigma'^\sT}{N+M}\Big). \notag \qedhere
\end{align}
\end{proof}

\subsection{Aizenman--Sims--Starr scheme} \label{ass_sec}
The A.S.S. scheme relates ratios of partition functions to the functional $\Psi_{N,\xi}$ from \eqref{Psi_def}.
In what follows, we use the following notation:
\eeq{ \label{gw4hg4}
f = O_{\nabla^2\xi}(g) \quad \text{means that}\quad |f| \leq g \cdot C\sup_{\|R\|_1\leq1}\|\nabla^2\xi(R)\|_\infty,
}
where $C$ is some universal constant.

\begin{proposition}\textup{(A.S.S. scheme)} \label{ass_prop}
Assume $\xi$ and $\tilde\xi$ satisfy~\ref{xi_power}.
Let $M$ be a positive integer, and suppose $d^N\in\DD_N$, $d^{N+M}\in\DD_{N+M}$, $\delta^M\in\DD_M$ satisfy
\eeq{ \label{perfect_cav}
N d^N + M\delta^M = (N+M)d^{N+M}.
}
We then have
\eeq{ \label{ass_lower}
\E\log\frac{\zee_{N+M,\tilde\xi}(d^{N+M})}{\zee_{N,\xi}(d^{N})} + O_{\nabla^2\xi}\Big(\frac{M^2}{N+M}\Big) 
&+ \frac{N+M}{2}\sup_{\|R\|_1\leq 1}|\xi(R)-\tilde\xi(R)| \\
&\geq M\Psi_{M,\xi}(\LL_{N,M};d^N,\Sigma^M(\delta^M)).
}
\end{proposition}

Proposition~\ref{ass_prop} follows from a standard application of Aizenman-Sims-Starr Scheme~\cite{aizenman-sims-starr03, chen13}. The proof is included in Appendix~\ref{sec_ass_proof} for completeness. In order to prove Proposition~\ref{prop:symmetric:parisi}, we will take $N,M$ to be a multiple of $\ka$ and use Proposition~\ref{ass_prop} with $d^N=d^{N+M}=\delta^M=d_{\bal}$.

\begin{remark} 
For the right-hand side of \eqref{ass_lower} to make sense,
it is essential that the Gibbs measure $G_{N,M,\xi}$ in \eqref{GNM_def} is restricted to the space $\Sigma^N(d^N)$, so that $N^{-1}\sigma\sigma^\sT = \diag(d^N)$ for all $\sigma$ in the support of $G_{N,M,\xi}$. 
This allows for Assumption~\hyperref[map_assumption_ii]{\ref*{map_assumption}\ref*{map_assumption_ii}} to hold with 
\eq{
(\XX,R,d)=\big(\Sigma^N(d^N),(\sigma,\sigma')\mapsto N^{-1}\sigma(\sigma')^\sT,d^N\big).
}
In this setting, Assumption~\hyperref[map_assumption_i]{\ref*{map_assumption}\ref*{map_assumption_i}} is obvious, and~\hyperref[map_assumption_iii]{\ref*{map_assumption}\ref*{map_assumption_iii}} follows from Proposition~\hyperref[prop:xi:theta_d]{\ref*{prop:xi:theta}\ref*{prop:xi:theta_d}}.
\end{remark}

\subsection{Ghirlanda--Guerra identities for generic models} \label{sec_GGgen}
In this section we prove that differentiability of the Parisi formula leads to the G.G. identities for generic models, thus fulfilling the outcomes discussed after Proposition~\ref{prop:diff}.
But in order to discuss the G.G. identies, we must first check that the relevant arrays are symmetric and positive semi-definite.

\begin{lemma} \label{lem_gramtype}
Given $\sigma^1,\dots,\sigma^n\in\Sigma^N$, denote $\RR_{\ell,\ell'} = N^{-1}\sigma^\ell(\sigma^{\ell'})^\sT$.
Then for any $w\in\R^\kappa$ and positive integer $p$, the array $(\QQ_{\ell,\ell'})_{1\leq\ell,\ell'\leq n}$ given by $\QQ_{\ell,\ell'}=\iprod{\RR^{\circ p}_{\ell,\ell'}w}{w}$ is symmetric and positive-semidefinite.
\end{lemma}

\begin{proof}
Fix $p\geq1$ and $w\in\R^\kappa$.
Consider the parameter $\theta=(p,1,1,w)\in\Theta$, for which \eqref{def:xi:theta} becomes
$\xi_\theta(R) = \iprod{R^{\circ p} w}{w}$.
By Proposition~\hyperref[prop:xi:theta_c]{\ref*{prop:xi:theta}\ref*{prop:xi:theta_c}}, there exists a centered Gaussian process $H\colon\Sigma^N\to\R$ with covariance
\eq{
\E[H(\sigma)H(\sigma')] =\xi_\theta\Big(\frac{\sigma\sigma'^\sT}{N}\Big)
= N\iprod{R(\sigma,\sigma')^{\circ p} w}{w},
}
where $R(\sigma,\sigma')=N^{-1}\sigma\sigma'^\sT$.
It follows that $(\sigma,\sigma')\mapsto\iprod{R(\sigma,\sigma')^{\circ p} w}{w}$ is a symmetric and positive-semidefinite map $\Sigma^N\times\Sigma^N\to\R$.
\end{proof}


\begin{proposition}\label{sym_gg}
Assume $\xi$ satisfies~\ref{xi_power}, \eqref{xi_strongly_convex}, and $\bbeta=(\beta_\theta)_{\theta\in\Theta_\Q}$ satisfies \eqref{beta_decay1}, \eqref{beta_decay2}.
Provided $\beta_{\theta}\neq 0$ for every $\theta\in \Theta_\Q$, the following holds for any positive integer $M$. 
Let $\LL_{N,M} = \Law(\RR;\E(G_{N,M,\xi_{\bbeta}}^{\otimes \infty}))$.
Then any subsequential weak limit of $\LL_{N,M}$ as $N\to\infty$ satisfies the $\kappa$-dimensional Ghirlanda--Guerra identities in Definition~\ref{def:GG}.
\end{proposition}

\begin{proof}
To simplify notation, let us assume $\LL_{N,M}$ converges weakly as $N\to\infty$.
Even if this convergence occurs only along a subsequence, the argument we give below can be restricted to that subsequence.
Under the assumption that $\LL_{N,M}$ converges, we must have that $d^N$ converges to some $d\in\DD$, since at volume $N$ the self-overlap matrix $\RR_{\ell,\ell}=N^{-1}\sigma^\ell(\sigma^\ell)^\sT$ is equal to $\diag(d^N)$ with probability one under $G_{N,M,\xi_{\bbeta}}^{\otimes\infty}$.

We divide the remainder of the proof into five steps.
In the first four steps, we consider a fixed $\theta\in\Theta_\Q$.

\medskip
\noindent \textit{Step 1. Isolate the role of $\theta$ in the Hamiltonian.}
\medskip

Similar to \eqref{HNbbeta_expanded}, we can realize $H_{N,M,\xi_{\bbeta}}$ as a sum:
\eeq{ \label{nj3e}
H_{N,M,\xi_{\bbeta}}(\sigma) = H_{N,M,\xi}(\sigma) + \sum_{\theta'\in\Theta_\Q}\beta_{\theta'} H_{N,M,\theta'}(\sigma),
}
where $H_{N,M,\xi}$ is as in \eqref{comma_covar}, $H_{N,M,\theta'}$ is as in \eqref{HNMtheta_cov}, and all Gaussian processes on the right-hand side of \eqref{nj3e} are independent.
We also consider the Hamiltonian without the contribution of $H_{N,M,\theta}$, as follows.
Define $\bbeta^- = (\beta^-_{\theta'})_{\theta'\in\Theta_\Q}$ by
\eq{ 
\beta^-_{\theta'} \coloneqq \begin{cases} 
\beta_{\theta'} &\text{if $\theta'\neq\theta$} \\
0 &\text{if $\theta'=\theta$}.
\end{cases}
}
We then have
\eeq{ \label{fjv8}
H_{N,M,\xi_{\bbeta^-}}(\sigma)
= H_{N,M,\xi}(\sigma) + \sum_{\substack{\theta'\in\Theta_{\Q} \\ \theta'\neq\theta}}\beta_{\theta'} H_{N,M,\theta'}(\sigma)
= H_{N,M,\xi_{\bbeta}}(\sigma) - \beta_{\theta} H_{N,M,\theta}(\sigma).
}
From our definitions,
\eeq{ \label{qkd8}
G_{N,M,\xi_{\bbeta}}(\sigma) 
&\stackrefpp{GNM_def}{fjv8}{=} \frac{\exp H_{N,M,\xi_{\bbeta}}(\sigma)}{\zee_{N,M,\xi_{\bbeta}}(d^N)} \\
&\stackref{fjv8}{=} \frac{\exp(\beta_\theta H_{N,M,\theta}(\sigma))}{\zee_{N,M,\xi_{\bbeta}}(d^N)/\zee_{N,M,\xi_{\bbeta^-}}(d^N)}\cdot\frac{\exp H_{N,M,\xi_{\bbeta^-}}(\sigma)}{\zee_{N,M,\xi_{\bbeta^-}}(d^N)} \\
&\stackrefpp{GNM_def}{fjv8}{=} \frac{\exp(\beta_{\theta}H_{N,M,\theta}(\sigma))}{\zee_{N,M,\xi_{\bbeta}}(d^N)/\zee_{N,M,\xi_{\bbeta^-}}(d^N)}\ G_{N,M,\xi_{\bbeta^-}}(\sigma).
}
Note that $H_{N,M,\xi_{\bbeta^-}}(\sigma)$ in independent of $H_{N,M,\theta}$.
Therefore, \eqref{qkd8} expresses the measure of interest $G_{N,M,\xi_{\bbeta}}$ as the Gibbs measure associated to $\beta_\theta H_{N,M,\theta}$ with respect to the independent reference measure $G_{N,M,\xi_{\bbeta^-}}$.
We thus write
\eeq{ \label{jdnv7}
\zee_{N,M,\theta}(\beta_{\theta}) \coloneqq \frac{\zee_{N,M,\xi_{\bbeta}}(d^N)}{\zee_{N,M,\xi_{\bbeta^-}}(d^N)} \quad \text{and} \quad
\eff_{N,M,\theta}(\beta_{\theta}) \coloneqq \frac{1}{N}\E\log\zee_{N,M,\theta}(\beta_{\theta}).
}
We think of these as functions of just $\beta_{\theta}$, keeping $\beta_{\theta'}$ fixed for every $\theta'\neq\theta$.

\medskip
\noindent \textit{Step 2. Show that $\eff_{N,M,\theta}(\beta_{\theta})$ converges as $N\to\infty$.}
\medskip

Upon defining $\tilde\xi_{\bbeta}\colon\R^{\kappa\times\kappa}\to\R$ by
\eeq{ \label{sxi3}
\tilde\xi_{\bbeta}(R) = \frac{N+M}{N}\xi_{\bbeta}\Big(\frac{N}{N+M}R\Big),
}
we can rewrite \eqref{comma_covar} for $\xi_{\bbeta}$ as
\eeq{ \label{cnr9}
\E[H_{N,M,\xi_{\bbeta}}(\sigma)H_{N,M,\xi_{\bbeta}}(\sigma')] = N\tilde\xi_{\bbeta}\Big(\frac{\sigma\sigma'^\sT}{N}\Big).
}
We can then use the notation of \eqref{explicit_xi} to write
\eq{
\frac{1}{N}\E\log \zee_{N,M,\xi_{\bbeta}}(d^N)
= \eff_{N,\tilde\xi_{\bbeta}}(d^N).
}
By Proposition~\hyperref[prop:continuity_a]{\ref*{prop:continuity}\ref*{prop:continuity_a}}, we have
\eq{
\big| \eff_{N,\xi_{\bbeta}}(d^N)&-\eff_{N,\tilde\xi_{\bbeta}}(d^N)\big|
\leq \sup_{\|R\|_1\leq1}|\xi_{\bbeta}(R)-\tilde\xi_{\bbeta}(R)|.
}
Assuming $\|R\|_1\leq1$, we now recall \eqref{sxi3} and apply the triangle inequality to obtain
\eq{
|\xi_{\bbeta}(R)-\tilde\xi_{\bbeta}(R)|
&\leq \Big|\xi_{\bbeta}(R) - \xi_{\bbeta}\Big(\frac{N}{N+M}R\Big)\Big|
+ \frac{M}{N}\Big|\xi(\beta)\Big(\frac{N}{N+M}R\Big)\Big| \\
&\leq \frac{M}{N+M}\sup_{\|Q\|_1\leq1}\|\nabla\xi_{\bbeta}(Q)\|_\infty
+ \frac{M}{N}\sup_{\|Q\|_1\leq1}|\xi_{\bbeta}(Q)|.
}
Since $M$ is fixed, we conclude from the two previous displays that
\eq{ 
\lim_{N\to\infty} \big| \eff_{N,\xi_{\bbeta}}(d^N)-\eff_{N,\tilde\xi_{\bbeta}}(d^N)\big| = 0.
}
Hence Theorem~\ref{gen_con_par_thm} gives 
\eq{
\lim_{N\to\infty} \frac{1}{N}\E\log\zee_{N,M,\xi_{\bbeta}}(d^N) = \inf_{\pi\in\Pi_d}\PP_{\xi_{\bbeta}}(\pi).
}
By the same argument (just replacing $\bbeta$ with $\bbeta^-$, which still satisfies \eqref{beta_decay1} and \eqref{beta_decay2}),
\eq{
\lim_{N\to\infty} \frac{1}{N}\E\log\zee_{N,M,\xi_{\bbeta^-}}(d^N) = \inf_{\pi\in\Pi_d}\PP_{\xi_{\bbeta^-}}(\pi).
}
In light of definition \eqref{jdnv7}, the two previous displays together yield
\eeq{ \label{zeij2}
\lim_{N\to\infty} \eff_{N,M,\theta}(\beta_\theta)
= \inf_{\pi\in\Pi_d}\PP_{\xi_{\bbeta}}(\pi) - \inf_{\pi\in\Pi_d}\PP_{\xi_{\bbeta^-}}(\pi).
}

\medskip
\noindent \textit{Step 3. Show that $Z_{N,M,\theta}(\beta_{\theta})$ concentrates.}
\medskip

By \eqref{cnr9} we have
\eq{
\E[H_{N,M,\xi_{\bbeta}}(\sigma)^2] = N\tilde\xi_{\bbeta}(\diag(d^N)) \quad \text{for all $\sigma\in\Sigma^N(d^N)$}.
}
Hence Gaussian concentration (see \cite[Thm.~1.2]{panchenko13a}) implies that for all $x\geq 0$,
\eq{
\P\Big(\Big|\frac{\log \zee_{N,M,\xi_{\bbeta}}(d^N)}{N} - \frac{\E\log \zee_{N,M,\xi_{\bbeta}}(d^N)}{N}\Big| \geq x\Big)
\leq 2\exp\Big(\frac{-x^2N}{4\tilde\xi_{\bbeta}(\diag(d^N))}\Big)
}
In particular, we have the following limit:
\eq{
\lim_{N\to\infty}\E\Big|\frac{\log \zee_{N,M,\xi_{\bbeta}}(d^N)}{N} - \frac{\E\log \zee_{N,M,\xi_{\bbeta}}(d^N)}{N}\Big| = 0. 
}
By the same argument (but replacing $\bbeta$ with $\bbeta^-$),
\eq{
\lim_{N\to\infty}\E\Big|\frac{\log \zee_{N,M,\xi_{\bbeta^-}}(d^N)}{N} - \frac{\E\log \zee_{N,M,\xi_{\bbeta^-}}(d^N)}{N}\Big| =0. 
}
In light of the definition \eqref{jdnv7}, the two previous displays together yield
\eeq{ \label{qld7}
\lim_{N\to\infty}\E\Big|\frac{\log\zee_{N,M,\theta}(\beta_{\theta})}{N} - \frac{\E\log\zee_{N,M,\theta}(\beta_{\theta})}{N}\Big| =0. 
}

\medskip
\noindent  \textit{Step 4. Conclude that $H_{N,M,\theta}(\sigma)$ from \eqref{nj3e} concentrates around $\E\langle H_{N,M,\theta}(\sigma)\rangle_{N}$, where $\langle\cdot\rangle_{N}$ denotes expectation with respect to $G_{N,M,\xi_{\bbeta}}^{\otimes \infty}$.}
\medskip

Since \eqref{beta_decay1} and \eqref{beta_decay2} remain true if $\beta_\theta$ is varied slightly, the limits \eqref{zeij2} and \eqref{qld7} remain true for all choices of $\beta_\theta$ in some open interval.
Furthermore, by Proposition~\ref{prop:diff}, the right-hand side of \eqref{zeij2} is differentiable with respect to $\beta_{\theta}$.
Therefore, by applying the result of \cite{panchenko10} (see the remark after Thm.~1) to the representation \eqref{qkd8}, we have
\eeq{ \label{dveq6}
\lim_{N\to\infty}\frac{1}{N}\E\big\langle\big|H_{N,M,\theta}(\sigma) - \E\langle H_{N,M,\theta}(\sigma)\rangle_{N}\big|\big\rangle_{N} = 0.
}

\medskip
\noindent \textit{Step 5. Conclude that the Ghirlanda--Guerra identities are satisfied in the large-$N$ limit.}
\medskip

Given any positive integer $n$, consider any bounded measurable function $f$ of the finite subarray $\RR^{(n)}= (\RR_{\ell,\ell'})_{1\leq\ell,\ell'\leq n}$.
With $\langle\cdot\rangle_{N}$ denoting expectation with respect to $G_{N,M,\xi_{\bbeta}}^{\otimes\infty}$, we have the following for any $\theta\in\Theta_\Q$:
\eq{
&\Big|\E\Big\langle f \cdot\Big(H_{N,M,\theta}(\sigma^1)-\E\langle H_{N,M,\theta}(\sigma^1)\rangle_{N}\Big)\Big\rangle_{N}\Big| 
\leq \|f\|_\infty\big\langle\big|H_{N,M,\theta}(\sigma^1) - \E\langle H_{N,M,\theta}(\sigma^1)\rangle_{N}\big|\big\rangle_{N}.
}
The right-hand side is $o(N)$ by \eqref{dveq6}, and so the left-hand side is $o(N)$ as well:
\eeq{ \label{dci9}
\lim_{N\to\infty} \frac{1}{N}\Big|\E\Big\langle f \cdot\Big(H_{N,M,\theta}(\sigma^1)-\E\langle H_{N,M,\theta}(\sigma^1)\rangle_{N}\Big)\Big\rangle_{N}\Big| = 0.
}
On the other hand, Gaussian integration by parts \cite[Lem.~1.2]{panchenko13a} gives
\eeq{ \label{xmwi2}
\E\Big\langle f \cdot H_{N,M,\theta}(\sigma^1)\Big\rangle_{N}
= \beta_\theta\E\Big\langle f\cdot\Big(\sum_{\ell=1}^n\CC(\sigma^1,\sigma^\ell)-n\CC(\sigma^1,\sigma^{n+1})\Big)\Big\rangle_N,
}
where $\CC\colon\Sigma^N(d^N)\times\Sigma^N(d^N)\to\R$ is given by
\eq{
\CC(\sigma,\sigma') = \E[H_{N,M,\theta}(\sigma)H_{N,M,\theta}(\sigma')]
\stackref{HNMtheta_cov3}{=} 
(N+M)\Big(\frac{N}{N+M}\Big)^{\deg(\theta)}\xi_\theta\Big(\frac{\sigma\sigma'^\sT}{N}\Big).
}
Hence \eqref{xmwi2} can be rewritten as
\eq{
\frac{1}{N}\E\Big\langle f \cdot H_{N,M,\theta}(\sigma^1)\Big\rangle_{N}
= \beta_\theta\Big(\frac{N}{N+M}\Big)^{\deg(\theta)-1}\E\Big\langle f\cdot\Big(\sum_{\ell=1}^n\xi_\theta(\RR_{1,\ell})-n\xi_\theta(\RR_{1,n+1})\Big)\Big\rangle_N.
}
Since $\RR_{1,1}=\diag(d^N)$ with probability one under $G_{N,M,\xi_{\bbeta}}$, we can further rewrite the right-hand side to obtain
\eeq{ \label{djc7}
\frac{1}{N}\E\Big\langle f \cdot H_{N,M,\theta}(\sigma^1)\Big\rangle_{N}
&= \beta_\theta\Big(\frac{N}{N+M}\Big)^{\deg(\theta)-1}\bigg[\E\langle f\rangle_N\cdot \xi_\theta(\diag(d^N)) \\
&\phantom{=}+ \E\Big\langle f\cdot\Big(\sum_{\ell=2}^n\xi_\theta(\RR_{1,\ell})-n\xi_\theta(\RR_{1,n+1})\Big)\Big\rangle_N\bigg].
}
In the special case of the constant function $f\equiv1$, we have
\eq{
\frac{\E\langle H_{N,M,\theta}(\sigma^1)\rangle_{N}}{N}
&= \beta_\theta\Big(\frac{N}{N+M}\Big)^{\deg(\theta)-1}\Big[\xi_\theta(\diag(d^N))
+ \E\Big\langle \sum_{\ell=2}^n\xi_\theta(\RR_{1,\ell})-n\xi_\theta(\RR_{1,n+1})\Big\rangle_N\Big] \\
&= \beta_\theta\Big(\frac{N}{N+M}\Big)^{\deg(\theta)-1}\Big[\xi_\theta(\diag(d^N))
- \E\langle\xi_\theta(\RR_{1,2})\rangle_N\Big].
}
Upon multiplying this last equation by $\E\langle f\rangle_N$, we obtain
\eeq{ \label{djc8}
&\frac{1}{N}\E\Big\langle f\cdot \E\langle H_{N,M,\theta}(\sigma^1)\rangle_N\Big\rangle_{N} \\
&= \beta_\theta\Big(\frac{N}{N+M}\Big)^{\deg(\theta)-1}\Big[\E\langle f\rangle_N\cdot \xi_\theta(\diag(d^N))
- \E\langle f\rangle_N\cdot \E\langle\xi_\theta(\RR_{1,2})\rangle_N\Big].
}
Subtracting \eqref{djc8} from \eqref{djc7} results in
\eq{
&\frac{1}{N}\E\Big\langle f \cdot\Big(H_{N,M,\theta}(\sigma^1)-\E\langle H_{N,M,\theta}(\sigma^1)\rangle_{N}\Big)\Big\rangle_{N} \\
&= \beta_\theta\Big(\frac{N}{N+M}\Big)^{\deg(\theta)-1}\Big[\E\Big\langle f\cdot\Big(\sum_{\ell=2}^n\xi_\theta(\RR_{1,\ell})-n\xi_\theta(\RR_{1,n+1})\Big)\Big\rangle_N + \E\langle f\rangle_N\cdot\E\langle\xi_\theta(\RR_{1,2})\rangle_N\Big].
}
By \eqref{dci9}, the left-hand side of this identity tends to $0$ as $N\to\infty$. 
By our crucial assumption that $\beta_\theta\neq0$, it follows that
\eq{
\lim_{N\to\infty}\bigg|\E\big\langle f\cdot \xi_\theta(\RR_{1,n+1})\big\rangle_N -
\frac{1}{n}\E\langle f\rangle_N\cdot\E\langle\xi_\theta(\RR_{1,2})\rangle_N - \frac{1}{n}\sum_{\ell=2}^n\E\big\langle f\cdot \xi_\theta(\RR_{1,\ell})\big\rangle_N\bigg| = 0.
}
Now let $\LL$ denote the weak limit of $\LL_{N,M}$ as $N\to\infty$.
The previous display means that when the array $\RR$ is distributed according to $\LL$, we have
\eeq{ \label{jfnv}
    \E[f(\RR^{(n)})\cdot \xi_{\theta}(\RR_{1,n+1})]=\frac{1}{n}\E[f(\RR^n)]\cdot \E\xi_{\theta}(\RR_{1,2})+\frac{1}{n}\sum_{\ell=2}^{n}\E[f(\RR^n)\cdot \xi_{\theta}(\RR_{1,\ell})].
}
This is a special case of the G.G. identity \eqref{eq:GG}.
All that remains is to argue that this special case implies the general case. 

Recall that $\xi_{\theta}(R)=\prod_{j=1}^m \iprod{R^{\circ p}w_j}{w_j}^{n_j}$, where $p,m,n_1,\ldots, n_m\geq 1$ and $w_1,\ldots, w_m\in [-1,1]^{\kappa}$ are the parameters defining $\theta$.
Since $(\Q\cap[-1,1])^{\kappa}$ is dense in $[-1,1]^\kappa$, and \eqref{jfnv} holds for every $\theta\in\Theta_\Q$, it follows that \eqref{jfnv} holds for every $\theta\in\Theta$.
That is, \eqref{eq:GG} holds whenever the function $\vphi$ in \eqref{Qphi} is a polynomial of the form $\vphi(x_1,\ldots, x_m)=\prod_{j=1}^{m}x_j^{n_j}$. 
By approximating continuous functions with linear combinations of such polynomials (together with constant functions, for which \eqref{eq:GG} is trivial), we deduce the same statement for any bounded continuous $\vphi\colon\R^{m}\to \R$. 
Finally, by approximating bounded measurable functions with bounded continuous functions (e.g.~using Lusin's theorem \cite[Thm.~2.24]{rudin87}), we obtain the identity \eqref{eq:GG} for any bounded measurable $\vphi$.
Note that these approximations are over compact domains, since every overlap matrix $\RR_{\ell,\ell'}$ almost surely belongs to the compact set $\nonpsd$ from \eqref{def_nonpsd}.
\end{proof}

Before we can make use of Proposition~\ref{sym_gg} in the next section, we need one more basic fact about the G.G. identities.

\begin{lemma} \label{lem_kappa_to_1}
Assume $\LL=\Law(\RR)$ is a Gram--de Finetti law that satisfies the $\kappa$-dimensional Ghirlanda--Guerra identities in Definition~\ref{def:GG}.
Then the scalar array $\QQ=(\QQ_{\ell,\ell'})_{\ell,\ell'\geq1}$ given by
$\QQ_{\ell,\ell'} = \tr(\RR_{\ell,\ell'})$
satisfies the $1$-dimensional Ghirlanda--Guerra identities.
That is, for any bounded measurable function $f$ of the finite subarray $\QQ^{(n)}= (\QQ_{\ell,\ell'})_{1\leq\ell,\ell'\leq n}$, and any bounded measurable $\psi\colon\R\to\R$, we have
    \begin{equation}\label{eq:GGT}
    \E[f(\QQ^{(n)})\psi(\QQ_{1,n+1})]=\frac{1}{n}\E[f(\QQ^{(n)})]\E[\psi(\QQ_{1,2})]+\frac{1}{n}\sum_{\ell=2}^{n}\E[f(\QQ^{(n)})\psi(\QQ_{1,\ell})].
    \end{equation}
\end{lemma}

\begin{proof}
Define $\vphi\colon\R^\kappa\to\R$ by $ \vphi(x_1,\dots,x_\kappa) = \psi(x_1+\dots+x_\kappa)$.
We then have
\eq{
\psi(\tr(R))=\vphi(\iprod{R\vv e_1}{\vv e_1},\dots,\iprod{R\vv e_\kappa}{\vv e_\kappa}),
}
and so \eqref{eq:GGT} is a special case of \eqref{eq:GG}.
\end{proof}

\subsection{Proof of lower bound from symmetry} \label{sec_lbsym}
For here to the end of Section~\ref{sec_lower_bound}, we always assume $d^N=d_\bal$.
Recall that $\Sigma^N(d_\bal)$ is nonempty if and only if $N$ is a multiple of $\kappa$, and so we will frequently replace $N$ with $\kappa N$ so that $N$ continues to be a generic positive integer.
For instance, we will write $G_{\kappa N,M,\xi}$ as in \eqref{GNM_def}, but now with the understanding that $d^{\kappa N}=d_\bal$.

\begin{lemma} \label{lem:symmetry_cavity}
Assume $\xi$ satisfies~\ref{xi_power} and \eqref{xi_symmetric}.
Then the following statements hold.
\begin{enumerate}[label=\textup{(\alph*)}]

\item \label{lem:symmetry_cavity_a}
For any permutation $\pmu\in S_\kappa$, we have
\eq{ 
\big(H_{N,M,\xi}(\sigma)\big)_{\sigma\in \Sigma^{N}}\stackrel{\mathrm{law}}{=}\big(H_{N,M,\xi}(\pmu\act\sigma)\big)_{\sigma\in \Sigma^{N}}.
}

\item \label{lem:symmetry_cavity_b}
If $d^{\kappa N}=d_\bal$, then $\Law\big(\RR;\E(G_{\kappa N,M,\xi}^{\otimes \infty})\big)$ is symmetric in the sense of Definition~\ref{def:L:symmetry}.

\end{enumerate}
\end{lemma}

\begin{proof} 
For $k\in\{1,\dots,\kappa\}$ and $\sigma=(\sigma_1,\dots,\sigma_N)\in\Sigma^N$, we isolate the $k^\mathrm{th}$ coordinates in $\sigma$ by writing $\sigma(k) = (\sigma_1(k),\dots,\sigma_N(k))^\sT\in\R^N$.
Recall the overlap map $R(\sigma,\sigma') = N^{-1}\sigma\sigma'^\sT$ from \eqref{def:overlap}.
That is, the $(k,k')$ entry of $R(\sigma,\sigma')$ is given by an inner product:
\eeq{ \label{entries_R}
R(\sigma,\sigma)_{k,k'} = N^{-1}\iprod{\sigma(k)}{\sigma'(k')}.
}
For any $\pmu\in S_\kappa$ and $\sigma,\sigma'\in\Sigma^N$, we claim that
\eeq{ \label{nvc8}
\pmu\act R(\sigma,\sigma') = R(\pmu\act\sigma,\pmu\act\sigma').
}
The action by $\pmu$ on the left-hand side of \eqref{nvc8} is defined in \eqref{pmuR_def}, whereas the action on the right-hand side is defined in \eqref{dvd2}.
For any $k,k'\in\{1,\dots,\kappa\}$, we have
\eq{
\text{$(k,k')$ entry of $\pmu\act R(\sigma,\sigma')$}
&\stackref{pmuR_def}{=} \text{$(\pmu^{-1}(k),\pmu^{-1}(k'))$ entry of $R(\sigma,\sigma')$} \\
&\stackrefpp{entries_R}{dvd2}{=} \iprod{\sigma(\pmu^{-1}(k))}{\sigma'(\pmu^{-1}(k'))} \\
&\stackref{dvd2}{=} \iprod{[\pmu\act\sigma](k)}{[\pmu\act\sigma'](k')} \\
&\stackrefpp{entries_R}{dvd2}{=} \text{$(k,k')$ entry of $R(\pmu\act\sigma,\pmu\act\sigma')$}.
}
This verifies \eqref{nvc8}.
We can now write
\eq{
\E[H_{N,M,\xi}(\pmu\act\sigma)H_{N,M,\xi}(\pmu\act\sigma)]
&\stackrefpp{comma_covar}{nvc8}{=} (N+M)\xi\Big(\frac{N}{\kappa N+M}R(\pmu\act\sigma,\pmu\act\sigma')\Big) \\
&\stackref{nvc8}{=} (N+M)\xi\Big(\frac{N}{N+M}\big(\pmu\act R(\sigma,\sigma')\big)\Big) \\
&\stackrefpp{xi_symmetric}{nvc8}{=} (N+M)\xi\Big(\frac{ N}{N+M}R(\sigma,\sigma')\Big) \\
&\stackrefpp{comma_covar}{nvc8}{=}\E[H_{N,M,\xi}(\sigma)H_{N,M,\xi}(\sigma)].
}
This proves part~\ref{lem:symmetry_cavity_a}.

For part~\ref{lem:symmetry_cavity_b}, consider the following Gibbs measure:
\eeq{ \label{djvnr_pmu}
G_{\kappa N,M,\xi}^\pmu(\sigma) \coloneqq \frac{\exp H_{\kappa N,M,\xi}(\pmu^{-1}\act \sigma)}{\sum_{\sigma'\in\Sigma^{\kappa N}(d_\bal)}\exp H_{\kappa N,M,\xi}(\pmu^{-1}\act\sigma')}, \quad \sigma\in\Sigma^{\kappa N}(d_\bal).
}
Part~\ref{lem:symmetry_cavity_a} implies $\E(G_{\kappa N,M,\xi}^{\otimes\infty})=\E((G_{\kappa N,M,\xi}^\pmu)^{\otimes\infty})$.
In particular,
\eeq{ \label{dcop1}
\Law\big(\RR;\E(G_{\kappa N,M,\xi}^{\otimes\infty})\big)
= \Law\big(\RR;\E((G_{\kappa N,M,\xi}^\pmu)^{\otimes\infty})\big).
}
But notice that $\sigma'\mapsto\pmu^{-1}\act\sigma'$ is a bijection on $\Sigma^{\kappa N}(d_\bal)$ since $\pmu\act\diag(d_\bal)=\diag(d_\bal)$.
Therefore, the denominator in \eqref{djvnr_pmu} can be rewritten to give
\eq{
G_{\kappa N,M,\xi}^\pmu(\sigma) = \frac{\exp H_{\kappa N,M,\xi}(\pmu^{-1}\act \sigma)}{\sum_{\sigma'\in\Sigma^{\kappa N}(d_\bal)}\exp H_{\kappa N,M,\xi}(\sigma')}.
}
Drawing from $G_{\kappa N,M,\xi}^\pmu$ is thus equivalent to sampling from $G_{\kappa N,M,\xi}$ and then applying $\pmu$ to the sample.
Consequently,
\eeq{ \label{dcop2}
\Law\big(\RR;\E((G_{\kappa N,M,\xi}^\pmu)^{\otimes\infty})\big)
= \Law\big(\RR^\pmu;\E(G_{\kappa N,M,\xi}^{\otimes\infty})\big),
}
where $\RR^\pmu = (\RR_{\ell,\ell'}^\pmu)_{\ell,\ell'\geq1}$ is the array given by
$\RR_{\ell,\ell'}^\pmu = R(\pmu\act\sigma^\ell,\pmu\act\sigma^{\ell'})$.
On the other hand, \eqref{nvc8} says that $\RR^\pmu = \pmu\act\RR$, hence
\eeq{ \label{dcop3}
\Law\big(\RR^\pmu;\E(G_{\kappa N,M,\xi}^{\otimes\infty})\big)
= \Law\big(\pmu\act\RR;\E(G_{\kappa N,M,\xi}^{\otimes\infty})\big).
}
Chaining together \eqref{dcop1}--\eqref{dcop3} yields $\Law\big(\RR;\E(G_{\kappa N,M,\xi}^{\otimes\infty})\big)=\Law\big(\pmu\act\RR;\E(G_{\kappa N,M,\xi}^{\otimes\infty})\big)$.
\end{proof}


We are now ready to complete the main objective of this section.

\begin{proof}[Proof of Proposition~\ref{prop:symmetric:parisi}]

It is an elementary fact that for any real-valued sequence $(a_N)_{N\geq1}$ and an integer $M\geq 1$,
\eq{ 
\liminf_{N\to\infty} \frac{a_N}{N} \geq \frac{1}{M}\liminf_{N\to\infty} (a_{N+M}-a_N).
}
Applying this observation to $a_N = \kappa^{-1}\E\log \zee_{\kappa N,\xi_{\bbeta}}(d_\bal)$ and replacing $M$ by $\ka M$ results in
\eq{ 
\liminf_{N\to\infty}\eff_{\kappa N,\xi_{\bbeta}}(d_\bal)
\geq \frac{1}{\kappa M}\liminf_{N\to\infty}\E\log\frac{\zee_{\kappa(N+M),\xi_{\bbeta}}(d_\bal)}{\zee_{\kappa N,\xi_{\bbeta}}(d_\bal)}.
}
Now apply Proposition~\ref{ass_prop} with $d^{\kappa N} = d^{\kappa(N+M)} =\delta^{\kappa M}=d_\bal$, so that \eqref{ass_lower} yields
\eeq{ \label{iwp00}
&\liminf_{N\to\infty}\eff_{\kappa N,\xi_{\bbeta}}(d_\bal)
\geq \liminf_{N\to\infty}\Psi_{\kappa M,\xi_{\bbeta}}\Big(\Law\big(\RR;\E(G_{\kappa N,\kappa M,\xi_{\bbeta}}^{\otimes\infty})\big);\Sigma^{\kappa M}(d_\bal)\Big).
}
Note that because $G_{\kappa N_k,\kappa M,\xi_{\bbeta}}$ is supported on $\Sigma^{\kappa N_k}(d_\bal)$, we trivially have
\eeq{ \label{pre_deterministic}
\RR_{\ell,\ell} = R(\sigma^\ell,\sigma^\ell) = \diag(d_\bal) \quad \text{for all $\ell,\ell'\geq1$, with probability one}.
}
By Lemma~\ref{lem_gramtype}, we know
\eeq{ \label{pre_type}
\Law\big(\RR;\E(G_{\kappa N,\kappa M,\xi_{\bbeta}}^{\otimes\infty})\big) \text{ is Gram--de Finetti (Definition~\ref{type_def})}.
}
By Lemma~\ref{lem:symmbeta}, $\xi_{\bbeta}$ satisfies the symmetry condition \eqref{xi_symmetric}.
So by Lemma~\hyperref[lem:symmetry_cavity_b]{\ref*{lem:symmetry_cavity}\ref*{lem:symmetry_cavity_b}}, we know 
\eeq{ \label{pre_sym}
\Law\big(\RR;\E(G_{\kappa N,\kappa M,\xi_{\bbeta}}^{\otimes\infty})\big) \text{ is symmetric (Definition~\ref{def:L:symmetry})}.
}
Now pass to a subsequence $(N_j)_{j\geq1}$ that achieves the infimum on the right-hand side of \eqref{iwp00}:
\eeq{ \label{iwp0}
\liminf_{N\to\infty}\eff_{\kappa N,\xi_{\bbeta}}(d_\bal)
\geq \lim_{j\to\infty}
\Psi_{\kappa M,\xi_{\bbeta}}\Big(\Law\big(\RR;\E(G_{\kappa N_j,\kappa M,\xi_{\bbeta}}^{\otimes\infty})\big);\Sigma^{\kappa M}(d_\bal)\Big).
}
By passing to a further subsequence, we may also assume there is some Gram--de Finetti law $\LL_M$ such that
\eeq{ \label{fjgx2}
\text{$\Law\big(\RR;\E(G_{\kappa N_j,\kappa M,\xi_{\bbeta}}^{\otimes\infty})\big)$ converges weakly to $\LL_M$ as $j\to\infty$}.
}

Let us check that $\LL_M$ satisfies the hypotheses of Lemma~\ref{lemma:symmetric}.
First, \eqref{pre_deterministic} obviously remains true under $\LL_M$.
Second, $\LL_M$ must be a Gram--de Finetti law by \eqref{pre_type}, because $\RR\mapsto(\iprod{\RR_{\ell,\ell'}^{\circ p}w}{w})_{\ell,\ell'\geq1}$ is a continuous operation on arrays for any $p\geq1$ and $w\in\R^\kappa$.
Similarly, $\LL_M$ inherits symmetry from \eqref{pre_sym}, since $\RR\mapsto\pmu\act\RR$ is a continuous operation on arrays for any $\pmu\in S_\kappa$.
Finally, according to Proposition~\ref{sym_gg}, the $\kappa$-dimensional G.G. identities are satisfied by $\LL_M$.

We can now invoke Lemma~\ref{lemma:symmetric}: if $\LL_M = \Law(\RR)$, then almost surely
\eeq{ \label{dje21}
R_{\ell,\ell^\prime}=\Phi^\star\big(\tr(R_{\ell,\ell^\prime})\big), \quad \text{where} \quad
\Phi^\star(q)=\frac{q}{\kappa}\bI_{\kappa}+\frac{1-q}{\kappa^2}\bone\bone^{\sT}.
}
In other words, if we denote the array of traces by $\QQ=(\QQ_{\ell,\ell'})_{\ell,\ell'\geq1}=(\tr(\RR_{\ell,\ell'}))_{\ell,\ell'\geq1}$ and define $\bar\LL_M = \Law(\QQ)$, then $\LL_M$ is the pushforward\footnote{Here there is a slight abuse of notation: in \eqref{dje21}, $\Phi^\star$ is a map $[0,1]\to\R^{\kappa\times\kappa}$, whereas in \eqref{pushforward}, $\Phi^\star$ is thought of as a map $[0,1]^{\N\times\N}\to(\R^{\kappa\times\kappa})^{\N\times\N}$ defined by performing \eqref{dje21} to every element in an array.}
of $\bar\LL_M$ under $\Phi^\star$:
\eeq{ \label{pushforward}
\LL_M = \bar\LL_M\circ(\Phi^\star)^{-1}.
}
Furthermore, $\bar\LL_M$ satisfies the $1$-dimensional G.G. identities by Lemma~\ref{lem_kappa_to_1}.
Therefore, Theorem~\ref{representation_thm} says that $\bar\LL_M = \bar\LL_{\mu_M}$, where $\mu_M=\Law(\tr(\RR_{1,2}))$.
This is a probability measure on $[0,1]$, and we consider its quantile function
\eq{
Q_{\mu_M}(t) \coloneqq \inf\{q\geq 0:\, \mu([0,q]) \geq t\}, \quad t\in(0,1].
}
Then define $\pi_M\colon(0,1]\to\Gamma^\star$ by $\pi_M(t) = \Phi^\star(Q_{\mu_M}(t))$.
Since $Q_{\mu_M}$ is left-continuous and $\Phi^\star$ is continuous, this $\pi_M$ is an element of the path space $\Pi^\star$ defined in \eqref{def:Pi:star}.

Now let $(\mu_{M,\,j})_{j\geq1}$ be a sequence of probability measures on $[0,1]$ such that
\begin{enumerate}[label=\textup{(\roman*)}]

\item \label{pi_prop_i}
$\mu_{M,\,j}$ is supported on finitely many points.

\item \label{pi_prop_ii}
$\mu_{M,\,j}(\{1\})>0$.

\item \label{pi_prop_iii}
$\mu_{M,\,j}$ converges weakly to $\mu_M$ as $j\to\infty$.

\end{enumerate}
Then define $\pi_{M,\,j}\colon(0,1]\to\Gamma^\star$ by $\pi_{M,\,j}(t)=\Phi^\star(Q_{\mu_{M,\,j}}(t))$.
By properties~\ref{pi_prop_i} and~\ref{pi_prop_ii}, this $\pi_{M,\,j}$ is an element of the path space $\Pi_{d_\bal}^\disc$ from \eqref{def_Pi_disc}.
Meanwhile, property~\ref{pi_prop_iii} implies
\begin{subequations} \label{vgwe}
\eeq{
\lim_{j\to\infty}\int_0^1|Q_{\mu_M}(t)-Q_{\mu_{M,\,j}}(t)|\ \dd t = 0.
}
As $\Phi^\star$ is Lipschitz and $\pi_M(t)-\pi_{M,\,j}(t) = \Phi^\star(Q_{\mu_M}(t))-\Phi^\star(Q_{\mu_{M,\,j}}(t))$, this limit implies
\eeq{
\lim_{j\to\infty}\int_0^1\| \pi_M(t)-\pi_{M,\,j}(t)\|_1\ \dd t = 0.
}
\end{subequations}
From this convergence and Lemma~\ref{pconlem}, it follows that
\eeq{ \label{qspd5}
\lim_{j\to\infty}\PPP_{\kappa M,\xi_{\bbeta}}^{(1)}(\pi_{M,\,j},0;\Sigma^{\kappa M}(d_\bal))
&= \PPP_{\kappa M,\xi_{\bbeta}}(\pi_M,0;\Sigma^{\kappa M}(d_\bal)), \\
\text{and} \quad
\lim_{j\to\infty}\PPP_{\xi_{\bbeta}}^{(2)}(\pi_{M,\,j}) &= \PPP_{\xi_{\bbeta}}^{(2)}(\pi_M).
}
On the other hand, we claim the following equality:
\eeq{ \label{dcj7}
\PPP_{\kappa M,\xi_{\bbeta}}^{(1)}(\pi_{M,\,j},0;\Sigma^{\kappa M}(d_\bal))
+ \PPP_{\xi_{\bbeta}}^{(2)}(\pi_{M,\,j})
= \Psi_{\kappa M,\xi_{\bbeta}}(\bar\LL_{\mu_{M,\,j}}\circ(\Phi^\star)^{-1};\Sigma^{\kappa M}(d_\bal)),
}
where $\bar\LL_{\mu_{M,\,j}}$ is the $1$-dimensional Gram--de Finetti law from Theorem~\ref{representation_thm}.
Indeed, suppose $\pi_{M,\,j}$ has the form
\eq{
\pi_{M,\,j}(t) = \gamma_r \quad \text{for $t\in(m_{r-1},m_r]$, $r\in\{1,\dots,s\}$},
}
where $0 = m_0<m_1<\cdots<m_s=1$ and $0\preceq\gamma_1\prec\cdots\prec\gamma_s=\diag(d_\bal)$.
This means $\mu_{M,\,j}$ has the form
\eq{
\mu_{M,\,j} = \sum_{r=1}^s (m_r-m_{r-1})\delta_{q_r}, \quad \text{where} \quad \Phi^\star(q_r) = \gamma_r.
}
Now let $\nu$ be the RPC from Theorem~\ref{rpcthm}, so that when $(\alpha^1,\alpha^2,\dots)$ is sampled from $\E(\nu^{\otimes\infty})$, the induced array $(q_{r(\alpha^\ell,\alpha^{\ell'})})_{\ell,\ell'\geq1}$ has law $\bar\LL_{\mu_{M,\,j}}$.
Applying $\Phi^\star$ to every entry, we obtain an array of matrices $(\gamma_{r(\alpha^\ell,\alpha^{\ell'})})_{\ell,\ell'\geq1}$ whose law is $\bar\LL_{\mu_{M,\,j}}\circ(\Phi^\star)^{-1}$.
Then \eqref{dcj7} follows from Lemma~\ref{PPPPsi_lemma}.

By Property~\ref{pi_prop_iii} and the final statement in Theorem~\ref{representation_thm}, 
$\bar\LL_{\mu_{M,\,j}}$ converges weakly to $\bar\LL_{\mu_M}$ as $j\to\infty$.
Since $\Phi^\star$ is continuous, it follows that 
\eeq{ \label{jfnc3}
\text{$\bar\LL_{\mu_{M,\,j}}\circ(\Phi^\star)^{-1}$ converges weakly to $\bar\LL_{\mu_{M}}\circ(\Phi^\star)^{-1}\stackref{pushforward}{=}\LL_M$ as $j\to\infty$}.
}
Since \eqref{jfnc3} and \eqref{fjgx2} have the same limit, Corollary~\ref{extension_cor} implies
\eeq{ \label{fjcq8}
&\lim_{j\to\infty}\Psi_{\kappa M,\xi_{\bbeta}}\Big(\Law\big(\RR;\E(G_{\kappa N_j,\kappa M,\xi_{\bbeta}}^{\otimes\infty})\big);\Sigma^{\kappa M}(d_\bal)\Big) \\
&=\lim_{j\to\infty}\Psi_{\kappa M,\xi_{\bbeta}}(\bar\LL_{\mu_{M,\,j}}\circ(\Phi^\star)^{-1};\Sigma^{\kappa M}(d_\bal)).
}
Now we put our various observations together:
\eq{
\liminf_{N\to\infty} \eff_{\kappa N,\xi_{\bbeta}}(d_\bal) 
&\stackref{iwp0}{\geq} \lim_{j\to\infty}\Psi_{\kappa M,\xi_{\bbeta}}\Big(\Law\big(\RR;\E(G_{\kappa N_j,\kappa M,\xi_{\bbeta}}^{\otimes\infty})\big);\Sigma^{\kappa M}(d_\bal)\Big) \\
&\stackref{fjcq8}{=} \lim_{j\to\infty} \Psi_{\kappa M,\xi_{\bbeta}}(\bar\LL_{M,\,j}\circ(\Phi^\star)^{-1};\Sigma^{\kappa M}(d_\bal)) \\
&\stackref{dcj7}{=}  
\lim_{j\to\infty}\big[\PPP_{\kappa M,\xi_{\bbeta}}^{(1)}(\pi_{M,\,j},0;\Sigma^{\kappa M}(d_\bal))
+ \PPP_{\xi_{\bbeta}}^{(2)}(\pi_{M,\,j})\big] \\
&\stackref{qspd5}{=} \PPP_{\kappa M,\xi_{\bbeta}}(\pi_M,0;\Sigma^{\kappa M}(d_\bal))
+ \PPP^{(2)}_{\xi_{\bbeta}}(\pi_M).
}
Finally, we must send $M\to\infty$.

By passing to a subsequence, we may assume $\mu_M$ converges weakly as $M\to\infty$ to some probability measure $\mu$ on $[0,1]$.
Define $\pi\colon(0,1]\to\Gamma^\star$ by $\pi(t) = \Phi^\star(Q_\mu(t))$.
As before, $\pi$ belongs to $\Pi^\star$ because of the left-continuity of $Q_\mu$ together with the continuity of $\Phi^\star$.
Using the same logic as in \eqref{vgwe}, we must have $\pi_M\to\pi$ in the $L^1$ norm \eqref{norm}.
Hence Lemmas~\ref{pr456} and~\ref{pr457} give
\eq{
\lim_{M\to\infty}\PPP_{\kappa M,\xi_{\bbeta}}^{(1)}(\pi_M,0;\Sigma^{\kappa M}(d_\bal))
\stackref{xjq32}{=} \inf_{\lambda\in\R^\kappa}[\PPP_{\xi_{\bbeta}}^{(1)}(\pi,\lambda) - \iprod{\lambda}{d_\bal}]
\stackref{xjq33}{=}\PPP_{\xi_{\bbeta}}^{(1)}(\pi,0).
}
Applying Lemma~\ref{pconlem} once more, we also have
\eq{
\lim_{M\to\infty}\PPP^{(2)}_{\xi_{\bbeta}}(\pi_M) = \PPP^{(2)}_{\xi_{\bbeta}}(\pi).
}
Combining the three previous displays, we obtain
\begin{equation*}
\liminf_{N\to\infty} \eff_{\kappa N,\xi_{\bbeta}}(d_\bal) 
\geq \PPP_{\xi_{\bbeta}}^{(1)}(\pi,0)+\PPP_{\xi_{\bbeta}}^{(2)}(\pi)
\stackref{eq:Parisi:ftl:lagrange}{=}\PPP_{\xi_{\bbeta}}(\pi,0). \qedhere
\end{equation*}
\end{proof}



\appendix
\section{Gaussian processes for generic model}\label{sec:app:generic}
Here we prove Proposition~\ref{prop:xi:theta}. 
We fix $\theta=(p,m,n_1,\ldots, n_m, w_{1},\ldots, w_{m})\in \Theta$ throughout this appendix and denote the coordinates of $w_j\in [-1,1]^{\ka}$ by $w_j=\big(w_j(k)\big)_{k=1}^\kappa$.
We first establish two lemmas.
\begin{lemma}\label{lem:xi:theta:increasing}
    For $Q,R\in \Gamma_{\ka}$ such that $Q\preceq R$, we have 
    \begin{equation*}
        0\leq \xi_{\theta}(Q)\leq\xi_{\theta}(R),\;\;\;\;0\preceq\nabla\xi_{\theta}(Q)\preceq\nabla\xi_{\theta}(R).
    \end{equation*}
    Furthermore, $\vartheta_{\theta}(R)=(\deg(\theta)-1)\xi_\theta(R)$ 
    for any $R\in \R^{\ka\times \ka}$.
\end{lemma}
\begin{proof}
First we make an elementary claim about Hadamard products: 
\eeq{ \label{elem_hada}
0\preceq Q\preceq R,\ 0\preceq Q'\preceq R' \quad \implies \quad 0\preceq Q\circ Q' \preceq R\circ R'.
}
Indeed, we have the decomposition
\eq{
R\circ R' - Q\circ Q'
= R\circ(R'-Q') + Q'\circ(R-Q),
}
and the Schur product theorem tells us that both $R\circ(R'-Q')$ and $Q'\circ(R-Q)$ are positive-semidefinite.
Hence \eqref{elem_hada} holds.

By repeatedly applying \eqref{elem_hada} with $Q'=Q$ and $R'=R$, we see that
\eq{ 
0\preceq Q\preceq R \quad \implies \quad
0\leq \iprod{Q^{\circ p}u}{u} \leq \iprod{R^{\circ p}u}{u} \quad \text{for any $p\geq 1$ and $u\in\R^\kappa$}.
}
In light of the definition of $\xi_\theta$ in \eqref{def:xi:theta}, it follows that $0\leq \xi_{\theta}(Q)\leq \xi_{\theta}(R)$.

Next we argue $0\preceq \nabla\xi_{\theta}(Q)\preceq \nabla\xi_{\theta}(R)$. 
By differentiating \eqref{def:xi:theta} with respect to $R_{k,k'}$, we compute the $(k,k^\prime)$ entry of $\nabla \xi_{\theta}(R)$ to be
\eeq{ \label{nablaxi_ab}
\nabla\xi_\theta(R)_{k,k^\prime}
= \sum_{j=1}^m\Big[\prod_{\ell \neq j}\iprod{R^{\circ p}w_\ell}{w_\ell}^{n_\ell}\Big] n_j\iprod{R^{\circ p}w_j}{w_j}^{n_j-1}\cdot pR_{k,k^\prime}^{p-1}w_j(k)w_j(k^\prime).
}
If we write $W_j=\diag(w_j)$, then \eqref{nablaxi_ab} says
\begin{equation}\label{eq:express:nabla:xi}
   \nabla\xi_\theta(R)=p\sum_{j=1}^{m}n_j\xi_{\theta_j}(R)\cdot W_j R^{\circ(p-1)}W_j,
\end{equation}
where $\theta_j\in \Theta$ is obtained from $\theta=(p,m,n_1,\ldots, n_m,w_1,\ldots, w_m)$ by modifying $n_j$ to $n_j-1$. 
For each $j$ we have $0\leq \xi_{\theta_j}(Q)\leq \xi_{\theta_j}(R)$ by the argument of the previous paragraph. 
In addition, \eqref{elem_hada} gives $0\preceq Q^{\circ (p-1)}\preceq R^{\circ(p-1)}$, and so $0\preceq WQ^{\circ (p-1)}W\preceq WR^{\circ(p-1)}W$ for any diagaonal matrix $W$.
It is thus apparent from \eqref{eq:express:nabla:xi} that $0\preceq \nabla\xi_{\theta}(Q)\preceq \nabla\xi_{\theta}(R)$.

To see the final assertion of the lemma, insert \eqref{nablaxi_ab} into the definition of $\vartheta_\theta(R)= \iprod{R}{\nabla\xi_{\theta}(R)}-\xi_{\theta}(R)$.
This results in
\eq{
\vartheta_\theta(R) 
=p\sum_{j=1}^m n_j \Big[\prod_{\ell \neq j}\iprod{R^{\circ p}w_{\ell}}{w_\ell}^{n_\ell}\Big] \iprod{R^{\circ p}w_j}{w_j}^{n_j}-\xi_{\theta}(R)
\stackref{def:xi:theta}{=} p\sum_{j=1}^m n_j\xi_\theta(R) - \xi_\theta(R),
}
which is exactly as desired.
\end{proof}

In the next lemma, we follow the construction in \cite[Sec.~5]{panchenko18b}. 
Our $H_{N,\theta}$ is equal to $\sqrt{N}\,h_{N,\theta}$ in the notation of \cite{panchenko18b}.

\begin{lemma}\label{lem:gaussian:existence}
For each $N\geq 1$, there exist a centered Gaussian process $\big(H_{N,\theta}(\sigma)\big)_{\sigma\in \R^{\kappa\times N}}$ with covariance
    \eeq{ \label{lem:gaussian:existence_1}
   \E\big[H_{N,\theta}(\sigma)H_{N,\theta}(\sigma')\big]
     = N\xi_\theta\Big(\frac{\sigma\sigma'^{\sT}}{N}\Big),
     }
     and a centered $\ka$-dimensional Gaussian process $\big(Z_{N,\theta}(\sigma)\big)_{\sigma\in \R^{\kappa\times N}}$ with covariance
    \begin{equation}\label{eq:Z:theta:cov}
        \E\Big[Z_{N,\theta}(\sigma)Z_{N,\theta}(\sigma')^{\sT}\Big]
     = \nabla \xi_\theta\Big(\frac{\sigma\sigma'^{\sT}}{N}\Big).
    \end{equation}
\end{lemma}

\begin{proof}
For any $p$-tuple of indices ${I}=(i_1,\dots,i_p)\in\{1,\dots,N\}^p$, let us write
\eq{
\sigma_{{I}} = (\sigma_{i_1},\dots,\sigma_{i_p})\in(\R^\kappa)^p \quad \text{for $\sigma\in(\R^\kappa)^N$}.
}
Furthermore, for any $n$-tuple of $p$-tuples $\II = ({I}_1,\dots,{I}_n)\in(\{1,\dots,N\}^p)^n$, we will write
\eq{
\sigma_{\II} = (\sigma_{{I}_1},\dots,\sigma_{{I}_n})\in((\R^\kappa)^p)^n \quad \text{for $\sigma\in(\R^\kappa)^N$}.
}
Given such $\II$ and some $w= \big(w(k)\big)_{1\leq k\leq \ka}\in\R^\kappa$, define the following polynomial:
\eq{
S_w(\sigma_\II) &\coloneqq S_w(\sigma_{{I}_1})\cdots S_w(\sigma_{{I}_n}), \quad \text{where} \quad
S_w(\sigma_{i_1},\ldots \sigma_{i_p}) \coloneqq \sum_{k=1}^\kappa w(k)\sigma_{i_1}(k)\cdots\sigma_{i_p}(k).
}
Finally, we sum over all choices of $\II_1,\dots,\II_m$ with $\II_j\in(\{1,\dots,N\}^p)^{n_j}$:
\eq{
H_{N,\theta}(\sigma) \coloneqq \frac{1}{N^{(\dgr(\theta)-1)/2}}\sum_{\II_1,\dots,\II_m}g_{\II_1,\dots,\II_m}S_{w_1}(\sigma_{\II_1})\cdots S_{w_m}(\sigma_{\II_m}),
}
where each $g_{\II_1,\dots,\II_m}$ is an independent standard normal random variable. 
Recall that $\deg(\theta)=p\sum_{j=1}^m n_j$.
The covariance of the Gaussian process $\big(H_{N,\theta}(\sigma)\big)_{\sigma\in \R^{\ka\times N}}$ is then
\eq{
\E\big[H_{N,\theta}(\sigma)H_{N,\theta}(\sigma')\big]
&=N\prod_{j=1}^{m}\bigg[\frac{1}{N^{pn_j}}\sum_{\II_j}S_{w_j}(\sigma_{\II_j})S_{w_j}(\sigma^\prime_{\II_j})\bigg]\\
&=N\prod_{j=1}^{m}\bigg(\frac{1}{N^{p}}\sum_{I\in \{1,\ldots, N\}^p}S_{w_j}(\sigma_{I})S_{w_j}(\sigma^\prime_{I})\bigg)^{n_j}.
}
If we write $R$ for the matrix $R(\sigma,\sigma')=\sigma\sigma'^\sT/N$, then
\eq{
    \frac{1}{N^{p}}\sum_{I\in \{1,\ldots, N\}^p}S_{w_j}(\sigma_{I})S_{w_j}(\sigma^\prime_{I})&=\sum_{k,k^\prime=1}^{\ka}w_j(k)w_j(k^\prime)\bigg(\frac{1}{N}\sum_{i=1}^{N}\sigma_i(k)\sigma'_i(k^\prime)\bigg)^p\\
    &= \iprod{R^{\circ p} w_j}{w_j}.
}
By combining the two previous displays and recalling the definition of $\xi_\theta$ from \eqref{def:xi:theta}, we obtain \eqref{lem:gaussian:existence_1}.

To prove the existence of the Gaussian process $Z_{N,\theta}$, we first claim that for any $p\geq 1$, there exists a $\ka$-dimensional centered Gaussian process $\big(z_{N,p}(\sigma)\big)_{\sigma \in \R^{\ka\times N}}$ with covariance
\begin{equation}\label{eq:cov:z:p}
    \E\Big[z_{N,p}(\sigma)z_{N,p}(\sigma^\prime)^{\sT}\Big]=\Big(\frac{\sigma\sigma'^{\sT}}{N}\Big)^{\circ (p-1)}.
\end{equation}
If $p=1$, then the right-hand side of \eqref{eq:cov:z:p} is interpreted as the $\kappa\times\kappa$ identity matrix, and so it suffices to take $z_{N,p}(\sigma)$ equal to a standard normal random vector not depending on $\sigma$.
If $p\geq2$, then we make the following construction. 
For $k\in\{1,\dots,\kappa\}$ define
\begin{equation*}
    z_{N,p,k}(\sigma)=\frac{1}{N^{p-1}}\sum_{i_1,\ldots i_{p-1}=1}^{N}g_{i_1,\ldots, i_{p-1}}\sigma_{i_1}(k)\cdots \sigma_{i_{p-1}}(k),
\end{equation*}
where each $g_{i_1,\ldots, i_{p-1}}$ is an independent standard normal random variable. 
By setting $z_{N,p}(\sigma)= \big(z_{N,p,k}(\sigma)\big)_{k=1}^\kappa$, we obtain a $\ka$-dimensional Gaussian process satisfying \eqref{eq:cov:z:p}.

Now let $z_{N,p}^1,\dots,z_{N,p}^m$ 
be independent copies of $z_{N,p}$.
As in the proof of Lemma~\ref{lem:xi:theta:increasing}, let $\theta_j\in \Theta$ be obtained from $\theta=(p,m,n_1,\ldots, n_m,w_1,\ldots, w_m)$ by modifying $n_j$ to $n_j-1$.
From \eqref{lem:gaussian:existence_1}, consider centered Gaussian processes $H_{N,\theta_1},\dots,H_{N,\theta_m}$, which we assume are independent of each other and of each $z_{N,p}^1,\dots,z_{N,p}^m$.
Finally, recalling the notation $W_j = \diag(w_j)$, we define
\begin{equation*}
Z_{N,\theta}(\sigma)\coloneqq p^{1/2}\sum_{j=1}^{m}n_j^{1/2}N^{-1/2}H_{N,\theta_j}(\sigma)\cdot W_j z_{N,p}^{j}(\sigma).
\end{equation*}
It follows from \eqref{eq:express:nabla:xi} and \eqref{eq:cov:z:p} that the covariance of $\big(Z_{N,\theta}(\sigma)\big)_{\sigma\in \R^{\kappa\times N}}$ is given by \eqref{eq:Z:theta:cov}.
\end{proof}

\begin{proof}[Proof of Proposition~\ref{prop:xi:theta}.]
Parts~\ref{prop:xi:theta_a} and~\ref{prop:xi:theta_b} follow immediately from Lemma~\ref{lem:xi:theta:increasing}. Part~\ref{prop:xi:theta_c} follows from Lemma~\ref{lem:gaussian:existence}. 
Finally, part~\ref{prop:xi:theta_d} follows from combining Lemmas~\ref{lem:xi:theta:increasing} and~\ref{lem:gaussian:existence} since we can set $Y_{N,\theta}(\sigma)=\big(\frac{\deg(\theta)-1}{N}\big)^{1/2}H_{N,\theta}(\sigma)$.
\end{proof}

The following lemma was used in Remark~\ref{rmk_compare_defs}.

\begin{lemma} \label{lem_dominance}
For any $d\in\DD$ and any $\gamma$ belonging to the set $\Gamma_\kappa(d)$ from \eqref{Gamma_d_def}, we have $\gamma\preceq\diag(d)$.
\end{lemma}

\begin{proof}
For any $u=(u_1,\dots,u_\kappa)\in\R^\kappa$, we have
\eq{
\iprod{\diag(d)u}{u} - \iprod{\gamma u}{u}
&\stackrefp{Gamma_d_def}{=} \sum_{k=1}^\kappa d_ku_k^2 - \sum_{k,k'=1}^\kappa \gamma_{k,k'}u_ku_{k'} \\
&\stackref{Gamma_d_def}{=} \sum_{k,k'=1}^\kappa \gamma_{k,k'}(u_k^2 - u_ku_{k'}) \\
&\stackrefp{Gamma_d_def}{=}\sum_{k<k'}\gamma_{k,k'}(u_k^2 + u_{k'}^2 - 2u_ku_{k'})
= \sum_{k<k'}\gamma_{k,k'}(u_k-u_{k'})^2 \geq 0,
}
where the final inequality uses the fact that every entry $\gamma_{k,k'}$ is nonnegative.
\end{proof}

\section{Proof of the Aizenman--Sims--Starr scheme} \label{sec_ass_proof}
In this appendix, we prove Proposition~\ref{ass_prop}.
To begin, we deduce from \eqref{perfect_cav} the inclusion
\eeq{ \label{cav_inclusion}
\Sigma^{N+M}(d^{N+M}) \supseteq \Sigma^N(d^N) \times \Sigma^M(\delta^M).
}
Indeed, for $\sigma = (\sigma_1,\dots,\sigma_N)\in\Sigma^N(d^N)$ and $\tau=(\tau_1,\dots,\tau_M)\in\Sigma^M(\delta^M)$, let us write $(\sigma,\tau)$ for the $\kappa\times (N+M)$ matrix whose first $N$ columns are $\sigma_1,\dots,\sigma_N$ and whose last $M$ columns are $\tau_1,\dots,\tau_M$.
In this notation, we have
\eq{
(\sigma,\tau)(\sigma,\tau)^\sT = \sigma\sigma^\sT + \tau\tau^\sT
= N\diag(d^N) + M\diag(\delta^M)
\stackref{perfect_cav}{=} (N+M)\diag(d^{N+M}),
}
which by definition means $(\sigma,\tau)\in\Sigma^{N+M}(d^{N+M})$.
From \eqref{cav_inclusion}, it immediately follows that
\eeq{ \label{inclusion_lower}
\zee_{N+M,\tilde\xi}(d^{N+M}) \geq \sum_{\sigma\in\Sigma^N(d^N)}\sum_{\tau\in\Sigma^M(\delta^M)}\exp H_{N+M,\tilde\xi}(\sigma,\tau).
}
Recall from \eqref{general_cov} that
\eq{ 
\E[H_{N+M,\tilde\xi}(\sigma,\tau)H_{N+M,\tilde\xi}(\sigma',\tau')]
&= (N+M)\tilde\xi\Big(\frac{(\sigma,\tau)(\sigma',\tau')^\sT}{N+M}\Big) \\
&= (N+M)\tilde\xi\Big(\frac{\sigma\sigma'^\sT+\tau\tau'^\sT}{N+M}\Big).
}
Replacing $\tilde\xi$ with $\xi$ on the right-hand side incurs an error:
\eeq{ \label{ghd3}
\Big|\E[H_{N+M,\tilde\xi}(\sigma,\tau)H_{N+M,\tilde\xi}(\sigma',\tau')] &- (N+M)\xi\Big(\frac{\sigma\sigma'^\sT+\tau\tau'^\sT}{N+M}\Big)\Big| \\
&\leq (N+M)\sup_{\|R\|_1\leq1}|\xi(R)-\tilde\xi(R)|.
}
By smoothness (see Remark~\ref{rmk_xi}), we have the following linearization about $(\sigma\sigma'^\sT)/(N+M)$:
\eq{
\xi\Big(\frac{\sigma\sigma'^\sT+\tau\tau'^\sT}{N+M}\Big)
&= \xi\Big(\frac{\sigma\sigma'^\sT}{N+M}\Big)
+ \iprod[\Big]{\nabla\xi\Big(\frac{\sigma\sigma'^\sT}{N+M}\Big)}{\frac{\tau\tau'^\sT}{N+M}} 
+ O_{\nabla^2\xi}\Big(\Big(\frac{M}{N+M}\Big)^2\Big),
}
where we are using the notation from \eqref{gw4hg4}.
Furthermore, we have the approximation
\eq{
\nabla\xi\Big(\frac{\sigma\sigma'^\sT}{N+M}\Big)
&= \nabla\xi\Big(\frac{\sigma\sigma'^\sT}{N}\Big)
+ O_{\nabla^2\xi}\Big(\frac{M}{N+M}\Big).
}
Combining the two previous displays, we arrive to
\eeq{ \label{after_1st_linearization}
(N+M)\xi\Big(\frac{\sigma\sigma'^\sT+\tau\tau'^\sT}{N+M}\Big)
= (N+M)\xi\Big(\frac{\sigma\sigma'^\sT}{N+M}\Big)
+ \iprod[\Big]{\nabla\xi\Big(\frac{\sigma\sigma'^\sT}{N}\Big)}{\tau\tau'^\sT} + O_{\nabla^2\xi}\Big(\frac{M^2}{N+M}\Big).
}
Regarding the first term on the right-hand side, we have the further linearization
\eq{
&(N+M)\xi\Big(\frac{\sigma\sigma'^\sT}{N+M}\Big) \\
&= (N+M)\Big[\xi\Big(\frac{\sigma\sigma'^\sT}{N}\Big) - \iprod[\Big]{\nabla\xi\Big(\frac{\sigma\sigma'^\sT}{N}\Big)}{\frac{M\sigma\sigma'^\sT}{N(N+M)}} + O_{\nabla^2\xi}\Big(\Big(\frac{M}{N+M}\Big)^2\Big)\Big] \\
&= N\xi\Big(\frac{\sigma\sigma'^\sT}{N}\Big) - M\Big[\iprod[\Big]{\nabla\xi\Big(\frac{\sigma\sigma'^\sT}{N}\Big)}{\frac{\sigma\sigma'^\sT}{N}} - \xi\Big(\frac{\sigma\sigma'^\sT}{N}\Big)\Big] + O_{\nabla^2\xi}\Big(\frac{M^2}{N+M}\Big).
}
The term being subtracted on the final line is simply $M\vartheta_\xi\big(\frac{\sigma\sigma'^\sT}{N}\big)$.
Moving this term to the left-hand side, we arrive to
\eeq{ \label{after_2nd_linearization}
(N+M)\xi\Big(\frac{\sigma\sigma'^\sT}{N+M}\Big)
+ M\vartheta_\xi\Big(\frac{\sigma\sigma'^\sT}{N}\Big)
=N\xi\Big(\frac{\sigma\sigma'^\sT}{N}\Big) + O_{\nabla^2\xi}\Big(\frac{M^2}{N+M}\Big).
}
Equipped with the approximations \eqref{after_1st_linearization} and \eqref{after_2nd_linearization}, we resume our probabilistic argument. 

Let $H_{N,M,\xi}$ be the cavity Hamiltonian from \eqref{comma_covar}.
Let $Z_1,\dots,Z_M\colon\Sigma^N\to\R^\kappa$ and $Y\colon\Sigma^N\to\R$ be centered Gaussian processes with covariances
\eeq{ \label{opdqg}
\E\Big[Z_{i}(\sigma)Z_{i}(\sigma')^\sT\Big] &= 
\nabla\xi\Big(\frac{\sigma\sigma'^\sT}{N}\Big), \qquad
\E[Y(\sigma)Y(\sigma')] = 
\vartheta_{\xi}\Big(\frac{\sigma\sigma'^\sT}{N}\Big).
}
Such processes exist by Proposition~\hyperref[prop:xi:theta_d]{\ref*{prop:xi:theta}\ref*{prop:xi:theta_d}} and assumption~\ref{xi_power}.
We assume all these processes are independent of each other and of $H_{N,M,\xi}$.
Now write
\eeq{ \label{initial_ineq_from_inclusion}
\frac{\zee_{N+M,\tilde\xi}(d^{N+M})}{\zee_{N,\xi}(d^N)}
&\stackref{inclusion_lower}{\geq} \frac{\sum_{\sigma\in\Sigma^N(d^N)}\sum_{\tau\in\Sigma^M(\delta^M)}\exp H_{N+M,\tilde\xi}(\sigma,\tau)}{\sum_{\sigma\in\Sigma^N(d^N)}\exp H_{N,\xi}(\sigma)}
= \frac{Q_1Q_1^{\err}}{Q_2Q_2^{\err}},
}
where $Q_1,Q_1^\err,Q_2,Q_2^\err$ are the following quotients:
\begin{align}
Q_1 &\coloneqq \notag
\frac{\sum_{\sigma\in\Sigma^N(d^N)}\sum_{\tau\in\Sigma^M(\delta^M)}\exp(H_{N,M,\xi}(\sigma)+\sum_{j=1}^M \iprod[\big]{Z_j(\sigma)}{\tau_j})}{\sum_{\sigma\in\Sigma^N(d^N)}\exp H_{N,M,\xi}(\sigma)}, \\
Q_1^\err &\coloneqq \label{Q1err_def}
\frac{\sum_{\sigma\in\Sigma^N(d^N)}\sum_{\tau\in\Sigma^M(\delta^M)}\exp H_{N+M,\tilde\xi}(\sigma,\tau)}{\sum_{\sigma\in\Sigma^N(d^N)}\sum_{\tau\in\Sigma^M(\delta^M)}\exp(H_{N,M,\xi}(\sigma)+\sum_{j=1}^M \iprod[\big]{Z_j(\sigma)}{\tau_j})}, \\
Q_2 &\coloneqq \notag
\frac{\sum_{\sigma\in\Sigma^N(d^N)}\exp (H_{N,M,\xi}(\sigma)+\sqrt{M}\, Y(\sigma))}{\sum_{\sigma\in\Sigma^{N}(d^N)}\exp H_{N,M,\xi}(\sigma)}, \\
Q_2^\err &\coloneqq \label{Q2err_def}
\frac{\sum_{\sigma\in\Sigma^N(d^N)}\exp H_{N,\xi}(\sigma)}{\sum_{\sigma\in\Sigma^{N}(d^N)}\exp(H_{N,M,\xi}(\sigma)+\sqrt{M}\, Y(\sigma))}.
\end{align}
Observe that $Q_1$ and $Q_2$ have the desired form for realizing the functional $\Psi_{M,\xi}$ from \eqref{Psi_def}.
Namely, if $\langle\cdot\rangle_N$ denotes expectation with respect to the measure $G_{N,M,\xi}$ from \eqref{GNM_def}, then
\eeq{ \label{desired_fun} \raisetag{1.5\baselineskip}
\E\log Q_1 = \E\log\sum_{\tau\in\Sigma^M(\delta^M)}\Big\langle\exp\Big(\sum_{j=1}^M\iprod[\big]{Z_j(\sigma)}{\tau_j}\Big)\Big\rangle_N
&\stackref{psi_1_def}{=}M\Psi_{M,\xi}^{(1)}(\LL_{N,M};d^N,\Sigma^M(\delta^M)), \\
\text{while} \quad \E\log Q_2 = \E\log\big\langle\exp(\sqrt{M}\, Y(\sigma))\big\rangle_N
&\stackref{psi_2_def}{=}M\Psi_{M,\xi}^{(2)}(\LL_{N,M}).
}
Therefore, the remainder of the proof is to show that
\begin{subequations} \label{Qerr_conclusion}
\begin{align}
|\E\log Q_1^\err| \label{Qerr_conclusion_a}
&\leq O_{\nabla^2\xi}\Big(\frac{M^2}{N+M}\Big)+\frac{N+M}{2}\sup_{\|R\|_1\leq1}|\xi(R)-\tilde\xi(R)|, \\ \qquad \text{and} \qquad \label{Qerr_conclusion_b}
\E\log Q_2^\err 
&= O_{\nabla^2\xi}\Big(\frac{M^2}{N+M}\Big).
\end{align}
\end{subequations}
Indeed, the claim \eqref{ass_lower} follows from using \eqref{desired_fun} and \eqref{Qerr_conclusion} in the initial inequality \eqref{initial_ineq_from_inclusion}.

We now argue each estimate in \eqref{Qerr_conclusion} separately, although the two arguments are very similar.
Since the one for $Q_2^\err$ is slightly simpler, we begin there.
\medskip

\noindent \textbf{Control of $Q_2^\err$.}
Define an interpolating Hamiltonian on $\Sigma^N(d^N)$:
\eeq{ \label{Q2_interpolate}
\HH_t(\sigma) = \sqrt{t}\, H_{N,\xi}(\sigma) + \sqrt{1-t}\big(H_{N,M,\xi}(\sigma)+\sqrt{M}\, Y(\sigma)\big), \quad t\in[0,1].
}
Consider the free energy associated to $\HH_t$:
\eeq{ \label{Q1_phi_def}
\phi(t) = \E\log \sum_{\sigma\in\Sigma^N(d^N)}\exp \HH_t(\sigma).
}
Recalling the definition of $Q_2^\err$ from \eqref{Q2err_def}, we see that
\eeq{ \label{Q2_difference}
\E\log Q_2^\err = \phi(1) - \phi(0).
}
We control this difference by studying the derivative of \eqref{Q1_phi_def} with respect to $t$, which is easily computed to be
\eq{
\phi'(t) = \E\Big\langle\frac{\partial}{\partial t}\HH_t(\sigma,\tau)\Big\rangle_t,
}
where $\langle\cdot\rangle_t$ denotes expectation with respect to the Gibbs measure $G_t(\sigma)\propto \exp\HH_t(\sigma)$ on $\Sigma^N(d^N)$.
Thanks to Gaussian integration by parts \cite[Lem.~1.1]{panchenko13a}, we can rewrite this as
\eq{
\phi'(t) &= \E\big\langle \CC(\sigma^1,\sigma^1) - \CC(\sigma^1,\sigma^2)\big\rangle_t, 
}
where $\sigma^1,\sigma^2$ denote independent samples from $G_t$, and $\CC\colon\Sigma^N(d^N)\times\Sigma^N(d^N)\to\R$ is defined by
\eeq{ \label{Q2_covar}
&\CC(\sigma,\sigma') 
= \E\Big[\frac{\partial\HH_t(\sigma)}{\partial t}\HH_t(\sigma')\Big].
}
Since the Gaussian processes on the right-hand side of \eqref{Q2_interpolate} are independent and centered, the right-hand side of \eqref{Q2_covar} reduces to a linear combination of their respective covariances:
\eq{
&\CC(\sigma,\sigma')
\stackref{Q2_interpolate}{=}
\frac{1}{2}\E[H_{N,\xi}(\sigma)H_{N,\xi}(\sigma')] - \frac{1}{2}\E[H_{N,M,\xi}(\sigma)H_{N,M,\xi}(\sigma')] - \frac{M}{2}\E[Y(\sigma)Y(\sigma')] \\
&\hspace{0.19in}\stackref{general_cov,comma_covar,opdqg}{=}
\frac{1}{2}\Big[N\xi\Big(\frac{\sigma\sigma'^\sT}{N}\Big)-(N+M)\xi\Big(\frac{\sigma\sigma'^\sT}{N+M}\Big) - M\vartheta_\xi\Big(\frac{\sigma\sigma'^\sT}{N}\Big)\Big] \\
&\hspace{0.19in}\stackrefpp{after_2nd_linearization}{general_cov,comma_covar,opdqg}{=} O_{\nabla^2\xi}\Big(\frac{M^2}{N+M}\Big).
}
Reading the three previous displays together, we have established the following estimate:
\eq{
\sup_{t\in(0,1)}|\phi'(t)| = O_{\nabla^2\xi}\Big(\frac{M^2}{N+M}\Big).
}
Therefore, \eqref{Q2_difference} leads to \eqref{Qerr_conclusion_b}.
\medskip

\noindent \textbf{Control of $Q_1^\err$.}
Define an interpolating Hamiltonian on $\Sigma^N(d^N)\times\Sigma^M(\delta^M)$:
\eeq{ \label{Q1_interpolate}
\HH_t(\sigma,\tau) = \sqrt{t}\, H_{N+M}(\sigma,\tau) + \sqrt{1-t}\Big(H_{N,M,\xi}(\sigma)+\sum_{j=1}^M \iprod[\big]{Z_j(\sigma)}{\tau_j}\Big), \quad t\in[0,1].
}
From \eqref{opdqg}, one can calculate the covariance of the final term on the right-hand side:
\eeq{ \label{sum_z_covar}
\E\Big[\sum_{j=1}^M\iprod[\big]{Z_j(\sigma)}{\tau_j}\iprod[\big]{Z_j(\sigma')}{\tau_j'}\Big]
= \iprod[\Big]{\nabla\xi\Big(\frac{\sigma\sigma'^\sT}{N}\Big)}{\tau\tau'^\sT}.
}
Now consider the free energy associated to $\HH_t$:
\eq{
\phi(t) = \E\log \sum_{\sigma\in\Sigma^N(d^N)}\sum_{\tau\in\Sigma^M(\delta^M)}\exp \HH_t(\sigma,\tau).
}
Recalling the definition of $Q_1^\err$ from \eqref{Q1err_def}, we see that
\eeq{ \label{Q1_difference}
\E\log Q_1^\err = \phi(1) - \phi(0).
}
We control this difference by studying the derivative of \eqref{Q1_phi_def} with respect to $t$, which is easily computed to be
\eq{
\phi'(t) = \E\Big\langle\frac{\partial}{\partial t}\HH_t(\sigma,\tau)\Big\rangle_t,
}
where $\langle\cdot\rangle_t$ denotes expectation with respect to the Gibbs measure $G_t(\sigma,\tau)\propto\exp\HH_t(\sigma,\tau)$ on $\Sigma^N(d^N)\times\Sigma^M(\delta^M)$. 
By Gaussian integration by parts \cite[Lem.~1.1]{panchenko13a}, we can rewrite this as
\eq{
\phi'(t) &= \E\Big\langle \CC\big((\sigma^1,\tau^1),(\sigma^1,\tau^1)\big) - \CC\big((\sigma^1,\tau^1),(\sigma^2,\tau^2)\big)\Big\rangle_t, 
}
where $(\sigma^1,\tau^1),(\sigma^2,\tau^2)$ denote independent samples from $G_t$, and $\CC\colon(\Sigma^N(d^N)\times\Sigma^M(\delta^M))^2\to\R$ is defined by
\eeq{ \label{Q1_covar}
&\CC\big((\sigma,\tau),(\sigma',\tau')\big) 
= \E\Big[\frac{\partial\HH_t(\sigma,\tau)}{\partial t}\HH_t(\sigma',\tau')\Big].
}
Since the Gaussian processes on the right-hand side of \eqref{Q1_interpolate} are independent and centered, the right-hand side of \eqref{Q1_covar} reduces to a linear combination of their respective covariances:
\eq{
&\bigg|\E\Big[\frac{\partial\HH_t(\sigma,\tau)}{\partial t}\HH_t(\sigma',\tau')\Big]\bigg|
\stackref{Q1_interpolate}{=} \bigg|\frac{1}{2}\E[H_{N+M,\tilde\xi}(\sigma,\tau)H_{N+M,\tilde\xi}(\sigma',\tau')] \\
&\hspace{1.8in}- \frac{1}{2}\E[H_{N,M,\xi}(\sigma)H_{N,M,\xi}(\sigma')] - \frac{1}{2}\E\Big[\sum_{j=1}^M\iprod[\big]{Z_j(\sigma)}{\tau_j}\iprod[\big]{Z_j(\sigma')}{\tau'_j}\Big]\bigg| \\
&\stackref{ghd3,comma_covar,sum_z_covar}{\leq} \frac{N+M}{2}\sup_{\|R\|_1\leq1}|\xi(R)-\tilde\xi(R)| + \frac{1}{2}\bigg|(N+M)\xi\Big(\frac{\sigma\sigma'^\sT+\tau\tau'^\sT}{N+M}\Big) \\
&\hspace{2in}- (N+M)\xi\Big(\frac{\sigma\sigma'^\sT}{N+M}\Big)- \iprod[\Big]{\nabla\xi\Big(\frac{\sigma\sigma'^\sT}{N}\Big)}{\tau\tau'^\sT}\bigg| \\
&\stackrefpp{after_1st_linearization}{ghd3,comma_covar,sum_z_covar}{=} \frac{N+M}{2}\sup_{\|R\|_1\leq1}|\xi(R)-\tilde\xi(R)|+ O_{\nabla^2\xi}\Big(\frac{M^2}{N+M}\Big).
}
Reading the three previous displays together, we have established the following estimate:
\eq{
\sup_{t\in(0,1)}|\phi'(t)| \leq O_{\nabla^2\xi}\Big(\frac{M^2}{N+M}\Big)+\frac{N+M}{2}\sup_{\|R\|_1\leq1}|\xi(R)-\tilde\xi(R)|.
}
Therefore, \eqref{Q1_difference} leads to \eqref{Qerr_conclusion_b}, which concludes the proof of Proposition~\ref{ass_prop}.

\section{Properties of Ruelle probability cascades} \label{sec_app_rpc}
This appendix proves Lemmas~\ref{PPPPsi_lemma} and~\ref{lem_parisi_recovered}.
%
%
Both proofs rely on the following fact. 

\begin{theirthm}\label{thm:RPC}
\textup{\cite[Thm.~14.2.1]{talagrand11b}}
Let $\eta^{(0)},\dots,\eta^{(s-1)}$ be independent random variables taking values in a metric space $T$. 
Assume $f\colon T^{s}\to\R$ is a deterministic function such that
\eq{
\E\exp f(\eta^{(0)},\eta^{(1)},\dots,\eta^{(s-1)}) < \infty \quad \text{and} \quad
\E|f(\eta^{(0)},\eta^{(1)},\dots,\eta^{(s-1)})| < \infty.
}
Let $\E_r$ denote expectation with respect to $\eta^{(r)}$.
Given the sequence \eqref{eq:m2}, inductively define
\eeq{ \label{following_just_zs}
X_s &\coloneqq f(\eta^{(0)},\eta^{(1)},\dots,\eta^{(s-1)}), \\
X_r &\coloneqq \frac{1}{m_r}\log \E_r\exp(m_r X_{r+1}) \quad \text{for $r\in\{1,\dots,s-1\}$}, \\
X_0 &\coloneqq \E_0(X_1).
}
On the other hand, let $\nu$ be the RPC associated to \eqref{eq:m2}.
For each $r\in\{0,\dots,s-1\}$, let $(\eta_\beta)_{\beta\in\N^r}$ be i.i.d.~copies of $\eta^{(r)}$ that are independent of $\nu$.
We then have
\eeq{ \label{magic}
X_0 = \E\log\sum_{\alpha\in\N^{s-1}}\nu_\alpha \exp f(\eta_\varnothing,\eta_{(\alpha_1)},\eta_{(\alpha_1,\alpha_2)},\dots,\eta_{(\alpha_1,\dots,\alpha_{s-1})}).
}
\end{theirthm}

\begin{proof}[Proof of Lemma~\ref{PPPPsi_lemma}]
The right-hand sides of \eqref{PPPPsi} are defined in \eqref{defining_Psi_parts} with
\eeq{ \label{rpc_setting}
\XX = \N^{s-1}, 
\qquad R(\alpha,\alpha') = \gamma_{r(\alpha,\alpha')}, 
\qquad \GG = \nu.
}
By inspection, \eqref{psi_1_def} is the same as \eqref{pre_par_a} with $\lambda=0$.
This proves \eqref{PPPPsi_a}.

Meanwhile, in the setting \eqref{rpc_setting}, definition \eqref{psi_2_def} becomes
\begin{subequations}
\label{xf573}
\eeq{
\Psi_{N,\xi}^{(2)}(\LL; d) = \frac{1}{N}\E\log\sum_{\alpha\in\N^{s-1}}\nu_\alpha\exp\big(\sqrt{N}\, Y(\alpha)\big).
}
Using the representation \eqref{xm334} of $Y$, we can rewrite this as
\eq{
\Psi_{N,\xi}^{(2)}(\LL; d) = \frac{1}{N}\E\log\sum_{\alpha\in\N^{s-1}}\nu_\alpha\exp\Big(\sqrt{N}\sum_{r=0}^{s-1}\sqrt{\vartheta_\xi(\gamma_{r+1})-\vartheta_\xi(\gamma_r)}\, \eta_{(\alpha_1,\dots,\alpha_r)}\Big).
}
By \eqref{magic}, the right-hand side can be transformed to yield
\eeq{ \label{cwje3}
\Psi_{N,\xi}^{(2)}(\LL;d)
= \frac{1}{N} X_0,
}
where $X_0, X_1,\dots,X_s$ are related inductively as in \eqref{following_just_zs}, and
\eq{
X_s = \sqrt{N}\sum_{r=0}^{s-1}\sqrt{\vartheta_\xi(\gamma_{r+1})-\vartheta_\xi(\gamma_r)}\, \eta^{(r)}.
}
Here $\eta^{(0)},\dots,\eta^{(s-1)}$ are i.i.d.~standard normal random variables.
Therefore, using the fact that $\E\exp(c\eta^{(r)}) = \frac{1}{2}\exp(c^2)$ for any $c\in\R$, it is straightforward to calculate
\eq{
X_0 = \frac{N}{2}\sum_{r=1}^{s-1}m_r\big(\vartheta_\xi(\gamma_{r+1})-\vartheta_\xi(\gamma_r)\big).
}
Rewriting the right-hand side using summation by parts, we obtain
\eeq{
X_0 &= \frac{N}{2}\vartheta_\xi(\gamma_s)-\frac{N}{2}\sum_{r=1}^{s} (m_r-m_{r-1})\vartheta_\xi(\gamma_r) \\
&= \frac{N}{2}\vartheta_\xi(\diag(d)) - \frac{N}{2}\int_0^1 \vartheta_\xi(\pi(t))\ \dd t
\stackref{pre_par_b}{=} -N\PPP_{\xi}^{(2)}(\pi).
}
\end{subequations}
Therefore, \eqref{cwje3} is exactly the desired statement \eqref{PPPPsi_b}.
\end{proof}

\begin{proof}[Proof of Lemma~\ref{lem_parisi_recovered}]
Inserting \eqref{Zialpha} into \eqref{pre_par_a}, we can express $\PPP_{N,\xi}^{(1)}(\pi,\lambda;S)$ as
\eq{
\frac{1}{N}\E\log\sum_{\alpha\in\N^{s-1}}\sum_{\sigma\in S}\nu_\alpha\exp\Big(\sum_{i=1}^N\iprod[\Big]{\sum_{r=0}^{s-1}\sqrt{\nabla\xi(\gamma_{r+1})-\nabla\xi(\gamma_{r})\one\{r>0\}}\, \eta_{i,(\alpha_1,\dots,\alpha_{r})}+\lambda}{\sigma_i}\Big).
}
By \eqref{magic}, we have the alternative representation
\begin{subequations} \label{z9dl2}
\eeq{ \label{mnqwe3_new}
\PPP_{N,\xi}^{(1)}(\pi,\lambda;S) = \frac{1}{N}X_0(S),
}
where $X_0(S), X_1(S),\dots,X_s(S)$ are related inductively as in \eqref{following_just_zs}, and
\eeq{ \label{nb71}
X_s(S) = \log\sum_{\sigma\in S}\exp\Big(\sum_{i=1}^N\iprod[\Big]{\sum_{r=0}^{s-1}z_{i}^{(r)}+\lambda}{\sigma_i}\Big).
}
Here each $z_i^{(r)}$ is an independent~centered Gaussian vector in $\R^\kappa$ with covariance matrix
\eeq{ \label{nxe8}
\E[z_{i}^{(r)}(z_i^{(r)})^\sT] = \nabla\xi(\gamma_{r+1})-\nabla\xi(\gamma_r)\one\{r>0\}.
}
\end{subequations}
When $S$ is the entire product set $\Sigma^N$, \eqref{nb71} can be rewritten as
\eq{
X_s(\Sigma^N)
&= \log \sum_{\sigma_1,\dots,\sigma_N\in\Sigma}\prod_{i=1}^N\exp\Big(\sum_{r=0}^{s-1}\iprod{z_i^{(r)}+\lambda}{\sigma_i}\Big) \\
&= \sum_{i=1}^N \log \sum_{\sigma\in\Sigma}\exp\Big(\sum_{r=0}^{s-1}\iprod{z_i^{(r)}+\lambda}{\sigma}\Big)
= \sum_{i=1}^N X_s^{(i)}(\Sigma),
}
where $X_s^{(1)}(\Sigma),\dots,X_s^{(N)}(\Sigma)$ are i.i.d.~random variables.
The inductive procedure \eqref{following_just_zs} can be applied to each one of these variables separately, resulting in i.i.d.~random variables $X_r^{(1)}(\Sigma),\dots,X_r^{(N)}(\Sigma)$ for each $r\in\{0,1,\dots,s-1\}$.
Assuming inductively that $X_{r+1}(\Sigma^N) = \sum_{i=1}^N X_{r+1}^{(i)}(\Sigma)$, we have
\eq{
X_r(\Sigma^N) 
&\stackref{following_just_zs}{=}\frac{1}{m_r}\log\E_r\exp\Big(m_r\sum_{i=1}^N X_{r+1}^{(i)}(\Sigma)\Big) \\
&\stackrefp{following_just_zs}{=} \frac{1}{m_r}\log\prod_{i=1}^N \E_r\exp(m_rX_{r+1}^{(i)}(\Sigma)) 
\stackref{following_just_zs}{=} \sum_{i=1}^N X_r^{(i)}(\Sigma),
}
where the middle equality uses independence.
Taking $r=0$, we conclude that
\eeq{ \label{ibjwe}
\PPP_{N,\xi}^{(1)}(\pi,\lambda;\Sigma^N) \stackref{mnqwe3_new}{=} \frac{1}{N} X_0(\Sigma^N) = \frac{1}{N}\sum_{i=1}^NX_0^{(i)}(\Sigma) = X_0^{(1)}(\Sigma) \stackref{mnqwe3_new}{=} \PPP_{1,\xi}^{(1)}(\pi,\lambda;\Sigma).
}
Moreover, a comparison of \eqref{mnqwe3_new} and \eqref{eq:def:Psi} reveals that
\eq{
\PPP_{1,\xi}^{(1)}(\pi,\lambda;\Sigma) = \PPP_{\xi}^{(1)}(\pi,\lambda),
}
and so \eqref{ibjwe} is exactly \eqref{qk38}.
\end{proof}

\section{Duality}
\label{subsec:lem:duality}
In this appendix, we prove Lemma~\ref{lem:duality}. We closely follow the proof of \cite[Lem.~2]{panchenko18a}. 
The key ingredient is the following result.
Recall the set $\Sigma^N(d,\eps)$ from \eqref{eq:def:Sigma:N:d:eps}:
using the notation $R(\sigma,\sigma') = \sigma\sigma'^\sT/N$, we have
\eeq{ \label{Sigma_N_redef}
\Sigma^N(d,\eps) = \{\sigma\in\Sigma^N:\, \|R(\sigma,\sigma)-\diag(d)\|_\infty\leq\eps\}.
}

\begin{lemma}\label{lem:PP:1:eps:approx}
Assume $\xi$ satisfies~\ref{xi_power}.
    There exists a constant $C$ depending only on $\xi$ and $\ka$, such that for every $N\geq 1$, $\pi\in\Pi_{d'}$, $\lambda\in\R^\kappa$, and $\eps>0$, we have
    \eeq{ \label{bq3t}
        \sup_{d\in \DD_N}\Big|\PPP^{(1)}_{N,\xi}\big(\pi,\lambda; \Sigma^{N}(d,\eps)\big)-\PPP^{(1)}_{N,\xi}\big(\pi,\lambda; \Sigma^{N}(d)\big)\Big|\leq C\sqrt{\eps}.
    }
\end{lemma}

\begin{proof}
We assume without loss of generality that $\eps\in(0,1/2]$.
Note for later that there exists a constant $C_0$ such that
\eeq{ \label{dhwe4r}
\eps\log(1/\eps) \leq C_0\sqrt{\eps} \quad \text{for all $\eps\in(0,1/2]$}.
}
For any $\lambda\in\R^\kappa$ and $S\subseteq\Sigma^N$, the map $\Pi_{d'}\ni\pi\mapsto\PPP^{(1)}_{N,\xi}(\pi,\lambda;S)$ is by definition a continuous extension of its restriction to the set $\Pi_{d'}^\disc$ (see the text following the proof of Lemma~\ref{pconlem}).
Therefore, it suffices to verify the lemma for $\pi\in\Pi_{d'}^\disc$.
We assume $\pi$ is given by \eqref{eq:pi:discrete}, with $d'$ replacing $d$ in \eqref{eq:gamma:discrete}.

    Given any $d\in\DD_N$, let $P\colon\Sigma^N(d,\eps)\to \Sigma^N(d)$ be any projection with respect to Hamming distance. 
    That is, for each $\sigma\in \Sigma^N(d,\eps)$, we choose some $P\sigma\in\Sigma^N(d)$ that minimizes $\sum_{i=1}^{N}\one\{\sigma_i \neq (P\sigma)_i\}$. 
Changing a single column $\sigma_i$ (from one standard basis vector to another) moves the self-overlap matrix $R(\sigma,\sigma)$ by distance $N^{-1}$ in the $\ell^\infty$ norm.
Therefore, $\sigma$ and $P\sigma$ differ in at most $\eps N$ columns: 
    \begin{equation}\label{eq:projection:guarantee}
        \sum_{i=1}^{N}\one\{\sigma_i \neq (P\sigma)_i\}\leq \eps N \quad \text{for every $\sigma\in \Sigma^N(d,\eps)$}.
    \end{equation}
    In particular, if $\eps N<1$, then $P$ is the identity map and so $\Sigma^N(d,\eps)=\Sigma^N(d)$.
    The right-hand side of \eqref{bq3t} is zero in this case, and so we assume henceforth that $\eps N \geq 1$.
    
    Consider the size of the preimage,
    \eeq{ \label{eq:def:NN}
\NN(\sigma)\coloneqq\big|\{\tau\in\Sigma^N(d,\eps):\, P\tau=\sigma\}\big|.
}
If $P\tau=\sigma$, then \eqref{eq:projection:guarantee} says there are at most $\eps N$ many columns in which $\tau$ and $\sigma$ differ.
Each such column of $\tau$ takes one of $\kappa$ possible values, and so
\eq{
\NN(\sigma) 
\leq {N \choose \lceil \eps N\rceil}\kappa^{\eps N}
\leq \Big(\frac{\exp(1)N}{\lceil \eps N\rceil}\Big)^{\lceil \eps N\rceil}\kappa^{\eps N}
\leq \Big(\frac{\exp(1) \kappa}{\eps}\Big)^{\eps N+1}
\leq \Big(\frac{\exp(1) \kappa}{\eps}\Big)^{2\eps N}.
}
Taking logarithms and dividing by $N$, we arrive to
\eeq{ \label{NN_bound}
\frac{1}{N}\log \NN(\sigma)
\leq 2\eps\big(1+\log \kappa+\log(1/\eps)\big)
\stackref{dhwe4r}{\leq} C_\kappa\sqrt{\eps},
}
where $C_\kappa = 2(1+\log \kappa+C_0)$.

    Next let $Z_1,\dots,Z_N,\wt Z_1,\dots,\wt Z_N\colon\N^{s-1}\to\R^\kappa$ be i.i.d.~centered Gaussian processes with covariance given by \eqref{bni}. 
    Define the interpolating process $Z_t\colon\N^{s-1}\times\Sigma^N(d,\eps)\to\R$ by
\eeq{ \label{wh4w45}
        Z_{t}(\alpha,\sigma)=\sum_{i=1}^N \Big[\sqrt{t}\iprod[\big]{Z_i(\alpha)+\lambda}{\sigma_i}+\sqrt{1-t}\iprod[\big]{\wt Z_i(\alpha)+\lambda}{(P\sigma)_i}\Big], \quad t\in[0,1].
}
    Upon defining the associated free energy,
    \eeq{ \label{hdnw73}
        \phi(t)=\frac{1}{N}\E \log\sum_{\alpha\in \N^{s-1}}\sum_{\sigma\in \Sigma^{N}(d,\eps)}\nu_\alpha\exp Z_t(\alpha,\sigma),
    }
    we have
\eeq{ \label{gwe4g}
\phi(1)&=\PPP_{N,\xi}^{(1)}\big(\pi,\lambda;\Sigma^N(d,\eps)\big), \\
\phi(0)&=\frac{1}{N}\E \log\sum_{\alpha\in \N^{s-1}}\sum_{\sigma\in \Sigma^{N}(d,\eps)}\exp\Big(\sum_{i=1}^N \iprod[\big]{\wt Z_i(\alpha)+\lambda}{(P\sigma)_i}\Big).
}
     Note that $\NN(\sigma)\geq 1$ since $P\sigma=\sigma$, and thus
     \eeq{ \label{eq:phi:zero:approx}
         &\PPP_{N,\xi}^{(1)}\big(\pi,\lambda;\Sigma^N(d)\big)
         \stackref{pre_par_a}{=} \frac{1}{N}\E\log\sum_{\alpha\in\N^{s-1}}\sum_{\sigma\in\Sigma^N(d)}\nu_\alpha\exp\Big(\sum_{i=1}^N\iprod[\big]{Z_{i}(\alpha)+\lambda}{\sigma_i}\Big) \\
         &\qquad\stackrefp{NN_bound}{\leq} \frac{1}{N}\E \log\sum_{\alpha\in \N^{s-1}}\sum_{\sigma\in \Sigma^{N}(d)}\NN(\sigma)\nu_\alpha\exp\Big(\sum_{i=1}^N \iprod[\big]{Z_i(\alpha)+
         \lambda }{\sigma_i}\Big)\stackref{gwe4g}{=}\phi(0)\\
         &\qquad\stackref{NN_bound}{\leq} \PPP_{N,\xi}^{(1)}\big(\pi,\lambda;\Sigma^N(d)\big)+C_\kappa\sqrt{\eps}.
}
Putting these observations together and using a triangle inequality, we have
\eeq{ \label{gxdw8}
&\big|\PPP_{N,\xi}^{(1)}\big(\pi,\lambda;\Sigma^N(d,\eps)\big)-\PPP_{N,\xi}^{(1)}\big(\pi,\lambda;\Sigma^N(d)\big)\big| \\
&\qquad\qquad\stackref{gwe4g}{=} \big|\phi(1) - \PPP_{N,\xi}^{(1)}\big(\pi,\lambda;\Sigma^N(d)\big)\big| \\
&\qquad\qquad\stackrefp{eq:phi:zero:approx}{\leq} |\phi(1)-\phi(0)| + \big|\phi(0) - \PPP_{N,\xi}^{(1)}\big(\pi,\lambda;\Sigma^N(d)\big)\big| \\
&\qquad\qquad\stackref{eq:phi:zero:approx}{\leq} |\phi(1)-\phi(0)|+C_\kappa\sqrt{\eps}.
}
    To control the first term in the last line, we consider the probability measure $G_t(\alpha,\sigma)\propto \nu_{\alpha}\exp Z_t(\alpha,\sigma)$ on $\N^{s-1}\times \Sigma^{N}(d,\eps)$, and
    let $\langle \cdot \rangle_t$ denote expectation with respect to  $G_t$. 
    Differentiation of \eqref{hdnw73} results in
    \eq{
        \phi^\prime(t)=\frac{1}{N}\E\Big\langle \frac{\partial Z_t(\alpha,\sigma) }{\partial t}\Big\rangle_t.
        }
    Applying Gaussian integration by parts \cite[Lem.~1.1]{panchenko13a}, we can rewrite this as
    \eq{
    \phi'(t)
    =\frac{1}{N}\E\Big\langle \CC\big((\alpha^1,\sigma^1),(\alpha^1,\sigma^1)\big)-\CC\big((\alpha^1,\sigma^1),(\alpha^1,\sigma^1)\big)\Big\rangle_t,
    }
    where $(\alpha^1,\sigma^1),(\alpha^2,\sigma^2)$ denote independent samples from $G_t$, and $\CC\colon(\N^{s-1}\times\Sigma^N(d,\eps))^2\to\R$ is given by
\eq{
        \CC\big((\alpha,\sigma),(\alpha^\prime,\sigma^\prime)\big)
        &\stackrefp{wh4w45}{=}\E\Big[\frac{\partial Z_t(\alpha,\sigma)}{\partial t} Z_t(\alpha^\prime,\sigma^\prime)\Big] \\
        &\stackref{wh4w45}{=} \frac{1}{2}\sum_{i=1}^N\E\Big[\iprod[\big]{Z_i(\alpha)}{\sigma_i}\iprod[\big]{Z_i(\alpha')}{\sigma_i'}
        -\iprod[\big]{\wt Z_i(\alpha)}{(P\sigma)_i}\iprod[\big]{\wt Z_i(\alpha')}{(P\sigma')_i}\Big] \\
        &\stackrefpp{bni}{wh4w45}{=}\frac{1}{2}\sum_{i=1}^{N}\Big(\sigma_i^{\sT} \nabla \xi(\gamma_{r(\alpha,\alpha')})\sigma_i^\prime-(P\sigma)_i^{\sT} \nabla\xi(\gamma_{r(\alpha,\alpha')})(P\sigma^\prime)_i\Big).
}
    By \eqref{eq:projection:guarantee}, there are at most $2\eps N$ many values of $i$ for which
    the $i^\mathrm{th}$ summand on the final line is nonzero.
    Since $\sigma_i,\sigma_i',(P\sigma)_i,(P\sigma')_i$ are all standard basis vectors, the nonzero summands trivially satisfy
    \eq{
    \Big|\sigma_i^{\sT} \nabla \xi(\gamma_{r(\alpha,\alpha')})\sigma_i^\prime-(P\sigma)_i^{\sT} \nabla\xi(\gamma_{r(\alpha,\alpha')})(P\sigma^\prime)_i\Big| \leq 2\sup_{\|R\|_1\leq1} \|\nabla\xi(R)\|_\infty.
}
    Therefore, the four previous displays together imply
    \eq{
|\phi(1)-\phi(0)|
\leq\sup_{t\in(0,1)}\big|\phi^\prime(t)\big| 
\leq 4\eps \sup_{\|R\|_1\leq1}\|\nabla\xi(R)\|_\infty.
}
    Inserting this inequality into \eqref{gxdw8} results in
    \begin{equation*}
        \big|\PPP^{(1)}_{N,\xi}\big(\pi,\lambda; \Sigma^{N}(d,\eps)\big)-\PPP^{(1)}_{N,\xi}\big(\pi,\lambda; \Sigma^{N}(d)\big)\big|
        \leq C\sqrt{\eps},
    \end{equation*}
    where $C = 4\sup_{\|R\|_1\leq1}\|\nabla\xi(R)\|_\infty+C_\kappa$. 
\end{proof}
With Lemma~\ref{lem:PP:1:eps:approx} in hand, the next result follows by the exact same argument as \cite[Lem.~4 and Lem.~5]{panchenko18a}, and thus we omit the proof.
\begin{lemma}\label{lem:convergence:f}
Assume $\xi$ satisfies~\ref{xi_power}.
Fix any $\pi\in\Pi_{d'}$.
    Then for any $d\in \DD$ and $\eps>0$, the following limit exists and is a continuous concave function of $d$:
    \begin{equation*}
        f_{\eps}(d)\coloneqq\lim_{N\to\infty}\PPP^{(1)}_{N,\xi}\big(\pi,0; \Sigma^{N}(d,\eps)\big).
    \end{equation*}
    Furthermore, the limit $f(d)\coloneqq\lim_{\eps\searrow 0}f_{\eps}(d)$ exists and satisfies the following statement. 
    If $d^N\in \DD_N$ converges to $d\in \DD$ as $N\to\infty$, then
    \eeq{ \label{bw843}
        f(d)=\lim_{N\to\infty}\PPP^{(1)}_{N,\xi}\big(\pi,0; \Sigma^{N}(d^N)\big).
    }
    Finally, there exists a constant $C$ depending only on $\xi$ and $\ka$, such that
    \begin{equation*}
        |f(d^1)-f(d^2)|\leq C\|d^1-d^2\|_{\infty}^{1/2} \quad \text{for all $d^1,d^2\in \DD$}.
    \end{equation*}
\end{lemma}

We are now ready to complete the objective of this appendix.
The following proof is very similar to that of \cite[Lem.~6]{panchenko18a}.

\begin{proof}[Proof of Lemma~\ref{lem:duality}]
As in the proof of Lemma~\ref{lem:PP:1:eps:approx}, let $\pi$ be given by \eqref{eq:pi:discrete}, with $d'$ replacing $d$ in \eqref{eq:gamma:discrete}.
We will show that
 \begin{equation}\label{eq:pre:duality}
     0\leq \PPP^{(1)}_{N,\xi}(\pi,\la; \Sigma^N)-\max_{d\in \DD_{N}}\PPP^{(1)}_{N,\xi}(\pi,\la; \Sigma^N(d))\leq \frac{\ka \log (N+1)}{m_1 N}.      
 \end{equation}
Before proving \eqref{eq:pre:duality}, we use it to prove the desired identity \eqref{mijr}.
For every $\sigma\in\Sigma^N(d)$, we have $\sum_{i=1}^n\iprod{\lambda}{\sigma_i} = N\iprod{\lambda}{d}$.
Therefore, when $S=\Sigma^N(d)$, definition \eqref{pre_par_a} results in
\eq{
\PPP_{N,\xi}^{(1)}(\pi,\lambda;\Sigma^N(d)) 
= \PPP_{N,\xi}^{(1)}(\pi,0,\Sigma^N(d))+\iprod{\lambda}{d} \quad \text{for $d\in\DD_N$}.
}
Using this identity and $\PPP_{N,\xi}^{(1)}(\pi,\lambda;\Sigma^N)=\PPP_\xi^{(1)}(\pi,\lambda)$ from \eqref{qk38}, we infer from \eqref{eq:pre:duality} that
\eeq{\label{hgnv21}
\PPP_\xi^{(1)}(\pi,\lambda) 
=\lim_{N\to\infty}\max_{d\in\DD_N}\big[\PPP^{(1)}_{N,\xi}(\pi,0; \Sigma^N(d))+\iprod{\la}{d}\big].
}
We then claim that
\eeq{ \label{hfdw133}
     \PPP^{(1)}_{\xi}(\pi,\la)
     =\sup_{d\in \DD}\big[ f(d)+\iprod{\lambda}{d} \big].
}
Indeed, let $d^N\in\DD_N$ be a maximizer in \eqref{hgnv21}. 
By compactness of $\DD$, we may assume (by passing to a subsequence) that $d^N$ converges to some $d$ as $N\to\infty$. 
It then follows that
\eeq{ \label{wwjg225}
\PPP^{(1)}_{\xi}(\pi,\lambda)
\stackref{hgnv21}{=} \lim_{N\to\infty}\big[\PPP^{(1)}_{N,\xi}(\pi,0; \Sigma^N(d^N))+\iprod{\la}{d^N}\big]
\stackref{bw843}{=} f(d) + \iprod{\lambda}{d}.
}
On the other hand, if we choose any other convergent sequence $d^N\to d$, not necessarily a sequence of maximizers, then the first equality in \eqref{wwjg225} becomes $\geq$, and \eqref{hfdw133} follows.
Finally, since $f$ is concave and continuous by Lemma~\ref{lem:convergence:f}, the Fenchel--Moreau theorem (e.g. \cite[Thm.~12.2]{rockafellar70}) implies the following dual version of \eqref{hfdw133}:
\eq{
f(d) = \inf_{\lambda\in\R^\kappa}\big[\PPP_\xi^{(1)}(\pi,\lambda) - \iprod{\lambda}{d}\big].
}
Once we use \eqref{bw843} to rewrite $f(d)$, this identity is exactly \eqref{mijr}.

 It remains to establish the two inequalities in \eqref{eq:pre:duality}.
 The first inequality is immediate from definition \eqref{pre_par_a}, since $\Sigma^N(d)\subseteq\Sigma^N$.
 For the second inequality, recall the notation in \eqref{z9dl2}.
 We claim
 \eeq{ \label{gxqng2}
\exp\big(m_rX_r(\Sigma^N)\big)\leq \sum_{d\in\DD_N}\exp\Big(m_r X_r\big(\Sigma^N(d)
\big)\Big) \quad \text{for each $r\in\{1,\dots,s\}$}.
}
To see this claim, first observe that there is equality when $r=s$, simply by inspecting \eqref{nb71} together with the fact that $m_s=1$.
Now proceed by downward induction.
Assuming the claim is true in the $(r+1)^\text{th}$ case, and noting that $m_r/m_{r+1}\leq1$, we have 
\eq{
     \exp\big(m_{r} X_{r}(\Sigma^N)\big)
     &\stackref{following_just_zs}{=}\E_{r}\exp\big(m_r X_{r+1}(\Sigma^N)\big)\\
     &\stackrefp{following_just_zs}{=}\E_{r}\Big[\exp\big(m_{r+1}X_{r+1}(\Sigma^N)\big)^{m_{r+1}/m_r}\Big]\\
     &\stackrefp{following_just_zs}{\leq} \E_{r}\bigg[\bigg(\sum_{d\in \DD_N}\exp\Big(m_{r+1} X_{r+1}\big(\Sigma^N(d)\big)\Big)\bigg)^{m_r/m_{r+1}}\bigg]\\
     &\stackrefp{following_just_zs}{\leq} \E_{r}\bigg[\sum_{d\in \DD_N}\exp\Big(m_{r} X_{r+1}\big(\Sigma^N(d)\big)\Big)\bigg]\\
    &\stackref{following_just_zs}{=}\sum_{d\in \DD_N}\exp\Big(m_rX_{r}\big(\Sigma^N(d)\big)\Big).
}
 With \eqref{gxqng2} established, we use the final $r=1$ case:
\eq{
      \PPP^{(1)}_{N,\xi}(\pi,\la; \Sigma^N)
      &\stackref{mnqwe3_new}{=}\frac{1}{N}\E_0 X_1(\Sigma^N)
      \stackref{gxqng2}{=} \frac{1}{m_1 N}\E_0\log\sum_{d\in \DD_N}\exp\Big(m_1 X_1\big(\Sigma^N(d)\big)\Big).
}
By definition of $\DD_N$ in \eqref{DN_def}, each $d=(d_1,\dots,d_\kappa)\in\DD_N$ has $Nd_i\in\{0,1,\dots,N\}$ for each $i$. 
Hence $|\DD_N|\leq(N+1)^\kappa$, and so the previous display yields
\eq{
 \PPP^{(1)}_{N,\xi}(\pi,\la; \Sigma^N)
      &\stackrefp{gxqng2}{\leq}\frac{\ka \log(N+1)}{m_1 N}+\max_{d\in \DD_N}\frac{1}{N}\E_0 X_1\big(\Sigma^N(d)\big)\\
      &\stackref{mnqwe3_new}{=}\frac{\ka \log(N+1)}{m_1 N}+\max_{d\in \DD_N}\PPP^{(1)}_{N,\xi}(\pi,\la; \Sigma^N(d)).
}
This completes the proof of \eqref{eq:pre:duality}, and so we are done.
\end{proof}

\section{Parisi formula for general model} \label{sec_gen_par_proof}
In this appendix, we prove Theorem~\ref{gen_con_par_thm} by generalizing the strategy used by Panchenko~\cite{panchenko18a,panchenko18b} for Theorem~\ref{thm:Panchenko}.
The following upper bound uses Guerra-style interpolation~\cite{guerra03} and thus requires the convexity assumption~\ref{xi_convex}.

\begin{proposition}\label{prop:upper:interpolation}
     Assume $\xi$ satisfies~\ref{xi_power} and~\ref{xi_convex}. 
     For every $d\in\DD$, we have
\eeq{ \label{ldw3er}
  \lim_{\eps\searrow0}\limsup_{N\to\infty}\eff_{N,\xi}(d,\eps)
  \leq \inf_{\pi\in \Pi_d} \PP_{\xi}(\pi).
}
\end{proposition}

Meanwhile, the lower bound does not require any convexity assumption.
The following result follows from the Aizenman--Sims--Starr scheme as 
in \cite[Sec.~7]{panchenko18a}.

\begin{proposition} \label{prop_general_lower_bound}
Assume $\xi$ satisfies~\ref{xi_power}.
Given $d=(d_k)_{k=1}^\kappa\in\DD$ and any constant $L$, assume $d^N\in\DD_N$ is such that 
\eeq{ \label{fw424}
\|d^N-d\|_{\infty}\leq L/N \quad \text{and} \quad \text{$d^N_k=0$ whenever $d_k=0$}.
}
Then
\eeq{ \label{eq_prop_general_lower_bound}
\liminf_{N\to\infty} \eff_{N,\xi}(d^N) \geq \inf_{\pi\in\Pi_d}\PP_\xi(\pi).
}
\end{proposition}

Although \cite{panchenko18a} considers only the case $\xi(R)=\tr(R^{\sT}R)$, the argument for Proposition~\ref{prop_general_lower_bound} proceeds in exactly the same way, and thus we omit the proof. 
Proposition~\ref{prop:upper:interpolation}, however, is inherently more sensitive to the covariance function $\xi$, and so we do provide its proof.



\begin{proof}[Proof of Proposition~\ref{prop:upper:interpolation}]
Since $\Pi_{d}^{\disc}$ is a dense subset of $\Pi_d$ with respect to the $L^1$ norm \eqref{norm}, and $\pi \mapsto \PPP_{\xi}(\pi)$ is continuous by Proposition~\hyperref[prop:continuity_b]{\ref*{prop:continuity}\ref*{prop:continuity_b}}, we have $\inf_{\pi\in \Pi_d} \PP_{\xi}(\pi)=\inf_{\pi\in \Pi_d^{\disc}}\PP_{\xi}(\pi)$. 
Therefore, to establish \eqref{ldw3er}, it suffices to prove that for any path $\pi$ of the form \eqref{eq:pi:discrete}, we have
\eeq{ \label{hd234o}
  \lim_{\eps\searrow0}\limsup_{N\to\infty}\eff_{N,\xi}(d,\eps)
  \leq \PP_{\xi}(\pi).
}
%
To this end, let $Z_1,\dots,Z_N\colon\N^{s-1}\to\R^\kappa$ and $Y\colon\N^{s-1}\to\R$ be independent centered Gaussian processes with covariances given by \eqref{bni}, and
define the following process
on $\N^{s-1}\times\Sigma^N$:
\eeq{ \label{25bx1}
\HH_t(\alpha,\sigma)=\sqrt{t}\, H_{N,\xi}(\sigma)+\sqrt{1-t}\sum_{i=1}^{N}\iprod[\big]{Z_{i}(\alpha)}{\sigma_i} +\sqrt{t}\sqrt{N}\,Y(\alpha), \quad t\in[0,1].
}
We assume that $H_{N,\xi}$ is independent of $Z_1,\dots,Z_N,Y$, and that all of these Gaussian processes are independent of the RPC weights $(\nu_{\alpha})_{\alpha\in \N^{s-1}}$ associated with \eqref{eq:m}.
Given $\eps>0$, consider the associated constrained free energy 
\begin{equation}\label{eq:interpolating:free:energy}
    \phi_N(t)=\frac{1}{N}\E \log \sum_{\alpha\in \N^{s-1}}\sum_{\sigma\in \Sigma^{N}(d,\eps)}\nu_{\alpha} \exp \HH_t(\alpha,\sigma).
\end{equation}

\begin{claim}\label{lem:monotone}
For any $N\geq 1$ and $t\in (0,1)$, the derivative of $\phi_N$ satisfies $\phi_N^\prime(t)\leq C\eps$ for some constant $C$ not depending $N$, $d$, or $\eps$.
\end{claim}

\begin{proofclaim}
    Denote by $\langle \cdot \rangle_t$ the average with respect to the following probability measure on $\N^{s-1}\times \Sigma^{N}(d,\eps)$:
    \begin{equation*}
    G_{t}(\alpha,\sigma)\propto \nu_{\alpha}\exp \HH_t(\alpha,\sigma).
    \end{equation*}
Differentiation of \eqref{eq:interpolating:free:energy} results in
\eq{ 
    \phi_N^\prime(t)
    =\frac{1}{N}\Big\langle\frac{\partial\HH_t(\alpha,\sigma)}{\partial t}\Big\rangle_{t}.
}
    Applying Gaussian integration by parts \cite[Lem.~1.1]{panchenko13a}, we rewrite this as
\eeq{ \label{eq:calculate:derivative}
    \phi_N'(t)
    =\E \Big\langle \CC\big((\alpha^1,\sigma^1),(\alpha^1,\sigma^1)\big)-\CC\big((\alpha^1,\sigma^1),(\alpha^2,\sigma^2)\big)\Big\rangle_t,
}
where $(\alpha^1,\sigma^1),(\alpha^2,\sigma^2)$ denote independent samples from $G_t$, and $\CC\colon(\N^{s-1}\times \Sigma^{N}(d,\eps))^2\to\R$ is defined by
    \begin{equation*}
    \CC\big((\alpha,\sigma), (\alpha',\sigma')\big)
    =\frac{1}{N}\E\Big[\frac{\partial\HH_t(\alpha,\sigma)}{\partial t}\HH_t(\alpha',\sigma')
\Big].
    \end{equation*}
    To compute this expectation, recall that the three Gaussian processes on the right-hand side of \eqref{25bx1} are independent, centered, and have covariances given by \eqref{general_cov} and \eqref{bni}.
    Therefore,
\eq{
    \CC\big((\alpha^1,\sigma^1), (\alpha^2,\sigma^2)\big)
    &\stackrefp{vthet_def}{=}\frac{1}{2}\Big(\xi(\RR_{1,2})-\iprod[\big]{\nabla \xi(\gamma_{r(\alpha^1,\alpha^2)})}{\RR_{1,2}} +\vartheta_\xi(\gamma_{r(\alpha^1,\alpha^2)})\Big)\\
    &\stackref{vthet_def}{=}\frac{1}{2}\Big(\xi(\RR_{1,2})-\xi(\gamma_{r(\alpha^1,\alpha^2)})-\iprod[\big]{\nabla \xi(\gamma_{r(\alpha^1,\alpha^2)})}{\RR_{1,2}-\gamma_{r(\alpha^1,\alpha^2)}} \Big),
}
    where $\RR_{1,2}=R(\sigma^1,\sigma^2)$ as in \eqref{def:overlap}. 
    By the convexity assumption~\ref{xi_convex}, 
    we have
    \eeq{ \label{hxw42}
\CC\big((\alpha^1,\sigma^1), (\alpha^2,\sigma^2)\big)\geq 0.
}
    In the special case $(\alpha^1,\sigma^1)=(\alpha^2,\sigma^2)$, we have $\gamma_{r(\alpha^1,\alpha^1)}=\gamma_{s}=\diag(d)$ by \eqref{eq:gamma:discrete}, and
    $\|\RR_{1,1}-\diag(d)\|_{1}\leq \kappa\eps$
     by definition of $\Sigma^{N}(d,\eps)$ in~\eqref{Sigma_N_redef}.
    It follows that
\eeq{ \label{cg72kf}
        \CC\big((\alpha^1,\sigma^1), (\alpha^1,\sigma^1)\big)\leq \frac{1}{2}\sup_{\|R\|_1\leq1}\|\nabla\xi(R)\|_\infty\cdot\kappa\eps.
}
    Combining \eqref{eq:calculate:derivative}, \eqref{hxw42}, and \eqref{cg72kf} concludes the proof.
\end{proofclaim}

Claim~\ref{lem:monotone} implies
\eeq{ \label{5d6f1m}
\phi_N(1) \leq \phi_N(0)+C\eps.
}
     When $t=1$, the $\alpha$ and $\sigma$ terms in \eqref{eq:interpolating:free:energy} fully decouple, resulting in
\eeq{ \label{5d6f2m}
         \phi_N(1)= \eff_{N,\xi}(d,\eps)+\frac{1}{N}\E\log \sum_{\alpha\in \N^{s-1}}\nu_{\alpha}\exp\big(\sqrt{N}\, Y_{\alpha}\big)
         \stackref{xf573}{=}\eff_{N,\xi}(d,\eps)-\PPP_{\xi}^{(2)}(\pi). 
}
     On the other hand, evaluating \eqref{eq:interpolating:free:energy} at $t=0$ yields
\eeq{ \label{hs4w2m}
     \phi_N(0)
     =\frac{1}{N}\E\log \sum_{\alpha\in \N^{s-1}}\sum_{\sigma\in \Sigma^{N}(d,\eps)}\nu_{\alpha}\exp\Big(\sum_{i=1}^N\iprod[\big]{Z_{i}(\alpha)}{\sigma_i}\Big).
}
     For any $\sigma\in\Sigma^N(d,\eps)$ and $\lambda\in\R^\kappa$, we have
     \eeq{ \label{hs4w2n}
\frac{1}{N}\sum_{i=1}^N\iprod{\lambda}{\sigma_i}
= \iprod{\lambda}{d} + \iprod[\Big]{\lambda}{\frac{1}{N}\sum_{i=1}^N\sigma_i - d}
\stackref{eq:def:Sigma:N:d:eps}{\geq} \iprod{\lambda}{d} - \eps\|\lambda\|_1.
}
     Now add and subtract $\iprod{\lambda}{\sigma_i}$ within the exponent appearing in \eqref{hs4w2m}.
     Using \eqref{hs4w2n} and the trivial containment $\Sigma^N(d,\eps)\subseteq\Sigma^N$, we deduce that
\eeq{ \label{5d6f3m}
         \phi_N(0)
         &\stackrefp{pre_par_a}{\leq} \frac{1}{N}\E\log \sum_{\alpha\in \N^{s-1}}\sum_{\sigma\in \Sigma^{N}}\nu_{\alpha}\exp\Big(\sum_{i=1}^N\iprod[\big]{Z_{i}(\alpha)+\la }{\sigma_i}\Big)-\iprod{\la}{d} + \eps\|\lambda\|_1\\
         &\stackref{pre_par_a}{=} \PPP_{N,\xi}^{(1)}(\pi,\la;\Sigma^N)-\iprod{\la}{d} + \eps\|\lambda\|_1
         \stackref{qk38}{=}\PPP_{\xi}^{(1)}(\pi,\la)-\iprod{\la}{d} + \eps\|\lambda\|_1.
}
     Inserting \eqref{5d6f2m} and \eqref{5d6f3m} into \eqref{5d6f1m}, we arrive to
 \eq{
         \eff_{N,\xi}(d,\eps)
         &\stackrefp{eq:Parisi:ftl:lagrange}{\leq} \PPP_{\xi}^{(1)}(\pi,\la)+\PPP_{\xi}^{(2)}(\pi) -\iprod{\lambda}{d}+(C+\|\lambda\|_1)\eps \\
         &\stackref{eq:Parisi:ftl:lagrange}{=}\PPP_\xi(\pi,\lambda)-\iprod{\lambda}{d}+(C+\|\lambda\|_1)\eps.
}
As this inequality holds for any $N$, we conclude
\eeq{ \label{wef4tw}
         \lim_{\eps\searrow0}\limsup_{N\to\infty}\eff_{N,\xi}(d,\eps)
         \leq \PPP_{\xi}(\pi,\lambda)-\iprod{\lambda}{d}.
}
    Finally, recall that $\PP_\xi(\pi) = \inf_{\lambda\in\R^\kappa}[\PPP_\xi(\pi,\lambda)-\iprod{\lambda}{d}]$, and so \eqref{wef4tw} implies \eqref{hd234o}.
\end{proof}


We are almost ready to prove Theorem~\ref{gen_con_par_thm}.
One technical detail that needs to be resolved is relaxing \eqref{fw424} to the weaker condition $d^N\to d$.
The crucial lemma is the following analogue of Lemma~\ref{lem:PP:1:eps:approx} for free energy.
\begin{lemma}\label{lem:free:energy:eps:approx}
Assume $\xi$ satisfies~\ref{xi_power}.
There exists a constant $C$ depending only on $\xi$ and $\ka$, such that for every $N\geq 1$ and $\eps>0$, we have
\eeq{ \label{cbw24}
        \sup_{d\in \DD_N}\big|\eff_{N,\xi}(d,\eps)-\eff_{N,\xi}(d)\big|\leq C\sqrt{\eps}.
}
\end{lemma}

\begin{proof}
    As in the proof of Lemma~\ref{lem:PP:1:eps:approx}, assume without loss of generality that $\eps\in[1/N,1/2]$, and consider any map $P\colon\Sigma^N(d,\eps)\to \Sigma^N(d)$ that is a projection with respect to  Hamming distance. 
    Let $\wt H_{N,\xi}$ be an independent copy of $H_{N,\xi}$, and then define an interpolating Hamiltonian on $\Sigma^N(d,\eps)$:
    \eq{
\HH_t(\sigma)=\sqrt{t}\, H_{N,\xi}(\sigma)+\sqrt{1-t}\, \wt{H}_{N,\xi}(P\sigma), \quad t\in[0,1].
}
    Define the associated free energy
\eeq{ \label{wrg23}
        \phi(t)=\frac{1}{N}\E\log \sum_{\sigma\in \Sigma^N(d,\eps)} \exp \HH_t(\sigma),
}
    so that
     \begin{subequations} \label{rh3wv1}
    \begin{equation}\label{eq:varphi:endpoint}
    \phi(1)=\eff_{N,\xi}(d,\eps) \quad \text{and} \quad
    \phi(0)=\frac{1}{N}\E \log \sum_{\sigma\in \Sigma^N(d)}\NN(\sigma)\exp\wt H_{N,\xi}(\sigma),
    \end{equation}
    where $\NN(\sigma)$ is defined in~\eqref{eq:def:NN}. 
    Since every $\sigma\in\Sigma^N(d)$ satisfies $P(\sigma)=\sigma$, we have $\NN(\sigma)\geq1$.
    Also $\NN(\sigma)\leq\exp(C_\kappa\sqrt{\eps}N)$ by \eqref{NN_bound}, and so
    \begin{equation}\label{eq:varphi:endpoint:approx}
    \eff_{N,\xi}(d) \leq \phi(0) \leq \eff_{N,\xi}(d) + C_\kappa\sqrt{\eps}.
    \end{equation}
    \end{subequations}
    Meanwhile, differentiation of \eqref{wrg23} results in
    \eq{
\phi'(t) = \frac{1}{N}\Big\langle \frac{\partial \HH_t(\sigma)}{\partial t}\Big\rangle_{t},
}
where $\langle\cdot\rangle_t$ denotes expectation with respect to the Gibbs measure 
$G_t(\sigma)\propto \exp \HH_t(\sigma)$ on $\Sigma^N(d,\eps)$.
    Applying Gaussian integration by parts \cite[Lem.~1.1]{panchenko13a}, we rewrite this as
    \begin{equation*}
        \phi^\prime(t)
        =\E\big\langle \CC(\sigma^1,\sigma^1)-\CC(\sigma^1,\sigma^2)\big\rangle_t,
    \end{equation*}
    where $\sigma^1,\sigma^2$ denote independent samples from $G_t$, and $\CC\colon\Sigma^N(d,\eps)\times\Sigma^N(d,\eps)\to\R$ is defined by
    \begin{equation*}
        \CC(\sigma,\sigma^\prime)
        = \frac{1}{N}\Big[\frac{\partial \HH_t(\sigma)}{\partial t}\HH_t(\sigma^\prime)\Big]
        \stackref{general_cov}{=}\frac{1}{2}\Big(\xi\big(R(\sigma,\sigma^\prime)\big)-\xi\big(R(P\sigma,P\sigma^\prime)\big)\Big).
    \end{equation*}
    It follows that
    \begin{equation*}
    \begin{split}
        |\phi'(t)|
        &\leq \sup_{\|R\|_1\leq 1}\|\nabla\xi(R)\|_{\infty}\cdot \sup_{\sigma,\sigma'\in \Sigma^N(d,\eps)}\|R(\sigma,\sigma^\prime)-R(P\sigma,P\sigma^\prime)\|_1.
    \end{split}
    \end{equation*}
    For any $\sigma,\sigma^\prime\in \Sigma^N(d,\eps)$, we have
    \begin{equation*}
    \begin{split}
       & \|R(\sigma,\sigma^\prime)-R(P\sigma,P\sigma^\prime)\|_1 \\
       &\stackrefpp{def:overlap_entry}{eq:projection:guarantee}{=} \frac{1}{N}\sum_{k,k'=1}^\kappa\bigg|\sum_{i=1}^N\Big(\one\{\sigma_i=\vv e_k\}\one\{\sigma'_i=\vv e_{k'}\}
       - \one\{(P\sigma)_i=\vv e_k\}\one\{(P\sigma')_i=\vv e_{k'}\}\Big)\bigg|\\
        &\stackrefp{eq:projection:guarantee}{\leq} \frac{1}{N}\sum_{i=1}^N\sum_{k,k^\prime=1}^{\ka}\Big|\one\{\sigma_i=\vv e_k\}\one\{\sigma'_i=\vv e_{k'}\}
       - \one\{(P\sigma)_i=\vv e_k\}\one\{(P\sigma')_i=\vv e_{k'}\}\Big| \\
       &\stackrefp{eq:projection:guarantee}{\leq} \frac{1}{N}\sum_{i=1}^N\sum_{k,k^\prime=1}^{\ka}\Big(\big|\one\{\sigma_i=\vv e_k\}-\one\{(P\sigma)_i=\vv e_k\}\big|+\big|\one\{\sigma'_i=\vv e_{k'}\}-\one\{(P\sigma')_i=\vv e_{k'}\}\big|\Big) \\
       &\stackref{eq:projection:guarantee}{\leq} 4\eps.
    \end{split}
    \end{equation*}
    The two previous displays together show that $\sup_{t\in [0,1]}|\phi^\prime(t)|\leq C'\eps$ for some constant $C'$ depending only on $\xi$ and $\kappa$. 
    Combining this fact with \eqref{rh3wv1}, we determine that
    \begin{equation*}
        \big|\eff_{N,\xi}(d,\eps)-\eff_{N,\xi}(d)\big|
        \leq C_\kappa\sqrt{\eps} + C'\eps.
    \end{equation*}
    We have thus proved \eqref{cbw24} with $C=C_\kappa+C'$.
\end{proof}

We now complete the main objective of this appendix.

\begin{proof}[Proof of Theorem~\ref{gen_con_par_thm}]
For any $d\in\DD$ and $N\geq1$, it follows from the definition of $\DD_N$ in \eqref{DN_def} that there exists $d^N\in\DD_N$ satisfying \eqref{fw424} with $L=1$.
      So for any fixed $\eps>0$, we have $\Sigma^N(d,\eps)\supseteq\Sigma^N(d^N)$  once $N$ is large enough that $\eps\geq 1/N$.
      Hence
    \begin{equation*}
        \liminf_{N\to\infty}\eff_{N,\xi}(d,\eps)\geq \liminf_{N\to \infty}\eff_{N,\xi}(d^N)\stackref{eq_prop_general_lower_bound}{\ge} \inf_{\pi\in \Pi_d}\PP_{\xi}(\pi).
    \end{equation*}
   Combining this with the upper bound from Proposition~\ref{prop:upper:interpolation}, we deduce \eqref{gen_con_par_eq_a}. 
    For \eqref{gen_con_par_eq_b}, we consider any sequence $d^N\in \DD_N$ such that $d^N\to d$ as $N\to\infty$, not necessarily satisfying \eqref{fw424}.
    Once $N$ is large enough that $\|d^N-d\|_\infty\leq\eps/2$, we have
    \begin{equation*}
        \Sigma^N(d, \eps/2) \subseteq \Sigma^N(d^N,\eps)\subseteq \Sigma^N(d,2\eps).
    \end{equation*}
    Hence $\eff_{N,\xi}(d,\eps/2)\leq \eff_{N,\xi}(d^N,\eps)\leq \eff_{N,\xi}(d,2\eps)$ for all large $N$, and so \eqref{gen_con_par_eq_a} forces
    \begin{equation}\label{eq:proof:prop:lower:2}
        \lim_{\eps\searrow 0}\limsup_{N\to\infty}\eff_{N,\xi}(d^N,\eps)=\lim_{\eps\searrow 0}\liminf_{N\to\infty}\eff_{N,\xi}(d^N,\eps)=\inf_{\pi\in \Pi_d}\PP_{\xi}(\pi).
    \end{equation}
    Meanwhile, by Lemma~\ref{lem:free:energy:eps:approx} we have
    \begin{equation*}
        \limsup_{N\to\infty}\big|\eff_{N,\xi}(d^N,\eps)-\eff_{N,\xi}(d^N)\big|\leq C\sqrt{\eps}.
    \end{equation*}
    By sending $\eps\searrow 0$ and appealing to \eqref{eq:proof:prop:lower:2}, 
    we obtain \eqref{gen_con_par_eq_b}.
\end{proof}

\bibliographystyle{acm}
\bibliography{erikbib}{}

\end{document}